\documentclass[reqno]{amsart}
\usepackage{amsthm, amssymb, amsfonts}
\usepackage{lmodern}
\usepackage{color}
\usepackage{hyperref}
\usepackage[T5]{fontenc}
\usepackage{amscd,amssymb}
\usepackage[v2,cmtip]{xy}
\usepackage{mathrsfs}
\theoremstyle{plain}
\newtheorem{thm}{Theorem}[section]
\newtheorem{prop}[thm]{Proposition}
\newtheorem{lem}[thm]{Lemma}

\newtheorem{corl}[thm]{Corollary}
\theoremstyle{definition}
\newtheorem{defn}[thm]{Definition}

\newtheorem{nota}[thm]{Notation}

\theoremstyle{plain}
\newtheorem{thms}{Theorem}[subsection]
\newtheorem{props}[thms]{Proposition}
\newtheorem{lems}[thms]{Lemma}

\theoremstyle{definition}

\allowdisplaybreaks

\begin{document} 
\title[The admissible monomial bases for the polynomial algebra]
{The admissible monomial bases for the polynomial algebra of five variables in some types of generic degrees}

\author{Nguy\~\ecircumflex n Sum}
\address{Department of Mathematics and Applications, S\`ai G\`on University, 273 An D\uhorn \ohorn ng V\uhorn \ohorn ng, District 5, H\`\ocircumflex\ Ch\'i Minh city, Viet Nam}
 
\email{nguyensum@sgu.edu.vn}


\bigskip
\begin{abstract}
Let $P_k$ be the graded polynomial algebra $\mathbb F_2[x_1,x_2,\ldots ,x_k]$ over the prime field of two elements, $\mathbb F_2$, with the degree of each $x_i$ being 1. We study the {\it hit problem}, set up by Frank Peterson, of finding a minimal set of generators for $P_k$ as a module over the mod-$2$ Steenrod algebra, $\mathcal{A}.$ In this paper, we explicitly determine a minimal set of $\mathcal{A}$-generators for $P_5$ in the case of the generic degrees $n = 2^{d+1} - 1$ and $n = 2^{d+1} - 2$ for all $d \geqslant 6$.
\end{abstract}
\maketitle

\section{Introduction and statement of main results}\label{s1} 

Let $E^k$ be an elementary abelian 2-group of rank $k$ and let $BE^k$ be the classifying space of $E^k$.  Then, 
$$P_k:= H^*(BE^k) \cong \mathbb F_2[x_1,x_2,\ldots ,x_k],$$ a polynomial algebra in  $k$ generators $x_1, x_2, \ldots , x_k$, each of degree 1. Here the cohomology is taken with coefficients in the prime field $\mathbb F_2$ of two elements. The algebra $P_k$ is a module over the mod-2 Steenrod algebra, $\mathcal A$.  The action of $\mathcal A$ on $P_k$ is determined by the elementary properties of the Steenrod squares $Sq^i$ and subject to the Cartan formula (see Steenrod and Epstein~\cite{st}).

An element $g$ in $P_k$ is called {\it hit} if it belongs to  $\mathcal{A}^+P_k$, where $\mathcal{A}^+$ is the augmentation ideal of $\mathcal A$. That means $g$ can be written as a finite sum $g = \sum_{u\geqslant 0}Sq^{2^u}(g_u)$ for suitable polynomials $g_u \in P_k$.  

We study the {\it Peterson hit problem} of determining a minimal set of generators for the polynomial algebra $P_k$ as a module over the Steenrod algebra. In other words, we want to determine a basis of the $\mathbb F_2$-vector space 
$$QP_k := P_k/\mathcal A^+P_k = \mathbb F_2 \otimes_{\mathcal A} P_k.$$ 

This problem was first studied by Peterson~\cite{pe}, Wood~\cite{wo}, Singer~\cite {si1}, and Priddy~\cite{pr}, who showed its relation to several classical problems in Algebraic Topology. Then, this problem  was investigated by Carlisle-Wood~\cite{cw}, Crabb-Hubbuck~\cite{ch}, Janfada-Wood~\cite{jw1}, Kameko~\cite{ka}, Mothebe \cite{mo}, Nam~\cite{na}, Repka-Selick~\cite{res}, Silverman~\cite{sl,sl2}, Silverman-Singer~\cite{ss}, Singer~\cite{si2}, Sum-T\'in \cite{su5}, Walker and Wood~\cite{wa3}, Wood~\cite{wo2}, the present author \cite{su1,su2} and others. 

The vector space $QP_k$ was explicitly computed by 
Peterson~\cite{pe} for $k=1, 2,$ by Kameko~\cite{ka} for $k=3$ and  by us \cite{su2} for $k = 4$. For $k > 4$, this problem is still open. 
Recently, the hit problem and its applications to representations of general linear groups have been presented in the monographs of Walker and Wood \cite{wa1,wa2}.

For any nonnegative integer $n$, set $\mu(n) = \min\{m \in \mathbb Z : \alpha (n+m)\leqslant m\}$ where $\alpha (a)$ denotes the number of one in dyadic expansion of a positive integer $a$. 
Denote by $(P_k)_n$ and  $(QP_k)_n$ the subspaces of degree $n$ homogeneous polynomials in the spaces $P_k$ and $QP_k$ respectively. The following is Peterson's conjecture, which was established by Wood.

\begin{thm}[See Wood~\cite{wo}]\label{dlmd1} 
If $\mu(n) > k$, then $(QP_k)_n = 0$.
\end{thm} 

One of the important tools in the study of the hit problem is Kameko's homomorphism 
$\widetilde{Sq}^0_*: QP_k \to QP_k$. 
This homomorphism is induced by the $\mathbb F_2$-linear map $\phi :P_k\to P_k$ given by
$$\phi (x) = 
\begin{cases}y, &\text{if }x=x_1x_2\ldots x_ky^2,\\  
0, & \text{otherwise,} \end{cases}$$
for any monomial $x \in P_k$. Note that $\phi $ is not an $\mathcal A$-homomorphism. However, 
$\phi Sq^{2t} = Sq^{t}\phi,$ and $\phi Sq^{2t+1} = 0$ for any non-negative integer $t$. 

\begin{thm}[See Kameko~\cite{ka}]\label{dlmd2} 
Let $m$ be a positive integer. If $\mu(2m+k)=k$, then 
$$(\widetilde{Sq}^0_*)_{(k,m)}:= \widetilde{Sq}^0_*: (QP_k)_{2m+k}\longrightarrow (QP_k)_m$$
is an isomorphism of the $\mathbb F_2$-vector spaces. 
\end{thm}

From Theorems \ref{dlmd1} and \ref{dlmd2}, the hit problem is reduced to the case of degree $n$ such that $\mu(n) < k$. For $\mu(n)=k-1$, the problem was partially studied by Crabb-Hubbuck~\cite{ch}, Nam~\cite{na}, Repka-Selick~\cite{res}, Walker-Wood \cite{wa3} and the present author ~\cite{su1,su2}. Recently, many authors study this problem for the case $k = 5$.

To classify the situations of degrees, one can use the function $\beta(n) = \alpha(n+\mu(n))$. For $k = 5$ and $\beta(n) = 1$, one gets $n = 2^{d+1} - s$ with $s = \mu(n) < 5$, and the problem has been respectively determined by Ph\'uc~\cite{ph} and Ph\'uc-Sum~\cite{sp2} for $s = 3,\, 4$. 

To complete the solution of hit problem for $k = 5$ and the degrees $n$ with $\beta(n) = 1$, in the present paper, we explicitly determine all the admissible monomials (see Section ~\ref{s2}) of $P_5$ in the generic degrees $n = 2^{d+1}-s$ with $s = 1,\, 2$ and $d \geqslant 6$. 

\begin{thm}\label{thm1}
For any integer $d \geqslant 6$, there exist exactly $1085$	
admissible monomials of degree $2^{d+1} -2$ in $P_5$. Consequently $\dim (QP_5)_{(2^{d+1} -2)} = 1085.$
\end{thm}
For the cases $d = 1,\, 2$, the computation is easy and we have $\dim (QP_5)_{2} = 10$, $\dim (QP_5)_{6} = 74.$ The cases $d = 3,\, 4$ are respectively determined by Ly-T\'in~\cite{Tin20} and Sum-Mong~\cite{smo}, they showed that $\dim (QP_5)_{14} = 320$ and $\dim (QP_5)_{30} = 840.$ For the case $d = 5$, the computation is extremely complicated. However, we also show in Subsection \ref{s54} that $1105 \leqslant \dim (QP_5)_{62} \leqslant 1175.$ 

For $n = 2^{d+1} -1$, the space $(QP_5)_n$ has been determined with $d \leqslant 4$ by T\'in~\cite{tin14}, Ph\'uc~\cite{ph21} and the present author \cite{su14}. For $d \geqslant 5$, we consider Kameko's homomorphism
$$(\widetilde{Sq}^0_*)_{(5,2^{d}-3)}: (QP_5)_{(2^{d+1}-1)}\longrightarrow (QP_5)_{(2^{d}-3)}.$$
This is an epimorphism and the space $(QP_k)_{(2^{d}-3)}$ has been computed in Ph\'uc~\cite{ph}. So, we need only to determine $\mbox{\rm Ker}\big((\widetilde{Sq}^0_*)_{(5,2^{d}-3)}\big)$.

\begin{thm}\label{thm2}
For any integer $d \geqslant 6$, there exist exactly $496$ classes of degree $2^{d+1} -1$ in $\mbox{\rm Ker}\big((\widetilde{Sq}^0_*)_{(5,2^{d}-3)}\big)$ represented by the admissible monomials in $P_5$. Consequently 
$\dim \mbox{\rm Ker}\big((\widetilde{Sq}^0_*)_{(5,2^{d}-3)}\big) = 496.$
\end{thm}

Combining Theorem \ref{thm2} and the results in Ph\'uc \cite{ph}, we obtain the following. 

\begin{corl}\label{hqdl2} For any integer $d \geqslant 6$, we have
$$\dim(QP_5)_{(2^{d+1}-1)} = \begin{cases}1441, \mbox{ if } d = 6,\\ 1611, \mbox{ if } d = 7,\\ 1612, \mbox{ if } d \geqslant 8. \end{cases} $$
\end{corl}
The case $d = 1$ is easy and $\dim (QP_5)_{3} = 25$. The cases $d = 2,\, 3,\, 4$ were respectively determined in T\'in~\cite{tin14}, Sum \cite{su14}, Ph\'uc \cite{ph21}, they showed that $\dim (QP_5)_{7} = 110$, $\dim (QP_5)_{15} = 432$ and $\dim (QP_5)_{31} = 866.$ For $d = 5$, the computation of $(QP_5)_{63}$ is complicated. However, we also show in Subsection \ref{s55} that $1141\leqslant \dim (QP_5)_{63} \leqslant 1155.$

The proofs of Theorems \ref{thm1} and \ref{thm2} are very complicated. We prove by combining hand computation with the aids of some simple computer programmes.

\medskip
The paper is organized as follows. In Section \ref{s2}, we recall some needed information on the admissible monomials in $P_k$, the criterions of Singer and Silverman on the hit monomials. The detailed proofs of Theorems \ref{thm1} and \ref{thm2} are respectively presented in Sections ~ \ref{s3} and \ref{s4}. Finally, in Section \ref{sect5} we list the needed admissible monomials of degrees $2^{d+1}-1$ and $2^{d+1}-2$ in $P_5$.
 
\section{Preliminaries}\label{s2}
\setcounter{equation}{0}

In this section, we recall some needed information from Kameko~\cite{ka}, Singer \cite{si2} and the first named author \cite{su2} which will be used in the next section.

\medskip
\begin{nota} We denote $\mathbb N_k = \{1,2, \ldots , k\}$ and
\begin{align*}
X_{\mathbb J} = X_{\{j_1,j_2,\ldots , j_s\}} = \prod_{j\in \mathbb N_k\setminus \mathbb J}x_j , \ \ \mathbb J = \{j_1,j_2,\ldots , j_s\}\subset \mathbb N_k,
\end{align*}
In particular, $X_{\mathbb N_k} =1,\
X_\emptyset = x_1x_2\ldots x_k,$ 
$X_j = x_1\ldots \hat x_j \ldots x_k, \ 1 \leqslant j \leqslant k,$ and $X:=X_k \in P_{k-1}.$
	
Let $\alpha_i(a)$ denote the $i$-th coefficient  in dyadic expansion of a non-negative integer $a$. That means
$a= \alpha_0(a)2^0+\alpha_1(a)2^1+\alpha_2(a)2^2+ \ldots ,$ for $ \alpha_i(a) =0$ or 1 with $i\geqslant 0$. 
	
Let $x=x_1^{a_1}x_2^{a_2}\ldots x_k^{a_k} \in P_k$. Denote $\nu_j(x) = a_j, 1 \leqslant j \leqslant k$ and $\nu(x) = \max\{\nu_j(x): 1 \leqslant j \leqslant k\}$.  
Set 
$$\mathbb J_t(x) = \{j \in \mathbb N_k :\alpha_t(\nu_j(x)) =0\},$$
for $t\geqslant 0$. Then, we have
$x = \prod_{t\geqslant 0}X_{\mathbb J_t(x)}^{2^t}.$ 
\end{nota}

\begin{defn} A weight vector $\omega$ is a sequence of non-negative integers $(\omega_1,\omega_2,\ldots $, $\omega_i, \ldots)$ such that $\omega_i = 0$ for $i \gg 0$.
	
For a monomial  $x$ in $P_k$,  define two sequences associated with $x$ by
\begin{align*} 
\omega(x)&=(\omega_1(x),\omega_2(x),\ldots , \omega_i(x), \ldots),\\
\sigma(x) &= (\nu_1(x),\nu_2(x),\ldots ,\nu_k(x)),
\end{align*}
where
$\omega_i(x) = \sum_{1\leqslant j \leqslant k} \alpha_{i-1}(\nu_j(x))= \deg X_{\mathbb J_{i-1}(x)},\ i \geqslant 1.$
The sequences $\omega(x)$ and $\sigma(x)$ are respectively called the weight vector and the exponent vector of $x$. 
\end{defn}

The sets of all the weight vectors and the exponent vectors are given the left lexicographical order. 

For any weight vector $\omega = (\omega_1,\omega_2,\ldots)$, we define $\deg \omega = \sum_{i > 0}2^{i-1}\omega_i$ and the length $\ell(\omega) = \max\{i : \omega_i >0\}$. We write $\omega = (\omega_1,\omega_2,\ldots, \omega_r)$ if $\ell(\omega) = r$.  For a weight vector $\eta = (\eta_1,\eta_2, \ldots)$, we define the concatenation of weight vectors 
\[\omega|\eta = (\omega_1,\ldots,\omega_r,\eta_1,\eta_2,\ldots)\] 
if $\ell(\omega) = r$ and $(a)|^b = (a)|(a)|\ldots|(a)$, ($b$ times of $(a)$'s), where $a,\, b$ are positive integers.  
Denote by   $P_k(\omega)$ the subspace of $P_k$ spanned by monomials $y$ such that
$\deg y = \deg \omega$ and $\omega(y) \leqslant \omega$, and by $P_k^-(\omega)$ the subspace of $P_k(\omega)$ spanned by monomials $y$  such that $\omega(y) < \omega$. 

\begin{defn}\label{dfn2} Let $\omega$ be a weight vector and $f, g$ two polynomials  of the same degree in $P_k$. 
	
i) $f \equiv g$ if and only if $f - g \in \mathcal A^+P_k$. If $f \equiv 0$ then $f$ is called {\it hit}.
	
ii) $f \equiv_{\omega} g$ if and only if $f - g \in \mathcal A^+P_k+P_k^-(\omega)$. 
\end{defn}

Obviously, the relations $\equiv$ and $\equiv_{\omega}$ are equivalence ones. Denote by $QP_k(\omega)$ the quotient of $P_k(\omega)$ by the equivalence relation $\equiv_\omega$. Then, we have 
$$QP_k(\omega)= P_k(\omega)/ ((\mathcal A^+P_k\cap P_k(\omega))+P_k^-(\omega)).$$  

For a  polynomial $f \in  P_k$, we denote by $[f]$ the class in $QP_k$ represented by $f$. If  $\omega$ is a weight vector and $f \in  P_k(\omega)$, then denote by $[f]_\omega$ the class in $QP_k(\omega)$ represented by $f$. Denote by $|S|$ the cardinal of a set $S$.

From \cite{su3}, we have
\begin{equation}\label{ct2.1}
(QP_k)_n \cong \bigoplus_{\deg \omega = n}QP_k(\omega).
\end{equation}

\begin{defn}\label{defn3} 
Let $x, y$ be monomials of the same degree in $P_k$. We say that $x <y$ if and only if one of the following holds:  
	
i) $\omega (x) < \omega(y)$;
	
ii) $\omega (x) = \omega(y)$ and $\sigma(x) < \sigma(y).$
\end{defn}

\begin{defn}
A monomial $x$ in $P_k$ is said to be inadmissible if there exist monomials $y_1,y_2,\ldots, y_t$ such that $y_j<x$ for $j=1,2,\ldots , t$ and $x - \sum_{j=1}^ty_j \in \mathcal A^+P_k.$ 
A monomial $x$ is said to be admissible if it is not inadmissible.
\end{defn} 

Obviously, the set of all the admissible monomials of degree $n$ in $P_k$ is a minimal set of $\mathcal{A}$-generators for $P_k$ in degree $n$. 
\begin{defn} 
A monomial $x$ in $P_k$ is said to be strictly inadmissible if and only if there exist monomials $y_1,y_2,\ldots, y_t$ such that $y_j<x,$ for $j=1,2,\ldots , t$ and $x = \sum_{j=1}^t y_j + \sum_{u=1}^{2^s-1}Sq^u(h_u)$ with $s = \max\{i : \omega_i(x) > 0\}$ and suitable polynomials $h_u \in P_k$.
\end{defn}

It is easy to see that if $x$ is strictly inadmissible, then it is inadmissible.

\begin{thm}[See Kameko \cite{ka}, Sum \cite{su1}]\label{dlcb1}  
Let $x, y, w$ be monomials in $P_k$ such that $\omega_i(x) = 0$ for $i > r>0$, $\omega_s(w) \ne 0$ and   $\omega_i(w) = 0$ for $i > s>0$.
	
{\rm i)}  If  $w$ is inadmissible, then  $xw^{2^r}$ is also inadmissible.
	
{\rm ii)}  If $w$ is strictly inadmissible, then $wy^{2^{s}}$ is also strictly inadmissible.
\end{thm} 

\begin{prop}[See \cite{su2}]\label{mdcb3} Let $x$ be an admissible monomial in $P_k$. Then we have

\smallskip
{\rm i)} If there is an index $i_0$ such that $\omega_{i_0}(x)=0$, then $\omega_{i}(x)=0$ for all $i > i_0$.
	
{\rm ii)} If there is an index $i_0$ such that $\omega_{i_0}(x)<k$, then $\omega_{i}(x)<k$ for all $i > i_0$.
\end{prop}

Now, we recall a result of Singer \cite{si2} on the hit monomials in $P_k$. 

\begin{defn}\label{spi}  A monomial $z$ in $P_k$   is called a spike if $\nu_j(z)=2^{d_j}-1$ for $d_j$ a non-negative integer and $j=1,2, \ldots , k$. If $z$ is a spike with $d_1>d_2>\ldots >d_{r-1}\geqslant d_r>0$ and $d_j=0$ for $j>r,$ then it is called the minimal spike.
\end{defn}

In \cite{si2}, Singer showed that if $\mu(n) \leqslant k$, then there exists uniquely a minimal spike of degree $n$ in $P_k$. 
The following is a criterion for the hit monomials in $P_k$.

\begin{thm}[See Singer~\cite{si2}]\label{dlsig} Suppose $x \in P_k$ is a monomial of degree $n$, where $\mu(n) \leqslant k$. Let $z$ be the minimal spike of degree $n$. If $\omega(x) < \omega(z)$, then $x$ is hit.
\end{thm}

We need Silverman's criterion for the hit polynomials in $P_k$.
\begin{thm}[{See Silverman~\cite[Theorem 1.2]{sl2}}]\label{dlsil} Let $p$ be a polynomial of the form $fg^{2^m}$ for some homogeneous  polynomials $f$ and $g$. If $\deg f < (2^m - 1)\mu(\deg g)$, then $p$ is hit.
\end{thm}

This result leads to a criterion in terms of minimal spike which is a stronger version of Theorem \ref{dlsig}.

\begin{thm}[{See Walker-Wood~\cite[Theorem 14.1.3]{wa1}}]\label{dlww} Let $x \in P_k$ be a monomial of degree $n$, where $\mu(n)\leqslant k$ and let $z$ be the minimal spike of degree $n$. If there is an index $h$ such that $\sum_{i=1}^{h}2^{i-1}\omega_i(x) < \sum_{i=1}^{h}2^{i-1}\omega_i(z),$ then $x$ is hit.
\end{thm}

We set 
\begin{align*} 
P_k^0 &=\langle\{x=x_1^{a_1}x_2^{a_2}\ldots x_k^{a_k} \ : \ a_1a_2\ldots a_k=0\}\rangle,\\ 
P_k^+ &= \langle\{x=x_1^{a_1}x_2^{a_2}\ldots x_k^{a_k} \ : \ a_1a_2\ldots a_k>0\}\rangle. 
\end{align*}

It is easy to see that $P_k^0$ and $P_k^+$ are the $\mathcal{A}$-submodules of $P_k$. Furthermore, we have the following.

\begin{prop}\label{2.7} We have a direct summand decomposition of the $\mathbb F_2$-vector spaces
$QP_k =QP_k^0 \oplus  QP_k^+.$
Here $QP_k^0 = P_k^0/\mathcal A^+P_k^0$ and  $QP_k^+ = P_k^+/\mathcal A^+P_k^+$.
\end{prop}

For $1 \leqslant i \leqslant k$, define the homomorphism $f_i: P_{k-1} \to P_k$ of algebras by substituting
\begin{equation}\label{ct22}
f_i(x_j) = \begin{cases} x_j, &\text{ if } 1 \leqslant j <i,\\
x_{j+1}, &\text{ if } i \leqslant j <k.
\end{cases}
\end{equation}

\begin{prop}[See Mothebe and Uys \cite{mo}]\label{mdmo} Let $i, d$ be positive integers such that $1 \leqslant i \leqslant k$. If $x$ is an admissible monomial in $P_{k-1}$ then $x_i^{2^d-1}f_i(x)$ is also an admissible monomial in $P_{k}$.
\end{prop}

Denote $\mathcal N_k =\{(i;I) : I=(i_1,i_2,\ldots,i_r),1 \leqslant  i < i_1 <  \ldots < i_r\leqslant  k,\ 0\leqslant r <k\}$. For any $(i;I) \in \mathcal N_k$, we define the homomorphism $p_{(i;I)}: P_k \to P_{k-1}$ of algebras by substituting
\begin{equation}\label{ct23}
p_{(i;I)}(x_j) =\begin{cases} x_j, &\text{ if } 1 \leqslant j < i,\\
\sum_{s\in I}x_{s-1}, &\text{ if }  j = i,\\  
x_{j-1},&\text{ if } i< j \leqslant k.
\end{cases}
\end{equation}
Then $p_{(i;I)}$ is a homomorphism of $\mathcal A$-modules. 
\begin{lem}[Ph\'uc-Sum \cite{sp}]\label{bdm} If $x$ is a monomial in $P_k$, then $$p_{(i;I)}(x) \in P_{k-1}(\omega(x)).$$ 
\end{lem}

This lemma implies that if $\omega$ is a weight vector and $x \in P_k(\omega)$, then $p_{(i;I)}(x) \in P_{k-1}(\omega)$. Moreover, $p_{(i;I)}$ passes to a homomorphism from $QP_k(\omega)$ to $QP_{k-1}(\omega)$. So, these homomorphisms can be used to prove certain subset of $QP_k(\omega)$ is linearly independent.

For $J= (j_1, j_2, \ldots, j_r) : 1 \leqslant j_1 <\ldots < j_r \leqslant k$, we define a monomorphism $\theta_J: P_r \to P_k$ of $\mathcal A$-algebras by substituting 
\begin{equation}\label{ctbs}
\theta_J(x_t) = x_{j_t} \ \mbox{ for } \ 1 \leqslant t \leqslant r.
\end{equation} 
It is easy to see that, for any weight vector $\omega$ of degree $n$, 
\[Q\theta_J(P_r^+)(\omega) \cong  QP_r^+(\omega)\mbox{ and } (Q\theta_J(P_r^+))_n \cong (QP_r^+)_n\] 
for $1 \leqslant r \leqslant k$, where $Q\theta_J(P_r^+) = \theta_J(P_r^+)/\mathcal A^+\theta_J(P_r^+)$. So, by a simple computation using Theorem~ \ref{dlmd1} and (\ref{ct2.1}), we get the following.
\begin{prop}[{See Walker and Wood~\cite{wa1}}]\label{mdbs} For a weight vector $\omega$ of degree $n$, we have direct summand decompositions of the $\mathbb F_2$-vector spaces
\begin{align*} QP_k(\omega)  &= \bigoplus_{\mu(n) \leqslant r\leqslant k}\bigoplus_{\ell(J) =r}Q\theta_J(P_r^+)(\omega), 
\end{align*}
where $\ell(J)$ is the length of $J$. Consequently 
\begin{align*}
\dim QP_k(\omega) &= \sum_{\mu(n) \leqslant r\leqslant k}{k\choose r}\dim QP_r^+(\omega),\\
\dim (QP_k)_n &= \sum_{\mu(n) \leqslant r\leqslant k}{k\choose r}\dim (QP_r^+)_n.
\end{align*}
\end{prop}

\begin{nota} From now on, we denote by $B_{k}(n)$ the set of all admissible monomials of degree $n$  in $P_k$, 
$$B_{k}^0(n) = B_{k}(n)\cap P_k^0,\ B_{k}^+(n) = B_{k}(n)\cap P_k^+.$$ 
For a weight vector $\omega$ of degree $n$, we set $$B_k(\omega) = B_{k}(n)\cap P_k(\omega),\ B_k^+(\omega) = B_{k}^+(n)\cap P_k(\omega).$$
For a subset $S \subset P_k$, we denote $[S] = \{[f] : f \in S\}$. If $S \subset P_k(\omega)$, then we set $[S]_\omega = \{[f]_\omega : f \in S\}$. 
Then, $[B_k(\omega)]_\omega$ and $[B_k^+(\omega)]_\omega$, are respectively the basses of the $\mathbb F_2$-vector spaces $QP_k(\omega)$ and $QP_k^+(\omega) := QP_k(\omega)\cap QP_k^+$.
\end{nota}

\section{Proof of Theorem \ref{thm1}}\label{s3}
\setcounter{equation}{0}

Firstly, we determine the weight vectors of the admissible monomials of degree $2^{d+1}-2$ for $d \geqslant 5$.
\begin{lem}\label{bdd5} If $x$ be an admissible monomial of degree $2^{d+1}-2$ in $P_5$ for $d \geqslant 5$, then either $\omega(x) = (2)|^d$ or $\omega(x) = (4)|(3)|^{d-2}|(1)$.
\end{lem}

We need the following lemma.
\begin{lem}\label{bda51} Let $(i,j,t,u,v)$ be an arbitrary permutation of $(1,2,3,4,5)$. The following monomials are strictly inadmissible:

\smallskip
{\rm i)} $x_i^2x_jx_t^3,\, i<j$;\ $x_i^2x_jx_tx_u^2,\, i<j<t$; $x_ix_j^2x_t^2x_u, i<j<t<u$.

{\rm ii)} $x_i^2x_j^2x_t^3x_u^3$;\ $x_i^2x_jx_t^2x_u^2x_v^3,\, $i < j$;\ x_i^2x_jx_tx_u^3x_v^3,\, i< j<t$.

{\rm iii)} $x_ix_j^{2}x_t^{6}x_u^{6}x_v^{7}$; $x_i^{6}x_j^{3}x_tx_u^{6}x_v^{6},\, i < j$.

{\rm iv)} $x_i^3x_j^4x_t^4x_u^4x_v^7$; $x_i^3x_j^4x_t^4x_u^5x_v^6$.

{\rm v)} $x_ix_j^{3}x_t^{14}x_u^{14}x_v^{14}$,  $x_i^{2}x_j^{3}x_t^{13}x_u^{14}x_v^{14}$, $x_i^{3}x_j^{5}x_t^{10}x_u^{14}x_v^{14}$, $x_i^{3}x_j^{6}x_t^{10}x_u^{13}x_v^{14}$.
\end{lem}
\begin{proof} Parts i) and ii) are proved in Ly-T\'in \cite{Tin20}. Parts ii) and iv) are prove in \cite{smo}. We prove Part v) for $x = x_i^{3}x_j^{5}x_t^{10}x_u^{14}x_v^{14}$. The others can be proved by the similar computations. By a direct computation, we have
\begin{align*}
x &= Sq^1\big(x_i^{3}x_j^{5}x_t^{10}x_u^{13}x_v^{14} + x_i^{3}x_j^{5}x_t^{12}x_u^{11}x_v^{14} + x_i^{5}x_j^{10}x_t^{10}x_u^{7}x_v^{13} + x_i^{5}x_j^{10}x_t^{12}x_u^{7}x_v^{11}\\ 
&\quad + x_i^{5}x_j^{12}x_t^{10}x_u^{7}x_v^{11} + x_i^{9}x_j^{6}x_t^{10}x_u^{7}x_v^{13} + x_i^{10}x_j^{5}x_t^{12}x_u^{7}x_v^{11} + x_i^{12}x_j^{5}x_t^{10}x_u^{7}x_v^{11}\big)\\
&\quad+  Sq^2\big(x_i^{3}x_j^{6}x_t^{10}x_u^{11}x_v^{14} + x_i^{6}x_j^{10}x_t^{10}x_u^{7}x_v^{11} + x_i^{10}x_j^{6}x_t^{10}x_u^{7}x_v^{11}\big)\\
&\hskip4.8cm +  Sq^4\big(x_i^{5}x_j^{6}x_t^{10}x_u^{7}x_v^{14}\big)\ 
\mbox{mod}(P_5^-(2,4,3,3)).
\end{align*}
The above equality shows that $x= x_i^{3}x_j^{5}x_t^{10}x_u^{14}x_v^{14}$ is strictly inadmissible.
\end{proof}

\begin{proof}[Proof of Lemma \ref{bdd5}] Observe  that $z=x_1^{2^d-1}x_2^{2^d-1}$ is the minimal spike of degree $2^{d+1}-2$ in $P_5$ and $ \omega (z) = (2)|^d$. Since $2^{d+1}-2$ is even, using Theorem \ref{dlsig}, we obtain $\omega_1(x)=2$ or $\omega_1(x)=4$. If $\omega_1(x) = 4$. Then $x = X_iy^2$ with $y$ an admissible monomial of degree $2^{d}-3$ in $P_5$ and $1 \leqslant i \leqslant 5$. Combining Proposition \ref{mdcb3} and a results in Ph\'uc \cite{ph} we get $\omega(y) = (3)|^{d-2}|(1)$, so $\omega(x) = (4)|(3)|^{d-2}|(1)$.
	
Suppose $\omega_1(x) = 2$. We prove $\omega(x) = (2)|^d$ by induction on $d \geqslant 5$. Since $\omega_1(x) = 2$, $x = x_ix_jy^2$ with $1 \leqslant i < j \leqslant 5$ and $y$ an admissible monomial of degree $2^d-1$. For $d=5$, in \cite{smo} we have proved that either $\omega(y)=(2)|^4$ or $\omega(y)=(2,4,3,1)$ or $\omega(y)=(4,3,3,1)$. By a direct computation we see that if $\omega(y)=(2,4,3,1)$, then there is a monomial $u$ as given in one of Parts i) and iii) of Lemma \ref{bda51} such that $x = uz^{2^r}$ with $r=2,3$ and $z$ a suitable monomial. If $\omega(y)=(4,3,3,1)$, then there is a monomial $v$ as given in one of Parts ii), iii) and v) of Lemma \ref{bda51} such that $x = vh^{2^r}$ with $r=2,3,4$ and $h$ a suitable monomial. By Theorem \ref{dlcb1}, $x$ is inadmissible. So, we get $\omega(y)=(2)|^4$ and $\omega(x)=(2)|^5$. Suppose $d > 5$, by the inductive hypothesis, we have $\omega(y)=(2)|^{d-1}$, hence $\omega(x) = (2)|^d$.	
 The lemma is proved.
\end{proof}

By Lemma \ref{bdd5}, we have
$$(QP_5)_{(2^{d+1}-2)} = (QP_5^0)_{(2^{d+1}-2)}\oplus QP_5^+((2)|^d)\oplus QP_5^+((4)|(3)|^{d-2}|(1)).$$
From Kameko \cite{ka} and our work \cite{su2}, we have 
\begin{align*}
&\dim (QP_2^+)_{(2^{d+1}-2)} = \dim QP_2^+((2)|^d) = 1,\\
&\dim (QP_3^+)_{(2^{d+1}-2)} = \dim QP_3^+((2)|^d) = 4,\\
&\dim QP_4^+((2)|^d) = 13,\ \dim QP_4^+((4)|(3)|^{d-2}|(1)) = 45
\end{align*} 
for any $d \geqslant 5$. Hence, we get 
\begin{align*}
&\dim QP_5^0((2)|^d) = {5\choose 2} + 4{5\choose 3} + 13{5\choose 4} = 115,\\
&\dim QP_5^0((4)|(3)|^{d-2}|(1)) = 45{5\choose 4} = 225.
\end{align*}
So, we need to determine the spaces $QP_5^+((2)|^d)$ and $QP_5^+((4)|(3)|^{d-2}|(1))$.

\subsection{Computation of $QP_5^+((2)|^d)$}\

\medskip
In this subsection we prove the following.
\begin{props}\label{mdd51}
For any $d \geqslant 5$, $B_5^+((2)|^d) =\{a_{d,t}:\, 1 \leqslant t \leqslant 40\}$, where the monomials $a_t = a_{d,t}$ are determined as in Subsection $\ref{s51}$. Consequently, 
$$\dim QP_5^+((2)|^d) = 40,\ \dim QP_5((2)|^d) = 155.$$
\end{props}
We need a lemma for the proof of the proposition.

\begin{lems}\label{bda61} Let $(i,j,t,u,v)$ be an arbitrary permutation of $(1,2,3,4,5)$. The following monomials are strictly inadmissible:
	
\smallskip
\ \! {\rm i)} $x_i^3x_j^4x_t^7;\ x_ix_j^6x_t^3x_u^4,\, x_i^3x_j^4x_tx_u^6,\, x_i^3x_j^4x_t^3x_u^4,\, i<j<t<u.$

\smallskip
\ {\rm ii)} $x_ix_j^7x_t^{10}x_u^{12},\, i<j<t<u;$ $x_i^7x_jx_t^{10}x_u^{12},\, i<j<t<u;$ $x_i^3x_j^3x_t^{12}x_u^{12}$,\, $x_i^3x_j^5x_t^8x_u^{14}$,\, $x_i^3x_j^5x_t^{14}x_u^{8}$,\, $x_i^7x_j^7x_t^8x_u^{8}$,\, $i < j < t < u$.
			
\smallskip
{\begin{tabular}{lllll}
{\rm iii)}&$x_1x_2^{6}x_3x_4^{10}x_5^{12}$ & $x_1x_2^{2}x_3^{12}x_4^{3}x_5^{12}$ & $x_1^{3}x_2^{12}x_3x_4^{2}x_5^{12}$ & $x_1x_2^{2}x_3^{4}x_4^{11}x_5^{12}$ \cr
&$x_1x_2^{2}x_3^{7}x_4^{8}x_5^{12}$ & $x_1x_2^{2}x_3^{7}x_4^{12}x_5^{8}$ &$x_1x_2^{7}x_3^{2}x_4^{8}x_5^{12}$ & $x_1x_2^{7}x_3^{2}x_4^{12}x_5^{8}$ \cr
&$x_1^{7}x_2x_3^{2}x_4^{8}x_5^{12}$ & $x_1^{7}x_2x_3^{2}x_4^{12}x_5^{8}$ & $x_1^{3}x_2^{4}x_3x_4^{10}x_5^{12}$ & $x_1x_2^{7}x_3^{10}x_4^{4}x_5^{8}$ \cr
&$x_1^{7}x_2x_3^{10}x_4^{4}x_5^{8}$ & $x_1x_2^{6}x_3^{7}x_4^{8}x_5^{8}$ & 
$x_1x_2^{7}x_3^{6}x_4^{8}x_5^{8}$ & $x_1^{7}x_2x_3^{6}x_4^{8}x_5^{8}$ \cr
&$x_1^{3}x_2^{4}x_3^{9}x_4^{2}x_5^{12}$ & $x_1^{3}x_2^{5}x_3^{8}x_4^{2}x_5^{12}$ & $x_1^{7}x_2^{9}x_3^{2}x_4^{4}x_5^{8}$ & $x_1^{3}x_2^{3}x_3^{4}x_4^{8}x_5^{12}$ \cr 
&$x_1^{3}x_2^{3}x_3^{4}x_4^{12}x_5^{8}$ & $x_1^{3}x_2^{3}x_3^{12}x_4^{4}x_5^{8}$ & $x_1^{3}x_2^{5}x_3^{6}x_4^{8}x_5^{8}$ & $x_1^{3}x_2^{5}x_3^{8}x_4^{6}x_5^{8}$. \cr
\end{tabular}}

\smallskip
\ \! {\rm iv)} $x_1^{3}x_2^{5}x_3^{10}x_4^{16}x_5^{28}$,  $x_1^{3}x_2^{5}x_3^{10}x_4^{28}x_5^{16}$.
\end{lems}

\begin{proof} Parts i) and ii) follow from Kameko \cite{ka} and our work \cite{su2}, Part iii) is proved in \cite{smo}. We prove Part iv) for $x = x_1^{3}x_2^{5}x_3^{10}x_4^{16}x_5^{28}$. By a direct computation we have
\begin{align*}
x &= x_1^{2}x_2^{3}x_3^{9}x_4^{20}x_5^{28} + x_1^{2}x_2^{5}x_3^{9}x_4^{18}x_5^{28} + x_1^{3}x_2^{3}x_3^{8}x_4^{20}x_5^{28} + x_1^{3}x_2^{4}x_3^{9}x_4^{18}x_5^{28} \\
&\quad + x_1^{3}x_2^{5}x_3^{8}x_4^{18}x_5^{28} +  Sq^1\big(x_1^{3}x_2^{3}x_3^{9}x_4^{18}x_5^{28}\big) +  Sq^2\big(x_1^{2}x_2^{3}x_3^{9}x_4^{18}x_5^{28} + x_1^{5}x_2^{3}x_3^{6}x_4^{18}x_5^{28}\big)\\
&\quad +  Sq^4\big(x_1^{3}x_2^{3}x_3^{6}x_4^{18}x_5^{28} + x_1^{3}x_2^{9}x_3^{6}x_4^{12}x_5^{28}\big) + Sq^8\big(x_1^{3}x_2^{5}x_3^{6}x_4^{12}x_5^{28}\big) \ 
\mbox{mod}(P_5^-((2)|^5)).
\end{align*}
Hence, the monomial $x$ is strictly inadmissible. The other monomials are carried out by a similar computation.
\end{proof}

\begin{proof}[Proof of Proposition $\ref{mdd51}$] Let $x$ be a monomial of wight vector $(2)|^d$ in $P_5^+$ with $d \geqslant 5$. By a direct computation we see that if $x \ne a_{d,t}$ for all $t,\, 1 \leqslant t \leqslant 40$, then there is a monomial $w$ as given in either Lemma \ref{bda51}(i) or Lemma \ref{bda61} such that $x = yw^{2^r}z^{2^{r+u}}$ with $r,\, u$ nonnegative integers, $r\geqslant 0,\, 2\leqslant u \leqslant 5$, and $y,\, z$ suitable monomials. By Theorem \ref{dlcb1}, $x$ is inadmissible. Hence, $$B_5^+((2)|^d) \subset \{a_{d,t}: 1\leqslant t\leqslant 40 \}.$$
	
We now prove the set $\{[a_{d,t}]: 1\leqslant t\leqslant 40 \}$ is linearly independent in $QP_5((2)|^d) \subset QP_5$. Suppose there is a linear relation
\begin{equation}\label{ctd61}
\mathcal S:= \sum_{1\leqslant t \leqslant 40}\gamma_ta_{d,t} \equiv 0,
\end{equation}
where $\gamma_t \in \mathbb F_2$. We denote $\gamma_{\mathbb J} = \sum_{t \in \mathbb J}\gamma_t$ for any $\mathbb J \subset \{t\in \mathbb N:1\leqslant t \leqslant 40\}$.

Let $u_{d,u},\, 1\leqslant u \leqslant 13$, be as in Subsection \ref{s51} and the homomorphism $p_{(i;I)}:P_5\to P_4$ defined by \eqref{ct23} for $k=5$. By using Lemma \ref{bdm}, we see that $p_{(i;I)}$ passes to a homomorphism from $QP_5((2)|^d)$ to $QP_4((2)|^d)$. By applying $p_{(1;2)}$ to \eqref{ctd61}, we obtain 
\begin{align}\label{c41}
p_{(1;2)}(\mathcal S) &\equiv \gamma_{\{7,9,11,13,14\}}u_{d,1} + \gamma_{8}u_{d,2} + \gamma_{9}u_{d,3} + \gamma_{10}u_{d,4} + \gamma_{11}u_{d,5}\notag\\ 
&\quad + \gamma_{13}u_{d,6} + \gamma_{14}u_{d,7} + \gamma_{16}u_{d,8} + \gamma_{\{5,9,11,14\}}u_{d,9}\notag\\ 
&\quad + \gamma_{\{6,13\}}u_{d,10} + \gamma_{12}u_{d,11} + \gamma_{15}u_{d,12} + \gamma_{25}u_{d,13}  \equiv 0.
\end{align}
From \eqref{c41}, we get
\begin{equation}\label{c411}
\gamma_t = 0 \mbox{ for } t \in \{5,\, 6,\, 7,\, 8,\, 9,\, 10,\, 11,\, 12,\, 13,\, 14,\, 15,\, 16,\, 25\}.
\end{equation}
Applying $p_{(1;3)}$ to \eqref{ctd61} and using \eqref{c411} we obtain
\begin{align}\label{c42}
p_{(1;3)}(\mathcal S) &\equiv \gamma_{\{3,17,18,20,29\}}u_{d,1} + \gamma_{\{1,3,18,21,28,29\}}u_{d,5}\notag\\ 
&\quad + \gamma_{\{2,3,20,22,29,30\}}u_{d,6} + \gamma_{\{4,24,33\}}u_{d,8} + \gamma_{18}u_{d,9}\notag\\ 
&\quad + \gamma_{20}u_{d,10} + \gamma_{19}u_{d,11} + \gamma_{23}u_{d,12} + \gamma_{26}u_{d,13}  \equiv 0.
\end{align}
The relation \eqref{c42} implies
\begin{equation}\label{c421}
\gamma_t = 0 \mbox{ for } t \in \{18,\, 20,\, 19,\, 23,\, 26\}.
\end{equation}
By a direct computation using \eqref{c411} and \eqref{c421} we obtain
\begin{align}\label{c43}
p_{(4;5)}(\mathcal S) &\equiv \gamma_{4}u_{d,1} + \gamma_{24}u_{d,6} + \gamma_{33}u_{d,9} + \gamma_{34}u_{d,10}\notag\\
&\quad + \gamma_{37}u_{d,11} + \gamma_{38}u_{d,12} + \gamma_{39}u_{d,13}  \equiv 0.
\end{align}
From \eqref{c43} it implies
\begin{equation}\label{c431}
\gamma_t = 0 \mbox{ for } t \in \{4,\, 24,\, 33,\, 34,\, 37,\, 38,\, 39\}.
\end{equation}
Based on \eqref{c411}, \eqref{c421} and \eqref{c431}, we get
\begin{align}\label{c44}
p_{(2;3)}(\mathcal S) &\equiv \gamma_{\{3,27,28,30,40\}}u_{d,1} + \gamma_{\{1,3,28,31,35,40\}}u_{d,5}\notag\\
&\quad + \gamma_{\{2,3,30,32,36,40\}}u_{d,6} + \gamma_{28}u_{d,9} + \gamma_{30}u_{d,10} + \gamma_{29}u_{d,11} \equiv 0.
\end{align}
From \eqref{c43} we have 
\begin{equation}\label{c441}
\gamma_t = 0 \mbox{ for } t = 28,\, 29,\, 30.
\end{equation}

By a similar computation using \eqref{c411}, \eqref{c421}, \eqref{c431}, \eqref{c441}, we obtain
\begin{align}\label{c45}
p_{(1;4)}(\mathcal S) &\equiv \gamma_{\{3,22,31,40\}}u_{d,1} + \gamma_{\{1,3,31,35,40\}}u_{d,3} + \gamma_{\{2,27,32,36\}}u_{d,4}\notag\\
&\quad + \gamma_{\{3,22,31,40\}}u_{d,6} + \gamma_{\{17,22,35\}}u_{d,10} + \gamma_{21}u_{d,11}  \equiv 0.\\
p_{(1;5)}(\mathcal S) &\equiv \gamma_{\{2,21,22,36,40\}}u_{d,1} + \gamma_{\{1,27,31,35\}}u_{d,2} + \gamma_{\{2,3,32,36,40\}}u_{d,3}\notag\\
&\quad + \gamma_{\{3,21,32\}}u_{d,5} + \gamma_{\{17,21,36,40\}}u_{d,9} + \gamma_{22}u_{d,11}  \equiv 0.
\end{align}
From the above equalities we get 
\begin{equation}\label{c4511}
\gamma_t = 0,\ t = 21,\, 22.
\end{equation}
By a direct computation using \eqref{c411}, \eqref{c421}, \eqref{c431}, \eqref{c441} and \eqref{c4511}, we have
\begin{align*}
p_{(2;4)}(\mathcal S) &\equiv \gamma_{\{3,32\}}u_{d,1} + \gamma_{\{1,3\}}u_{d,3} + \gamma_{\{2,17\}}u_{d,4}\notag\\
&\quad + \gamma_{\{3,32\}}u_{d,6} + \gamma_{\{27,32,35\}}u_{d,10} + \gamma_{31}u_{d,11}  \equiv 0.\\
p_{(2;5)}(\mathcal S) &\equiv \gamma_{\{2,31,32,40\}}u_{d,1} + \gamma_{\{1,17\}}u_{d,2} + \gamma_{\{2,3\}}u_{d,3}\notag\\
&\quad + \gamma_{\{3,31,40\}}u_{d,5} + \gamma_{\{27,31,36,40\}}u_{d,9} + \gamma_{32}u_{d,11}  \equiv 0.
\end{align*}
From the above equalities we get $\gamma_t =0$ for all $t, \ 1 \leqslant t \leqslant 40$. The proof is completed.
\end{proof}

\subsection{Computation of $QP_5^+((4)|(3)|^{d-1})$}\label{s32}\

\medskip
Observe that if $x$ is an admissible monomial of weight vector $(4)|(3)|^{d-1}|(1)$, then $x = yx_r^{2^{d}}$ with $1 \leqslant r \leqslant 5$ and $y$ an admissible monomial of weight vector $(4)|(3)|^{d-1}$. So, we need to determine the set $B_5((4)|(3)|^{d-1})$ for any $d \geqslant 5$.

From our work \cite{su2}, we have $|B_4((4)|(3)|^{d-1})| = |B_4^+((4)|(3)|^{d-1})| = 15$, hence we get $|B_5^0((4)|(3)|^{d-1})| = 15{5\choose 4} = 75$.

\begin{props}\label{mdd52} For any $d \geqslant 5$, $B_5^+((4)|(3)|^{d-1})$ is the set of $235$ admissible monomials which are determined as in Subsection $\ref{s52}$. Consequently, 
$$\dim QP_5^+((4)|(3)|^{d-1}) = 235,\ \dim QP_5((4)|(3)|^{d-1}) = 310.$$
\end{props}
The proposition is proved by using Theorems \ref{dlcb1}, \ref{dlsig} and the following lemmas.
\begin{lems}\label{bda62} Let $(i,j,t,u,v)$ be an arbitrary permutation of $(1,2,3,4,5)$. The following monomials are strictly inadmissible:
	
\smallskip
\ \! {\rm i)} $x_i^2x_jx_tx_u^3x_v^3,\, i<j<t$.
	
\smallskip
\ {\rm ii)} $x_i^3x_j^5x_t^5x_u^2x_v^7, \, x_i^3x_j^5x_t^5x_u^3x_v^6,$ $x_i^3x_j^4x_t^5x_u^3x_v^7$, $x_i^3x_j^5x_t^4x_u^3x_v^7$, $i<j<t<u$.

\smallskip
\ {\rm iii)} $x_r^{15}f_r(w),\, 1\leqslant r\leqslant 5$, with $w$ one of the following monomials:
 
\smallskip
\centerline{\begin{tabular}{lllll}
$x_1^{3}x_2^{4}x_3^{9}x_4^{15}$& $x_1^{3}x_2^{4}x_3^{15}x_4^{9}$& $x_1^{3}x_2^{5}x_3^{8}x_4^{15}$& $x_1^{3}x_2^{5}x_3^{9}x_4^{14}$& $x_1^{3}x_2^{5}x_3^{14}x_4^{9} $\cr  $x_1^{3}x_2^{5}x_3^{15}x_4^{8}$& $x_1^{3}x_2^{7}x_3^{12}x_4^{9}$& $x_1^{3}x_2^{7}x_3^{13}x_4^{8}$& $x_1^{3}x_2^{12}x_3x_4^{15}$& $x_1^{3}x_2^{12}x_3^{3}x_4^{13} $\cr  $x_1^{3}x_2^{12}x_3^{15}x_4$& $x_1^{3}x_2^{15}x_3^{4}x_4^{9}$& $x_1^{3}x_2^{15}x_3^{5}x_4^{8}$& $x_1^{3}x_2^{15}x_3^{12}x_4$& $x_1^{7}x_2^{3}x_3^{12}x_4^{9} $\cr  $x_1^{7}x_2^{3}x_3^{13}x_4^{8}$& $x_1^{7}x_2^{11}x_3^{4}x_4^{9}$& $x_1^{7}x_2^{11}x_3^{5}x_4^{8}$& $x_1^{7}x_2^{11}x_3^{12}x_4$& $x_1^{15}x_2^{3}x_3^{4}x_4^{9} $\cr  $x_1^{15}x_2^{3}x_3^{5}x_4^{8}$& $x_1^{15}x_2^{3}x_3^{12}x_4$& &&\cr   
\end{tabular}}

\smallskip
{\begin{tabular}{lllll}
{\rm iv)}&$x_1^{3}x_2^{7}x_3^{13}x_4^{9}x_5^{14}$& $x_1^{3}x_2^{7}x_3^{13}x_4^{14}x_5^{9}$& $x_1^{3}x_2^{12}x_3^{7}x_4^{11}x_5^{13}$& $x_1^{7}x_2^{3}x_3^{13}x_4^{9}x_5^{14} $\cr  &$x_1^{7}x_2^{3}x_3^{13}x_4^{14}x_5^{9}$& $x_1^{7}x_2^{8}x_3^{7}x_4^{11}x_5^{13}$& $x_1^{7}x_2^{9}x_3^{6}x_4^{11}x_5^{13}$& $x_1^{7}x_2^{11}x_3^{5}x_4^{9}x_5^{14} $\cr  &$x_1^{7}x_2^{11}x_3^{5}x_4^{14}x_5^{9}$& $x_1^{7}x_2^{11}x_3^{12}x_4^{3}x_5^{13}$& $x_1^{7}x_2^{11}x_3^{12}x_4^{7}x_5^{9}$& $x_1^{7}x_2^{11}x_3^{13}x_4^{6}x_5^{9} $\cr  &$x_1^{7}x_2^{11}x_3^{13}x_4^{7}x_5^{8}$& & &\cr 
\end{tabular}}
	
\smallskip
\ \! {\rm v)} $x_r^{31}f_r(x_1^{3}x_2^{7}x_3^{24}x_4^{29}),\, x_r^{31}f_r(x_1^{7}x_2^{3}x_3^{24}x_4^{29}),\, 1\leqslant r\leqslant 5$, and 

\smallskip
\centerline{\begin{tabular}{llll}
$x_1^{3}x_2^{15}x_3^{20}x_4^{27}x_5^{29}$& $x_1^{3}x_2^{15}x_3^{23}x_4^{24}x_5^{29}$& $x_1^{7}x_2^{7}x_3^{27}x_4^{28}x_5^{25}$& $x_1^{7}x_2^{9}x_3^{23}x_4^{26}x_5^{29} $\cr  $x_1^{7}x_2^{9}x_3^{23}x_4^{27}x_5^{28}$& $x_1^{7}x_2^{11}x_3^{20}x_4^{27}x_5^{29}$& $x_1^{7}x_2^{11}x_3^{23}x_4^{24}x_5^{29}$& $x_1^{7}x_2^{11}x_3^{23}x_4^{28}x_5^{25} $\cr  $x_1^{7}x_2^{11}x_3^{23}x_4^{29}x_5^{24}$& $x_1^{7}x_2^{15}x_3^{16}x_4^{27}x_5^{29}$& $x_1^{7}x_2^{15}x_3^{19}x_4^{24}x_5^{29}$& $x_1^{7}x_2^{27}x_3^{7}x_4^{24}x_5^{29}$\cr  $x_1^{7}x_2^{27}x_3^{7}x_4^{28}x_5^{25}$& $x_1^{7}x_2^{27}x_3^{7}x_4^{29}x_5^{24}$& $x_1^{15}x_2^{3}x_3^{20}x_4^{27}x_5^{29}$& $x_1^{15}x_2^{3}x_3^{23}x_4^{24}x_5^{29}$\cr  $x_1^{15}x_2^{7}x_3^{16}x_4^{27}x_5^{29}$& $x_1^{15}x_2^{7}x_3^{19}x_4^{24}x_5^{29}$& $x_1^{15}x_2^{15}x_3^{16}x_4^{19}x_5^{29}$& $x_1^{15}x_2^{15}x_3^{19}x_4^{16}x_5^{29}$\cr  $x_1^{15}x_2^{23}x_3^{3}x_4^{24}x_5^{29}$& & & \cr  
\end{tabular}}
\end{lems}
\begin{proof} Each monomial in the lemma is of weight vector $\omega_{(t)} = (4)|(3)|^t$ with $1 \leqslant t \leqslant 4$. Part i) is easy. Part ii) is proved in Ph\'uc \cite{ph1}. We prove the remain for some monomials. The other can be proved by a similar computation. Let $x = x_5^{15}f_5(x_1^{7}x_2^{11}x_3^{4}x_4^{9}) = x_1^{7}x_2^{11}x_3^{4}x_4^{9}x_5^{15}$ be a monomial in Part iii). By a direct computation we have
\begin{align*}
x &= x_1^{5}x_2^{11}x_3^{2}x_4^{13}x_5^{15} + x_1^{5}x_2^{11}x_3^{4}x_4^{11}x_5^{15} + x_1^{5}x_2^{11}x_3^{8}x_4^{7}x_5^{15} + x_1^{7}x_2^{9}x_3^{4}x_4^{11}x_5^{15} \\ &\quad + x_1^{7}x_2^{10}x_3x_4^{13}x_5^{15} + x_1^{7}x_2^{11}x_3x_4^{12}x_5^{15} +  Sq^1\big(x_1^{7}x_2^{7}x_3x_4^{7}x_5^{23} + x_1^{7}x_2^{7}x_3x_4^{11}x_5^{19}\\ &\quad + x_1^{7}x_2^{11}x_3x_4^{7}x_5^{19} + x_1^{7}x_2^{11}x_3x_4^{11}x_5^{15}\big) +  Sq^2\big(x_1^{7}x_2^{7}x_3^{2}x_4^{7}x_5^{21} + x_1^{7}x_2^{7}x_3^{2}x_4^{13}x_5^{15}\\ &\quad + x_1^{7}x_2^{10}x_3x_4^{11}x_5^{15} + x_1^{7}x_2^{13}x_3^{2}x_4^{7}x_5^{15} + x_1^{9}x_2^{7}x_3^{2}x_4^{11}x_5^{15}\big) +  Sq^4\big(x_1^{5}x_2^{7}x_3^{2}x_4^{13}x_5^{15}\\ &\quad + x_1^{5}x_2^{11}x_3^{4}x_4^{7}x_5^{15} + x_1^{7}x_2^{11}x_3^{2}x_4^{7}x_5^{15}\big) +  Sq^8\big(x_1^{7}x_2^{7}x_3^{2}x_4^{7}x_5^{15}\big) \ 
\mbox{mod}(P_5^-(\omega_{(3)})).
\end{align*}
Hence, the monomial $x$ is strictly inadmissible. Let $y = x_1^{7}x_2^{9}x_3^{6}x_4^{11}x_5^{13}$ be a monomial in Part iv). We have
\begin{align*}
y &= x_1^{5}x_2^{7}x_3^{10}x_4^{11}x_5^{13} + x_1^{5}x_2^{11}x_3^{3}x_4^{13}x_5^{14} + x_1^{5}x_2^{11}x_3^{3}x_4^{14}x_5^{13} + x_1^{5}x_2^{11}x_3^{6}x_4^{11}x_5^{13}\\ &\quad + x_1^{7}x_2^{7}x_3^{8}x_4^{11}x_5^{13} + x_1^{7}x_2^{9}x_3^{3}x_4^{13}x_5^{14} + x_1^{7}x_2^{9}x_3^{3}x_4^{14}x_5^{13} + Sq^1\big(x_1^{7}x_2^{7}x_3^{5}x_4^{13}x_5^{13}\big)\\ &\quad +  Sq^2\big(x_1^{7}x_2^{7}x_3^{3}x_4^{13}x_5^{14} + x_1^{7}x_2^{7}x_3^{3}x_4^{14}x_5^{13} + x_1^{7}x_2^{7}x_3^{6}x_4^{11}x_5^{13}\big)\\ &\quad + Sq^4\big(x_1^{5}x_2^{7}x_3^{3}x_4^{13}x_5^{14} + x_1^{5}x_2^{7}x_3^{3}x_4^{14}x_5^{13} + x_1^{5}x_2^{7}x_3^{6}x_4^{11}x_5^{13} \big) \ \mbox{mod}(P_5^-(\omega_{(3)})).
\end{align*}
This equality shows that the monomial $y$ is strictly inadmissible. 

Let $u= x_1^{7}x_2^{27}x_3^{7}x_4^{28}x_5^{25}$ and $v = x_1^{15}x_2^{15}x_3^{19}x_4^{16}x_5^{29}$ as listed in Part v). We have
\begin{align*}
u &= x_1^{5}x_2^{30}x_3^{3}x_4^{27}x_5^{29} + x_1^{5}x_2^{30}x_3^{7}x_4^{25}x_5^{27} + x_1^{5}x_2^{30}x_3^{11}x_4^{25}x_5^{23} + x_1^{7}x_2^{27}x_3^{3}x_4^{28}x_5^{29}\\ &\quad + x_1^{7}x_2^{27}x_3^{5}x_4^{25}x_5^{30} + x_1^{7}x_2^{27}x_3^{5}x_4^{28}x_5^{27} + x_1^{7}x_2^{27}x_3^{6}x_4^{25}x_5^{29} + x_1^{7}x_2^{27}x_3^{7}x_4^{25}x_5^{28}\\ &\quad + Sq^1\big(x_1^{7}x_2^{29}x_3^{3}x_4^{25}x_5^{29} + x_1^{7}x_2^{29}x_3^{5}x_4^{25}x_5^{27} + x_1^{7}x_2^{29}x_3^{7}x_4^{25}x_5^{25}\big)\\ &\quad +  Sq^2\big(x_1^{7}x_2^{27}x_3^{3}x_4^{25}x_5^{30} + x_1^{7}x_2^{27}x_3^{3}x_4^{26}x_5^{29} + x_1^{7}x_2^{27}x_3^{5}x_4^{26}x_5^{27} + x_1^{7}x_2^{27}x_3^{6}x_4^{25}x_5^{27}\\ &\quad + x_1^{7}x_2^{27}x_3^{7}x_4^{25}x_5^{26} + x_1^{7}x_2^{27}x_3^{7}x_4^{26}x_5^{25} + x_1^{7}x_2^{30}x_3^{3}x_4^{15}x_5^{37} + x_1^{7}x_2^{30}x_3^{3}x_4^{23}x_5^{29}\\ &\quad + x_1^{7}x_2^{30}x_3^{7}x_4^{25}x_5^{23} + x_1^{7}x_2^{30}x_3^{9}x_4^{19}x_5^{27} + x_1^{7}x_2^{38}x_3^{3}x_4^{15}x_5^{29} + x_1^{7}x_2^{38}x_3^{3}x_4^{21}x_5^{23}\big)\\ &\quad +  Sq^4\big(x_1^{5}x_2^{30}x_3^{3}x_4^{23}x_5^{29} + x_1^{5}x_2^{30}x_3^{7}x_4^{25}x_5^{23} + x_1^{11}x_2^{30}x_3^{5}x_4^{15}x_5^{29} + x_1^{11}x_2^{30}x_3^{5}x_4^{21}x_5^{23}\big)\\ &\quad +  Sq^8\big(x_1^{7}x_2^{30}x_3^{5}x_4^{15}x_5^{29} + x_1^{7}x_2^{30}x_3^{5}x_4^{21}x_5^{23}\big) \ \mbox{mod}(P_5^-(\omega_{(4)})).
\end{align*}
Hence, the monomial $u$ is strictly inadmissible.
\begin{align*}
v &= x_1^{9}x_2^{15}x_3^{19}x_4^{21}x_5^{30} + x_1^{9}x_2^{23}x_3^{11}x_4^{21}x_5^{30} + x_1^{11}x_2^{7}x_3^{19}x_4^{28}x_5^{29} + x_1^{11}x_2^{7}x_3^{21}x_4^{26}x_5^{29}\\ &\quad + x_1^{11}x_2^{13}x_3^{19}x_4^{22}x_5^{29} + x_1^{11}x_2^{14}x_3^{21}x_4^{19}x_5^{29} + x_1^{11}x_2^{21}x_3^{11}x_4^{22}x_5^{29} + x_1^{11}x_2^{22}x_3^{13}x_4^{19}x_5^{29}\\ &\quad + x_1^{11}x_2^{23}x_3^{11}x_4^{20}x_5^{29} + x_1^{11}x_2^{23}x_3^{19}x_4^{12}x_5^{29} + x_1^{12}x_2^{15}x_3^{19}x_4^{19}x_5^{29} + x_1^{12}x_2^{23}x_3^{11}x_4^{19}x_5^{29}\\ &\quad + x_1^{15}x_2^{7}x_3^{17}x_4^{26}x_5^{29} + x_1^{15}x_2^{11}x_3^{19}x_4^{20}x_5^{29} + x_1^{15}x_2^{14}x_3^{17}x_4^{19}x_5^{29} + x_1^{15}x_2^{15}x_3^{16}x_4^{19}x_5^{29}\\ &\quad +  Sq^1\big(x_1^{15}x_2^{15}x_3^{11}x_4^{19}x_5^{33} + x_1^{15}x_2^{17}x_3^{11}x_4^{21}x_5^{29} + x_1^{15}x_2^{19}x_3^{11}x_4^{19}x_5^{29}\\ &\quad + x_1^{19}x_2^{15}x_3^{11}x_4^{19}x_5^{29}\big) +  Sq^2\big(x_1^{15}x_2^{11}x_3^{11}x_4^{26}x_5^{29} + x_1^{15}x_2^{15}x_3^{13}x_4^{19}x_5^{30}\\ &\quad + x_1^{15}x_2^{18}x_3^{11}x_4^{19}x_5^{29}\big) +  Sq^4\big(x_1^{15}x_2^{7}x_3^{11}x_4^{28}x_5^{29} + x_1^{15}x_2^{7}x_3^{13}x_4^{26}x_5^{29}\\ &\quad + x_1^{15}x_2^{13}x_3^{11}x_4^{22}x_5^{29} + x_1^{15}x_2^{14}x_3^{13}x_4^{19}x_5^{29} + x_1^{15}x_2^{15}x_3^{11}x_4^{19}x_5^{30} + x_1^{15}x_2^{15}x_3^{12}x_4^{19}x_5^{29}\\ &\quad + x_1^{15}x_2^{15}x_3^{19}x_4^{12}x_5^{29}\big) +  Sq^8\big(x_1^{9}x_2^{15}x_3^{11}x_4^{21}x_5^{30} + x_1^{11}x_2^{7}x_3^{11}x_4^{28}x_5^{29}\\ &\quad + x_1^{11}x_2^{7}x_3^{13}x_4^{26}x_5^{29} + x_1^{11}x_2^{13}x_3^{11}x_4^{22}x_5^{29} + x_1^{11}x_2^{14}x_3^{13}x_4^{19}x_5^{29}\\ &\quad + x_1^{11}x_2^{15}x_3^{11}x_4^{20}x_5^{29} + x_1^{11}x_2^{15}x_3^{19}x_4^{12}x_5^{29} + x_1^{12}x_2^{15}x_3^{11}x_4^{19}x_5^{29}\\ &\quad + x_1^{23}x_2^{11}x_3^{11}x_4^{12}x_5^{29}\big) +  Sq^{16}\big(x_1^{15}x_2^{11}x_3^{11}x_4^{12}x_5^{29}\big) \ \mbox{mod}(P_5^-(\omega_{(4)})).
\end{align*}
Hence, the monomial $v$ is strictly inadmissible.
\end{proof}

\begin{lems}\label{bda63} Let $(i,j,t,u,v)$ be an arbitrary permutation of $(1,2,3,4,5)$. The following monomials are strictly inadmissible:
	
\smallskip
\ \! {\rm i)} $x_i^2x_jx_t^3x_u^3,\, i< j;$\, $x_i^2x_jx_tx_u^2x_v^3,\, i< j< t.$
	
\smallskip
\ {\rm ii)} $x_r^7f_r(w)$, $1 \leqslant r \leqslant 5$, with $w$ one of the following monomials in $P_4$:

\smallskip
\centerline{\begin{tabular}{lllll}
$x_i^3x_j^4x_t^7$& $x_1x_2^{2}x_3^{6}x_4^{5}$& $x_1x_2^{2}x_3^{6}x_4^{5}$& $x_1x_2^{6}x_3^{3}x_4^{4}$& $x_1x_2^{6}x_3^{6}x_4 $\cr  
$x_1^{3}x_2^{4}x_3x_4^{6}$& $x_1^{3}x_2^{4}x_3^{3}x_4^{4}$& $x_1^{3}x_2^{4}x_3^{4}x_4^{3}$& $x_1^{3}x_2^{4}x_3^{5}x_4^{2}$& $x_1^{3}x_2^{5}x_3^{4}x_4^{2} $\cr  
\end{tabular}}
		
\smallskip
	{\begin{tabular}{llllll}
{\rm iii)}&$x_1x_2^{6}x_3^{3}x_4^{6}x_5^{5}$& $x_1x_2^{6}x_3^{6}x_4^{3}x_5^{5}$& $x_1^{3}x_2^{5}x_3^{5}x_4^{2}x_5^{6}$& $x_1^{3}x_2^{5}x_3^{5}x_4^{6}x_5^{2}$& $x_1^{3}x_2^{5}x_3^{6}x_4^{5}x_5^{2} $\cr  &$x_1^{3}x_2^{4}x_3^{5}x_4^{3}x_5^{6}$& $x_1^{3}x_2^{4}x_3^{5}x_4^{6}x_5^{3}$& $x_1^{3}x_2^{5}x_3^{4}x_4^{3}x_5^{6}$& $x_1^{3}x_2^{5}x_3^{4}x_4^{6}x_5^{3}$& $x_1^{3}x_2^{5}x_3^{6}x_4^{4}x_5^{3} $\cr 
\end{tabular}}

\ {\rm iv)} $x_r^{15}f_r(x)$, $1 \leqslant r \leqslant 5$, with $x$ one of the following monomials:
$$x_1x_2^{7}x_3^{10}x_4^{12},\, x_1^{3}x_2^{3}x_3^{12}x_4^{12},\, x_1^{3}x_2^{5}x_3^{8}x_4^{14},\, x_1^{3}x_2^{5}x_3^{14}x_4^{8},\, x_1^{7}x_2x_3^{10}x_4^{12},\, x_1^{7}x_2^{7}x_3^{8}x_4^{8}$$ 
	
\smallskip
{\begin{tabular}{lllll}
{\rm v)}&$x_1^{3}x_2^{5}x_3^{9}x_4^{14}x_5^{14}$& $x_1^{3}x_2^{5}x_3^{14}x_4^{9}x_5^{14}$& $x_1^{3}x_2^{5}x_3^{14}x_4^{11}x_5^{12}$& $x_1^{3}x_2^{7}x_3^{11}x_4^{12}x_5^{12} $\cr  &$x_1^{3}x_2^{7}x_3^{13}x_4^{8}x_5^{14}$& $x_1^{3}x_2^{7}x_3^{13}x_4^{14}x_5^{8}$& $x_1^{3}x_2^{12}x_3^{3}x_4^{13}x_5^{14}$& $x_1^{3}x_2^{13}x_3^{14}x_4^{3}x_5^{12} $\cr  &$x_1^{7}x_2^{3}x_3^{11}x_4^{12}x_5^{12}$& $x_1^{7}x_2^{3}x_3^{13}x_4^{8}x_5^{14}$& $x_1^{7}x_2^{3}x_3^{13}x_4^{14}x_5^{8}$& $x_1^{7}x_2^{9}x_3^{7}x_4^{10}x_5^{12} $\cr  &$x_1^{7}x_2^{11}x_3^{3}x_4^{12}x_5^{12}$& $x_1^{7}x_2^{11}x_3^{5}x_4^{8}x_5^{14}$& $x_1^{7}x_2^{11}x_3^{5}x_4^{14}x_5^{8}$& $x_1^{7}x_2^{11}x_3^{13}x_4^{6}x_5^{8} $\cr 
\end{tabular}}
	
\smallskip
\ {\rm vi)} $x_1^{3}x_2^{7}x_3^{24}x_4^{29}x_5^{30}$,\, $x_1^{7}x_2^{3}x_3^{24}x_4^{29}x_5^{30}$,\, $x_1^{7}x_2^{7}x_3^{25}x_4^{26}x_5^{28}$,\, $x_1^{15}x_2^{15}x_3^{17}x_4^{18}x_5^{28}.$
\end{lems}
\begin{proof} Each monomial in the lemma is of weight vector $(3)|^t$ with $2 \leqslant t \leqslant 5$. Parts i), ii), iii) are due to Ph\'uc \cite{ph1}. Let $x = x_5^{15}f_5(x_1^{7}x_2x_3^{10}x_4^{12}) = x_1^{7}x_2x_3^{10}x_4^{12}x_5^{15}$ be a monomial in Part iv). We have
\begin{align*}
x &= x_1^{4}x_2^{4}x_3^{11}x_4^{11}x_5^{15} + x_1^{4}x_2^{8}x_3^{7}x_4^{11}x_5^{15} + x_1^{5}x_2^{2}x_3^{11}x_4^{12}x_5^{15} + x_1^{7}x_2x_3^{8}x_4^{14}x_5^{15}\\ &\quad +  Sq^1\big(x_1^{7}x_2x_3^{8}x_4^{13}x_5^{15} + x_1^{7}x_2x_3^{9}x_4^{12}x_5^{15} + x_1^{7}x_2^{4}x_3^{7}x_4^{11}x_5^{15}\big) +  Sq^2\big(x_1^{7}x_2^{2}x_3^{7}x_4^{12}x_5^{15}\\ &\quad + x_1^{7}x_2^{2}x_3^{8}x_4^{11}x_5^{15}\big) + Sq^4\big(x_1^{4}x_2^{4}x_3^{7}x_4^{11}x_5^{15} + x_1^{5}x_2^{2}x_3^{7}x_4^{12}x_5^{15}\big) \ \mbox{mod}(P_5^-((3)|^4)).
\end{align*}
Hence, the monomial $x$ is strictly inadmissible.

Let $y = x_1^{3}x_2^{7}x_3^{13}x_4^{8}x_5^{14}$ be the monomial in Part v). By a direct computation we have
\begin{align*}
y &= x_1^{2}x_2^{7}x_3^{13}x_4^{9}x_5^{14} + x_1^{3}x_2^{5}x_3^{11}x_4^{12}x_5^{14} + x_1^{3}x_2^{5}x_3^{13}x_4^{10}x_5^{14} + x_1^{3}x_2^{7}x_3^{9}x_4^{12}x_5^{14}\\ &\quad + x_1^{3}x_2^{7}x_3^{12}x_4^{9}x_5^{14} +  Sq^1\big(x_1^{3}x_2^{7}x_3^{7}x_4^{5}x_5^{22} + x_1^{3}x_2^{7}x_3^{7}x_4^{9}x_5^{18} + x_1^{3}x_2^{7}x_3^{11}x_4^{5}x_5^{18}\\ &\quad + x_1^{3}x_2^{7}x_3^{11}x_4^{9}x_5^{14}\big) +  Sq^2\big(x_1^{2}x_2^{7}x_3^{7}x_4^{5}x_5^{22} + x_1^{2}x_2^{7}x_3^{11}x_4^{9}x_5^{14} + x_1^{5}x_2^{7}x_3^{7}x_4^{10}x_5^{14}\\ &\quad + x_1^{5}x_2^{7}x_3^{11}x_4^{6}x_5^{14}\big) + Sq^4\big(x_1^{3}x_2^{5}x_3^{7}x_4^{12}x_5^{14} + x_1^{3}x_2^{5}x_3^{13}x_4^{6}x_5^{14} + x_1^{3}x_2^{11}x_3^{7}x_4^{6}x_5^{14}\big)\\ &\quad +  Sq^8\big(x_1^{3}x_2^{7}x_3^{7}x_4^{6}x_5^{14}\big) \ \mbox{mod}(P_5^-((3)|^4)).
\end{align*}
This equality implies the monomial $y$ is strictly inadmissible.

Let $z = x_1^{15}x_2^{15}x_3^{17}x_4^{18}x_5^{28}$ be the monomial in Part v). A direct computation shows
\begin{align*}
z &= x_1^{8}x_2^{15}x_3^{23}x_4^{19}x_5^{28} + x_1^{8}x_2^{23}x_3^{15}x_4^{19}x_5^{28} + x_1^{9}x_2^{15}x_3^{23}x_4^{18}x_5^{28} + x_1^{9}x_2^{23}x_3^{15}x_4^{18}x_5^{28}\\ &\quad + x_1^{11}x_2^{12}x_3^{23}x_4^{19}x_5^{28} + x_1^{11}x_2^{13}x_3^{23}x_4^{18}x_5^{28} + x_1^{11}x_2^{20}x_3^{15}x_4^{19}x_5^{28}\\ &\quad + x_1^{11}x_2^{21}x_3^{15}x_4^{18}x_5^{28} + x_1^{15}x_2^{12}x_3^{19}x_4^{19}x_5^{28} + x_1^{15}x_2^{13}x_3^{19}x_4^{18}x_5^{28}\\ &\quad + x_1^{15}x_2^{15}x_3^{16}x_4^{19}x_5^{28} + Sq^1\big(x_1^{15}x_2^{15}x_3^{15}x_4^{19}x_5^{28}\big) +  Sq^2\big(x_1^{15}x_2^{15}x_3^{15}x_4^{18}x_5^{28}\big)\\ &\quad + Sq^4\big(x_1^{15}x_2^{12}x_3^{15}x_4^{19}x_5^{28} + x_1^{15}x_2^{13}x_3^{15}x_4^{18}x_5^{28}\big) + Sq^8\big(x_1^{8}x_2^{15}x_3^{15}x_4^{19}x_5^{28}\\ &\quad + x_1^{9}x_2^{15}x_3^{15}x_4^{18}x_5^{28} + x_1^{11}x_2^{12}x_3^{15}x_4^{19}x_5^{28} + x_1^{11}x_2^{13}x_3^{15}x_4^{18}x_5^{28}\big) \ \mbox{mod}(P_5^-((3)|^5)).
\end{align*}
Hence, the monomial $z$ is strictly inadmissible. The others can be proved by a similar computation.
\end{proof}
\begin{proof}[Proof of Proposition $\ref{mdd52}$] Let $A(d)$ and $C(d)$ be as in Subsection \ref{s52} with $d \geqslant 5$ and let $x \in P_5^+(\omega)$ be an admissible monomial with $\omega := (4)|(3)|^{d-1}$. By a direct computation we see that if $x \notin A(d)\cup C(d)$, then there is a monomial $w$ as given in one of Lemmas \ref{bda62} or \ref{bda63} such that $x = yw^{2^r}z^{2^{r+u}}$ with $r,\, u$ nonnegative integers, $2\leqslant u \leqslant 5$, and $y,\, z$ suitable monomials. By Theorem \ref{dlcb1}, $x$ is inadmissible. Hence, $B_5^+(\omega) \subset A(d)\cup C(d)$.
	
Now we prove that the set $[A(d)\cup C(d)]_{\omega}$ is linearly independent in $QP_5(\omega)$.

Consider the subspaces $\langle [A(d)]_{\omega}\rangle \subset QP_5(\omega)$ and $\langle [C(d)]_{\omega}\rangle \subset QP_5(\omega)$. It is easy to see that for any $x\in A(d)$, we have $x = f_i(y)$ with $y$ an admissible monomial of weight vector $(3)|(2)|^{d-1}$ in $P_4$. By Proposition \ref{mdmo}, $x$ is admissible. This implies $\dim \langle [A(d)]_{\omega}\rangle = 160$. Since $\nu(x) = 2^d-1$ for all $x\in A(d)$ and $\nu(x) < 2^d-1$ for all $x\in C(d)$, we obtain $\langle [A(d)]_{\omega}\rangle \cap \langle [C(d)]_{\omega}\rangle = \{0\}$. Hence, we need only to prove the set $[C(d)]_{\omega}=\{[b_{d,t}]_{\omega}: 1 \leqslant t \leqslant 75 \}$ is linearly independent in $QP_5(\omega)$, where the monomials $b_{d,t}: 1 \leqslant t \leqslant 75$, are determined as in Subsection \ref{s52}. 

Suppose there is a linear relation
\begin{equation}\label{ctd612}
\mathcal S:= \sum_{1\leqslant t \leqslant 75}\gamma_tb_{d,t} \equiv_{\omega} 0,
\end{equation}
where $\gamma_t \in \mathbb F_2$. We denote $\gamma_{\mathbb J} = \sum_{t \in \mathbb J}\gamma_t$ for any $\mathbb J \subset \{t\in \mathbb N:1\leqslant t \leqslant 75\}$.
	
Let $v_{d,u},\, 1\leqslant u \leqslant 15$, be as in Subsection \ref{s52} and the homomorphism $p_{(i;I)}:P_5\to P_4$ which is defined by \eqref{ct23} for $k=5$. From Lemma \ref{bdm}, we see that $p_{(i;I)}$ passes to a homomorphism from $QP_5(\omega)$ to $QP_4(\omega)$. By applying $p_{(i;j)}$, $1\leqslant i < j \leqslant 5,$ to (\ref{ctd612}), we obtain
\begin{align*}
&p_{(1;2)}(S) \equiv_{\omega} \gamma_{4}v_{d,8} + \gamma_{8}v_{d,12}   \equiv_{\omega} 0,\\
&p_{(1;3)}(S) \equiv_{\omega} \gamma_{\{3,18,43\}}v_{d,7} + \gamma_{5}v_{d,8}  \equiv_{\omega} 0,\\
&p_{(1;4)}(S) \equiv_{\omega} \gamma_{\{2,16,46\}}v_{d,6} + \gamma_{6}v_{d,8}  \equiv_{\omega} 0,\\
&p_{(1;5)}(S) \equiv_{\omega}  \gamma_{\{1,17,20,23,25,26,31,48\}}v_{d,5} + \gamma_{7}v_{d,8}  \equiv_{\omega} 0,\\
&p_{(2;3)}(S) \equiv_{\omega} \gamma_{\{35,40,43,57\}}v_{d,7} + \gamma_{65}v_{d,8}  \equiv_{\omega} 0,\\
&p_{(2;4)}(S) \equiv_{\omega} \gamma_{\{34,38,46\}}v_{d,6} + \gamma_{66}v_{d,8}  \equiv_{\omega} 0,\\
&p_{(2;5)}(S) \equiv_{\omega} \gamma_{\{33,39,42,48,69,70\}}v_{d,5} + \gamma_{67}v_{d,8}  \equiv_{\omega} 0,\\
&p_{(3;4)}(S) \equiv_{\omega} \gamma_{\{54,55,57,58,62\}}v_{d,6} + \gamma_{72}v_{d,8}  \equiv_{\omega} 0,\\
&p_{(3;5)}(S) \equiv_{\omega} \gamma_{\{53,56,74\}}v_{d,5} + \gamma_{73}v_{d,8}  \equiv_{\omega} 0,\\
&p_{(4;5)}(S) \equiv_{\omega} \gamma_{\{61,62,63,64\}}v_{d,5} + \gamma_{75}v_{d,8}  \equiv_{\omega} 0.
\end{align*}
From the above equalities, we get
\begin{equation}\label{c61}
\gamma_t = 0 \mbox{ for } t \in \{4,\, 5,\, 6,\, 7,\, 8,\, 65,\, 66,\, 67,\, 72,\, 73,\, 75\}.
\end{equation}
By applying the homomorphism $p_{(1;(u,v))}$, $2\leqslant u < v \leqslant 4,$ to \eqref{ctd612} and using \eqref{c61}, we obtain
\begin{align*}
p_{(1;(2,3))}(S) &\equiv_{\omega} \gamma_{\{3,13,18,35,40,43\}}v_{d,7} + \gamma_{21}v_{d,8} + \gamma_{\{29,57\}}v_{d,12} + \gamma_{32}v_{d,14}  \equiv_{\omega} 0,\\
p_{(1;(2,4))}(S) &\equiv_{\omega} \gamma_{\{2,12,16,34,38,46\}}v_{d,6} + \gamma_{22}v_{d,8} + \gamma_{\{28,55\}}v_{d,11} + \gamma_{30}v_{d,12}  \equiv_{\omega} 0,\\
p_{(1;(3,4))}(S) &\equiv_{\omega} \gamma_{\{10,37,38,40,41,68\}}v_{d,4} + \gamma_{\{2,15,16,43,44,46,49,54,55,57,58\}}v_{d,6}\\ 
&\hskip3.8cm+ \gamma_{\{3,18,19,43,44,46,62\}}v_{d,7} + \gamma_{24}v_{d,8}  \equiv_{\omega} 0.
\end{align*}
These equalities imply
\begin{equation}\label{c62}
\gamma_t = 0 \mbox{ for } t \in \{21,\, 22,\, 24,\, 30,\, 32\},\ \gamma_{55} = \gamma_{28},\ \gamma_{57} = \gamma_{29} .
\end{equation}
Applying the homomorphism $p_{(1;(u,5))}$, $2\leqslant u  \leqslant 4,$ to \eqref{ctd612} and using \eqref{c61}, \eqref{c62} give
\begin{align*}
p_{(1;(2,5))}(S) &\equiv_{\omega} \gamma_{\{1,11,17,20,23,25,26,31,33,39,42,48,50,51,69,70\}}v_{d,5}\\ 
&\qquad + \gamma_{23}v_{d,8} + \gamma_{\{27,56,59,60,63\}}v_{d,10} + \gamma_{31}v_{d,12} \equiv_{\omega} 0,\\
p_{(1;(3,5))}(S) &\equiv_{\omega}  \gamma_{\{1,14,17,18,20,23,25,26,29,31,43,45,48,52,53,56,74\}}v_{d,5}\\
&\qquad + \gamma_{\{9,13,36,39,47,71\}}v_{d,3} + \gamma_{\{3,18,20,43,45,48\}}v_{d,7} + \gamma_{25}v_{d,8} \equiv_{\omega} 0,\\
p_{(1;(4,5))}(S) &\equiv_{\omega} \gamma_{\{9,10,11,12,14,15,27,28,41,42,44,45,49,50,58,59\}}v_{d,2}\\
&\qquad + \gamma_{\{1,16,17,19,20,23,25,26,31,46,47,48,51,52,60,61,62,63\}}v_{d,5}\\
&\qquad+ \gamma_{\{2,16,17,46,47,48,51,52,60,64\}}v_{d,6} + \gamma_{26}v_{d,8} +  \equiv_{\omega} 0.
\end{align*}
These equalities imply
\begin{equation}\label{c63}
\gamma_t = 0 \mbox{ for } t \in \{23,\, 25,\, 26,\, 31\}.
\end{equation}

By applying the homomorphism $p_{(i;(u,v))}$, $2\leqslant i< u < v \leqslant 5,$ to \eqref{ctd612} and using \eqref{c61}, \eqref{c62}, \eqref{c63}, we have
\begin{align*}
p_{(2;(3,4))}(S) &\equiv_{\omega} \gamma_{\{10,12,13,15,16,18,19,28,29\}}v_{d,4} + \gamma_{\{28,34,37,38,43,44,46,54\}}v_{d,6}\\ &\qquad+ \gamma_{\{29,35,40,41,43,44,46,49,58,62\}}v_{d,7} + \gamma_{68}v_{d,8}  \equiv_{\omega} 0,\\
p_{(2;(3,5))}(S) &\equiv_{\omega} \gamma_{\{9,11,14,17,27,47,51,52,60\}}v_{d,3}\\ &\qquad+ \gamma_{\{29,33,36,39,40,42,43,45,48,50,53,56,59,69,70,71,74\}}v_{d,5}\\ &\qquad+ \gamma_{\{29,35,40,42,43,45,48,50,59\}}v_{d,7} + \gamma_{69}v_{d,8} \equiv_{\omega} 0,\\
p_{(2;(4,5))}(S) &\equiv_{\omega} \gamma_{\{9,10,19,20,36,37,44,45\}}v_{d,2}\\ &\qquad+ \gamma_{\{33,38,39,41,42,46,47,48,61,62,63,68,69,70,71\}}v_{d,5}\\ &\qquad+ \gamma_{\{34,38,39,46,47,48,64,71\}}v_{d,6} + \gamma_{70}v_{d,8} \equiv_{\omega} 0,\\
p_{(3;(4,5))}(S) &\equiv_{\omega} \gamma_{\{27,28,29\}}v_{d,2} + \gamma_{\{28,53,56,60,61,74\}}v_{d,5}\\ &\qquad+ \gamma_{\{28,29,54,56,58,59,60,62,63,64\}}v_{d,6} + \gamma_{74}v_{d,8}   \equiv_{\omega} 0.
\end{align*}
These equalities imply
\begin{equation}\label{c64}
\gamma_t = 0 \mbox{ for } t \in \{68,\, 69,\, 70,\,	74\},\ \gamma_{56} = \gamma_{53}.
\end{equation}
By applying the homomorphisms $p_{(1;(2,3,4))}$, $p_{(1;(2,3,5))}$ to \eqref{ctd612} and using \eqref{c61}, \eqref{c62}, \eqref{c63}, \eqref{c64}, we obtain
\begin{align*}
p_{(1;(2,3,4))}(S) &\equiv_{\omega} \gamma_{\{2,12,15,16,28,29,34,37,38,43,44,46,49,54,58\}}v_{d,6}\\ &\qquad+ \gamma_{\{10,34,35,37,38,40,41\}}v_{d,4} +  \gamma_{\{3,13,18,19,35,40,41,43,44,46,62\}}v_{d,7}\\ &\qquad+ \gamma_{49}v_{d,8} + \gamma_{54}v_{d,11} + \gamma_{58}v_{d,12} + \gamma_{62}v_{d,14} + \gamma_{64}v_{d,15} +  \equiv_{\omega} 0,\\
p_{(1;(2,3,5))}(S) &\equiv_{\omega} \gamma_{\{9,13,33,35,36,39,47,71\}}v_{d,3}\\ &\qquad + \gamma_{\{1,11,14,17,18,20,29,33,36,39,40,42,43,45,48,50,51,52,71\}}v_{d,5}\\ &\qquad + \gamma_{\{3,13,18,20,35,40,42,43,45,48\}}v_{d,7} + \gamma_{50}v_{d,8}\\ &\qquad+ \gamma_{\{27,29,59,60,63\}}v_{d,10} + \gamma_{59}v_{d,12} + \gamma_{61}v_{d,13} + \gamma_{63}v_{d,14} \equiv_{\omega} 0.
\end{align*}
A direct computation from the above equalities shows
\begin{equation}\label{c65}
\begin{cases}
\gamma_t = 0,\ t \in \{27,\, 49,\, 50,\, 54,\, 58,\, 59,\, 61,\, 62,\, 63,\, 64\},\\ \gamma_{60} = \gamma_{53}=\gamma_{29} = \gamma_{28}.
\end{cases}
\end{equation}
Applying the homomorphisms $p_{(1;(2,4,5))}$, $p_{(1;(3,4,5))}$ to \eqref{ctd612} and using \eqref{c61}-\eqref{c65}, we have
\begin{align*}
p_{(1;(2,4,5))}(S) &\equiv_{\omega} \gamma_{\{9,10,11,12,14,15,28,33,34,36,37,41,42,44,45\}}v_{d,2}\\ &\qquad+ \gamma_{\{1,11,16,17,19,20,28,33,38,39,41,42,46,47,48,52,71\}}v_{d,5}\\ &\qquad + \gamma_{\{2,12,16,17,28,34,38,39,46,47,48,51,52,71\}}v_{d,6}\\ &\qquad + \gamma_{51}v_{d,8} + \gamma_{28}v_{d,9} + \gamma_{28}v_{d,10} + \gamma_{28}v_{d,11} \gamma_{28}v_{d,12} \equiv_{\omega} 0,\\
p_{(1;(3,4,5))}(S) &\equiv_{\omega} \gamma_{\{33,34,35\}}v_{d,1} + \gamma_{\{9,10,11,12,14,15,28,36,37,40,41,42,43,44,45\}}v_{d,2}\\ &\qquad + \gamma_{\{9,13,36,38,39,46,47,71\}}v_{d,3} + \gamma_{\{10,37,38,39,40,41,42,48\}}v_{d,4}\\ &\qquad + \gamma_{\{1,14,16,17,18,19,20,28,43,44,45,46,47,48,51,71\}}v_{d,5}\\ &\qquad + \gamma_{\{2,15,16,17,43,44,45,46,47,48,51,52,71\}}v_{d,6}\\ &\qquad + \gamma_{\{3,18,19,20,43,44,45,46,47,48,71\}}v_{d,7} + \gamma_{52}v_{d,8} \equiv_{\omega} 0.
\end{align*}
By a direct computation using the above equalities we get 
\begin{equation}\label{c66}
\gamma_t = 0 \mbox{ for } t \notin \mathbb J,\ \ 
\gamma_{t} = \gamma_{1} \mbox{ for } t \in \mathbb J,
\end{equation}
where $\mathbb J =  \{$2, 3, 11, 12, 13, 38, 39, 40, 43, 44, 45, 46, 47,  48, 71$\}$. Now, by applying the homomorphism $p_{(2;(3,4,5))}$ to \eqref{ctd612} and using \eqref{c66}, we obtain
\begin{align*}
p_{(2;(3,4,5))}(S) &\equiv_{\omega} \gamma_{1}(v_{d,1} + v_{d,2} +\ldots + v_{d,7} + v_{d,8})  \equiv_{\omega} 0.
\end{align*}
This equality implies $\gamma_1 = 0$, hence $\gamma_t =0$ for all $t, \ 1 \leqslant t \leqslant 75$. The proposition is proved.
\end{proof}

\subsection{Computation of $QP_5^+((4)|(3)|^{d-2}|(1))$}\

\medskip
For $1 \leqslant i < j \leqslant 5$, denote 
$f_{(i,j)} = f_if_{j-1}: P_{3} \stackrel{{\scriptstyle{f_{j-1}}}}{\longrightarrow} P_{4} \stackrel{{\scriptstyle{f_i}}}{\longrightarrow} P_5.$
Here $f_i$ and $f_{j-1}$ are defined by \eqref{ct22}.

Let $d \geqslant 6$. From Proposition \ref{mdd52} we see that for any $x \in B_5((4)|(3)|^{d-2})$ there exists uniquely a pair $(i_x,j_x)$ such that $1 \leqslant i_x < j_x \leqslant 5$ and $\nu_{i_x}(x)<16$, $\nu_{j_x}(x) <16$. The main result of this subsection is the following.

\begin{props}\label{mdd53} For any $d \geqslant 6$, we have
$$B_5((4)|(3)|^{d-2}|(1)) = \left\{xf_{(i_x,j_x)}\left(x_t^{2^{d-1}}\right): x \in B_5((4)|(3)|^{d-2}),\, t = 1,\, 2,\, 3\right\}.$$
 Consequently, 
$$\dim QP_5((4)|(3)|^{d-2}|(1)) = 930,\ \dim (QP_5)_{(2^{d+1}-2)} = 155 + 930 = 1085.$$
\end{props}
Thus, Theorem \ref{thm1} is proved. In the remaining part of the section, we prove Proposition \ref{mdd53}. We need the following. 
\begin{lems}[Ph\'uc \cite{ph}]\label{bda64} Let $(i,j,t,u,v)$ be an arbitrary permutation of $(1,2,3,4,5)$. The following monomials are strictly inadmissible:
	
\smallskip
 {\rm i)} \ $x_i^2x_jx_tx_u,\, i<j<t<u;$.
	
\smallskip
{\rm ii)} \ $x_1^3x_2^5x_3^2x_4^2x_5$,\ $x_i^3x_j^4x_t^3x_u^3$,\, $x_ix_j^2x_t^2x_ux_v^7$,\, $x_ix_j^2x_t^6x_ux_v^3$,\, $x_ix_j^6x_t^2x_ux_v^3$,\, $x_ix_j^2x_t^2x_u^5x_v^3$, $i<j<t<u$.
	
\smallskip
{\rm iii)} $x_i^7x_j^7x_t^8x_u^7,\, t<u$, and 
	
\smallskip
\centerline{\begin{tabular}{lllll} 
$x_1x_2^{6}x_3^{7}x_4^{8}x_5^{7}$& $x_1x_2^{7}x_3^{6}x_4^{8}x_5^{7}$& $x_1x_2^{7}x_3^{10}x_4^{4}x_5^{7}$& $x_1x_2^{7}x_3^{10}x_4^{7}x_5^{4}$& $x_1^{3}x_2^{3}x_3^{4}x_4^{12}x_5^{7} $\cr  $x_1^{3}x_2^{3}x_3^{12}x_4^{4}x_5^{7}$& $x_1^{3}x_2^{3}x_3^{12}x_4^{7}x_5^{4}$& $x_1^{3}x_2^{5}x_3^{6}x_4^{3}x_5^{12}$& $x_1^{3}x_2^{5}x_3^{6}x_4^{11}x_5^{4}$& $x_1^{3}x_2^{5}x_3^{8}x_4^{6}x_5^{7} $\cr  $x_1^{3}x_2^{5}x_3^{8}x_4^{7}x_5^{6}$& $x_1^{3}x_2^{5}x_3^{9}x_4^{6}x_5^{6}$& $x_1^{3}x_2^{5}x_3^{14}x_4^{3}x_5^{4}$& $x_1^{3}x_2^{12}x_3^{3}x_4^{5}x_5^{6}$& $x_1^{3}x_2^{13}x_3^{6}x_4^{3}x_5^{4} $\cr  $x_1^{7}x_2x_3^{6}x_4^{8}x_5^{7}$& $x_1^{7}x_2x_3^{10}x_4^{4}x_5^{7}$& $x_1^{7}x_2x_3^{10}x_4^{7}x_5^{4}$& $x_1^{7}x_2^{8}x_3^{3}x_4^{5}x_5^{6}$& $x_1^{7}x_2^{9}x_3^{2}x_4^{4}x_5^{7} $\cr  $x_1^{7}x_2^{9}x_3^{2}x_4^{7}x_5^{4}$& $x_1^{7}x_2^{9}x_3^{7}x_4^{2}x_5^{4}$& & &\cr  
\end{tabular}}
	
\smallskip
{\rm iii)} $x_i^{15}x_j^{15}x_t^{15}x_u^{16}$,\ $x_i^{}x_j^{14}x_t^{15}x_u^{15}x_v^{16}$ and 

\smallskip
\centerline{\begin{tabular}{llll} 
$x_1^{3}x_2^{7}x_3^{8}x_4^{13}x_5^{30}$& $x_1^{3}x_2^{7}x_3^{8}x_4^{29}x_5^{14}$& $x_1^{3}x_2^{7}x_3^{24}x_4^{13}x_5^{14}$& $x_1^{3}x_2^{13}x_3^{14}x_4^{15}x_5^{16} $\cr  $x_1^{3}x_2^{13}x_3^{15}x_4^{14}x_5^{16}$& $x_1^{3}x_2^{13}x_3^{15}x_4^{18}x_5^{12}$& $x_1^{3}x_2^{15}x_3^{13}x_4^{14}x_5^{16}$& $x_1^{3}x_2^{15}x_3^{13}x_4^{18}x_5^{12} $\cr  $x_1^{3}x_2^{15}x_3^{21}x_4^{10}x_5^{12}$& $x_1^{7}x_2^{3}x_3^{8}x_4^{13}x_5^{30}$& $x_1^{7}x_2^{3}x_3^{8}x_4^{29}x_5^{14}$& $x_1^{7}x_2^{3}x_3^{24}x_4^{13}x_5^{14} $\cr  $x_1^{7}x_2^{7}x_3^{8}x_4^{9}x_5^{30}$& $x_1^{7}x_2^{7}x_3^{9}x_4^{8}x_5^{30}$& $x_1^{7}x_2^{7}x_3^{9}x_4^{14}x_5^{24}$& $x_1^{7}x_2^{7}x_3^{9}x_4^{26}x_5^{12} $\cr  $x_1^{7}x_2^{7}x_3^{9}x_4^{30}x_5^{8}$& $x_1^{7}x_2^{7}x_3^{25}x_4^{10}x_5^{12}$& $x_1^{7}x_2^{11}x_3^{13}x_4^{16}x_5^{14}$& $x_1^{15}x_2^{3}x_3^{13}x_4^{14}x_5^{16} $\cr  $x_1^{15}x_2^{3}x_3^{13}x_4^{18}x_5^{12}$& $x_1^{15}x_2^{3}x_3^{21}x_4^{10}x_5^{12}$& $x_1^{15}x_2^{15}x_3^{17}x_4^{2}x_5^{12}$& $x_1^{15}x_2^{19}x_3^{5}x_4^{10}x_5^{12} $\cr    
\end{tabular}}
\end{lems}
Each monomial in this lemma is of weight vector $(3)|^t|(1)$ with $2 \leqslant t \leqslant 4$.
\begin{lems}\label{bda65} The monomials $x = x_1^{15}x_2^{23}x_3^{27}x_4^{29}x_5^{32}$ is strictly inadmissible.
\end{lems}
\begin{proof} By a direction computation we have
\begin{align*}
x &=x_1^{9}x_2^{15}x_3^{23}x_4^{26}x_5^{53} + x_1^{9}x_2^{15}x_3^{23}x_4^{50}x_5^{29} + x_1^{9}x_2^{23}x_3^{23}x_4^{26}x_5^{45} + x_1^{9}x_2^{23}x_3^{23}x_4^{42}x_5^{29}\\ &\quad + x_1^{9}x_2^{23}x_3^{27}x_4^{29}x_5^{38} + x_1^{9}x_2^{23}x_3^{27}x_4^{30}x_5^{37} + x_1^{9}x_2^{23}x_3^{27}x_4^{37}x_5^{30} + x_1^{9}x_2^{23}x_3^{27}x_4^{38}x_5^{29}\\ &\quad + x_1^{9}x_2^{23}x_3^{30}x_4^{27}x_5^{37} + x_1^{9}x_2^{23}x_3^{30}x_4^{35}x_5^{29} + x_1^{9}x_2^{23}x_3^{35}x_4^{29}x_5^{30} + x_1^{9}x_2^{23}x_3^{35}x_4^{30}x_5^{29}\\ &\quad + x_1^{9}x_2^{23}x_3^{38}x_4^{27}x_5^{29} + x_1^{9}x_2^{23}x_3^{39}x_4^{26}x_5^{29} + x_1^{11}x_2^{15}x_3^{20}x_4^{27}x_5^{53} + x_1^{11}x_2^{15}x_3^{20}x_4^{51}x_5^{29}\\ &\quad + x_1^{11}x_2^{15}x_3^{21}x_4^{26}x_5^{53} + x_1^{11}x_2^{15}x_3^{21}x_4^{50}x_5^{29} + x_1^{11}x_2^{15}x_3^{23}x_4^{28}x_5^{49} + x_1^{11}x_2^{15}x_3^{23}x_4^{48}x_5^{29}\\ &\quad + x_1^{11}x_2^{21}x_3^{27}x_4^{29}x_5^{38} + x_1^{11}x_2^{21}x_3^{27}x_4^{30}x_5^{37} + x_1^{11}x_2^{21}x_3^{27}x_4^{37}x_5^{30} + x_1^{11}x_2^{21}x_3^{27}x_4^{38}x_5^{29}\\ &\quad + x_1^{11}x_2^{21}x_3^{30}x_4^{27}x_5^{37} + x_1^{11}x_2^{21}x_3^{30}x_4^{35}x_5^{29} + x_1^{11}x_2^{21}x_3^{35}x_4^{29}x_5^{30} + x_1^{11}x_2^{21}x_3^{35}x_4^{30}x_5^{29}\\ &\quad + x_1^{11}x_2^{21}x_3^{38}x_4^{27}x_5^{29} + x_1^{11}x_2^{23}x_3^{20}x_4^{27}x_5^{45} + x_1^{11}x_2^{23}x_3^{20}x_4^{43}x_5^{29} + x_1^{11}x_2^{23}x_3^{21}x_4^{26}x_5^{45}\\ &\quad + x_1^{11}x_2^{23}x_3^{21}x_4^{42}x_5^{29} + x_1^{11}x_2^{23}x_3^{23}x_4^{28}x_5^{41} + x_1^{11}x_2^{23}x_3^{23}x_4^{40}x_5^{29} + x_1^{11}x_2^{23}x_3^{36}x_4^{27}x_5^{29}\\ &\quad + x_1^{11}x_2^{23}x_3^{37}x_4^{26}x_5^{29} + x_1^{15}x_2^{15}x_3^{20}x_4^{27}x_5^{49} + x_1^{15}x_2^{15}x_3^{21}x_4^{26}x_5^{49} + x_1^{15}x_2^{17}x_3^{23}x_4^{26}x_5^{45}\\ &\quad + x_1^{15}x_2^{17}x_3^{23}x_4^{42}x_5^{29} + x_1^{15}x_2^{17}x_3^{39}x_4^{26}x_5^{29} + x_1^{15}x_2^{19}x_3^{20}x_4^{27}x_5^{45} + x_1^{15}x_2^{19}x_3^{20}x_4^{43}x_5^{29}\\ &\quad + x_1^{15}x_2^{19}x_3^{21}x_4^{26}x_5^{45} + x_1^{15}x_2^{19}x_3^{21}x_4^{42}x_5^{29} + x_1^{15}x_2^{19}x_3^{27}x_4^{28}x_5^{37} + x_1^{15}x_2^{19}x_3^{27}x_4^{36}x_5^{29}\\ &\quad + x_1^{15}x_2^{19}x_3^{36}x_4^{27}x_5^{29} + x_1^{15}x_2^{19}x_3^{37}x_4^{26}x_5^{29} + x_1^{15}x_2^{21}x_3^{27}x_4^{28}x_5^{35} + x_1^{15}x_2^{21}x_3^{27}x_4^{29}x_5^{34}\\ &\quad + x_1^{15}x_2^{21}x_3^{27}x_4^{30}x_5^{33} + x_1^{15}x_2^{21}x_3^{27}x_4^{33}x_5^{30} + x_1^{15}x_2^{21}x_3^{30}x_4^{27}x_5^{33} + x_1^{15}x_2^{21}x_3^{34}x_4^{27}x_5^{29}\\ &\quad + x_1^{15}x_2^{23}x_3^{24}x_4^{27}x_5^{37} + x_1^{15}x_2^{23}x_3^{24}x_4^{35}x_5^{29} + x_1^{15}x_2^{23}x_3^{25}x_4^{28}x_5^{35} + x_1^{15}x_2^{23}x_3^{25}x_4^{34}x_5^{29}\\ &\quad + x_1^{15}x_2^{23}x_3^{27}x_4^{28}x_5^{33} + Sq^1\big(x_1^{15}x_2^{23}x_3^{29}x_4^{29}x_5^{29}\big) +  Sq^2\big(x_1^{15}x_2^{15}x_3^{23}x_4^{26}x_5^{45}\\ &\quad + x_1^{15}x_2^{15}x_3^{23}x_4^{42}x_5^{29} + x_1^{15}x_2^{15}x_3^{39}x_4^{26}x_5^{29} + x_1^{15}x_2^{23}x_3^{23}x_4^{28}x_5^{35} + x_1^{15}x_2^{23}x_3^{23}x_4^{34}x_5^{29}\\ &\quad + x_1^{15}x_2^{23}x_3^{27}x_4^{29}x_5^{30} + x_1^{15}x_2^{23}x_3^{27}x_4^{30}x_5^{29} + x_1^{15}x_2^{23}x_3^{30}x_4^{27}x_5^{29}\big)\\ &\quad +
Sq^4\big(x_1^{15}x_2^{15}x_3^{20}x_4^{27}x_5^{45} + x_1^{15}x_2^{15}x_3^{20}x_4^{43}x_5^{29} + x_1^{15}x_2^{15}x_3^{21}x_4^{26}x_5^{45}\\ &\quad + x_1^{15}x_2^{15}x_3^{21}x_4^{42}x_5^{29} + x_1^{15}x_2^{15}x_3^{23}x_4^{28}x_5^{41} + x_1^{15}x_2^{15}x_3^{23}x_4^{40}x_5^{29} + x_1^{15}x_2^{15}x_3^{36}x_4^{27}x_5^{29}\\ &\quad + x_1^{15}x_2^{15}x_3^{37}x_4^{26}x_5^{29} + x_1^{15}x_2^{21}x_3^{23}x_4^{28}x_5^{35} + x_1^{15}x_2^{21}x_3^{23}x_4^{34}x_5^{29} + x_1^{15}x_2^{21}x_3^{27}x_4^{29}x_5^{30}\\ &\quad + x_1^{15}x_2^{21}x_3^{27}x_4^{30}x_5^{29} + x_1^{15}x_2^{21}x_3^{30}x_4^{27}x_5^{29} + x_1^{15}x_2^{27}x_3^{23}x_4^{28}x_5^{29}\big)\\ &\quad +  Sq^8\big(x_1^{9}x_2^{15}x_3^{23}x_4^{26}x_5^{45} + x_1^{9}x_2^{15}x_3^{23}x_4^{42}x_5^{29} + x_1^{9}x_2^{15}x_3^{39}x_4^{26}x_5^{29} + x_1^{9}x_2^{23}x_3^{27}x_4^{29}x_5^{30}\\ &\quad + x_1^{9}x_2^{23}x_3^{27}x_4^{30}x_5^{29} + x_1^{9}x_2^{23}x_3^{30}x_4^{27}x_5^{29} + x_1^{11}x_2^{15}x_3^{20}x_4^{27}x_5^{45} + x_1^{11}x_2^{15}x_3^{20}x_4^{43}x_5^{29}\\ &\quad + x_1^{11}x_2^{15}x_3^{21}x_4^{26}x_5^{45} + x_1^{11}x_2^{15}x_3^{21}x_4^{42}x_5^{29} + x_1^{11}x_2^{15}x_3^{23}x_4^{28}x_5^{41} + x_1^{11}x_2^{15}x_3^{23}x_4^{40}x_5^{29}\\ &\quad + x_1^{11}x_2^{15}x_3^{36}x_4^{27}x_5^{29} + x_1^{11}x_2^{15}x_3^{37}x_4^{26}x_5^{29} + x_1^{11}x_2^{21}x_3^{27}x_4^{29}x_5^{30} + x_1^{11}x_2^{21}x_3^{27}x_4^{30}x_5^{29}\\ &\quad + x_1^{11}x_2^{21}x_3^{30}x_4^{27}x_5^{29} + x_1^{23}x_2^{15}x_3^{23}x_4^{28}x_5^{29} + x_1^{23}x_2^{15}x_3^{24}x_4^{27}x_5^{29}\big)\\ &\quad + Sq^{16}\big(x_1^{15}x_2^{15}x_3^{23}x_4^{28}x_5^{29} + x_1^{15}x_2^{15}x_3^{24}x_4^{27}x_5^{29}\big) \ \mbox{mod}(P_5^-((4)|(3)|^4|(1)).
\end{align*}	
Thus, the monomial $x$ is strictly inadmissible.
\end{proof}
\begin{proof}[Proof of Proposition $\ref{mdd53}$] Let $d \geqslant 6$ and let $\tilde x \in P_5((4)|(3)|^{d-2}|(1))$ be an admissible monomial. Then $\tilde x $ is of the form $\tilde x = x(x_r^{2^{d-1}})$ with $1 \leqslant r \leqslant 5$ and $x$ a monomial of weight vector $(4)|(3)|^{d-2}$. From the proof of Proposition \ref{mdd52} we see that if $x \notin B_5\big((4)|(3)|^{d-2}\big)$, then $x$ is strictly inadmissible, hence by Theorem \ref{dlcb1}, $\tilde x $ is strictly inadmissible. This implies a contradiction. 
	
By a direct computation we see that if either $r = \nu_{i_x}(x)$ or $r = \nu_{j_x}(x)$, then there is a monomial $w$ as given in one of Lemmas \ref{bda64} or \ref{bda65} such that $x = zw^{2^{d-u}}$ with $u$ a nonnegative integer, $2\leqslant u \leqslant 6$, and $z$ a monomial of weight vector $(4)|(3)|^{d-u}$. By Theorem \ref{dlcb1}, $x$ is inadmissible. So, $r \ne \nu_{i_x}(x)$, $r \ne \nu_{j_x}(x)$ and $x_r^{2^{d-1}} = f_{(i_x,j_x)}\left(x_t^{2^{d-1}}\right)$ with $t = 1,\, 2,\, 3$. Hence, we obtain
 $$B_5((4)|(3)|^{d-2}|(1))\subset  \left\{xf_{(i_x,j_x)}\big(x_t^{2^{d-1}}\big): x \in B_5\big((4)|(3)|^{d-2}\big),\, t = 1,\, 2,\, 3\right\}.$$
	
Now we prove that the set in the right hand side of the last relation is a minimal set of $\mathcal A$-generators for $P_5(\overline \omega)$ with $\overline\omega =(4)|(3)|^{d-2}|(1)$.

Suppose there is a linear relation
\begin{equation*}
S:= \sum_{x\in B_5^+(\overline\omega);\ t=1,2,3}\gamma_{x,t}\left(xf_{(i_x,j_x)}\big(x_t^{2^{d-1}}\big)\right) \equiv_{\overline\omega} 0,
\end{equation*}
where $\gamma_{x,t} \in \mathbb F_2$. 

Let $B_4(\overline\omega)$ be as in our work \cite{su2} and the homomorphism $p_{(i;I)}:P_5\to P_4$ which is defined by \eqref{ct23} for $k=5$. Note that $|B_4(\overline\omega)| = 45$. From Lemma \ref{bdm}, we see that $p_{(i;I)}$ passes to a homomorphism from $QP_5(\overline\omega)$ to $QP_4(\overline\omega)$. By a direct computation similar to the one in the proof of Proposition \ref{mdd52}, we compute $p_{(i;I)}(S))$ in terms of the monomials in $B_4(\overline\omega)$ (mod($\mathcal A^+P_4 + P_5^-(\overline\omega)$). Based on the relations $p_{(i;I)}(S))\equiv_{\overline\omega} 0$ with all $(i;I) \in \mathcal N_5$ and $I\ne (2,3,4,5)$, we obtain $\gamma_{x,t} = 0$ for all $x\in B_5(\overline\omega)$ and $t=1,\, 2,\, 3.$ The proof is completed.
\end{proof}

\section{Proof of Theorem \ref{thm2}}\label{s4}
\setcounter{equation}{0}

First of all, we determine the weight vectors of the admissible monomials of degree $2^{d+1}-1$ for $d \geqslant 6$.
\begin{lem}\label{bdd62} Let $x$ be an admissible monomial of degree $2^{d+1}-1$ in $P_5$ for $d \geqslant 6$. If $[x] \in \mbox{\rm Ker}\big((\widetilde{Sq}^0_*)_{(5,2^{d}-3)}\big)$, then either $\omega(x) = (1)|^{d+1}$ or $\omega(x) = (3)|(2)|^{d-1}$.
\end{lem}

To prove the lemma, we need the following.
\begin{lem}\label{bda71} Let $(i,j,t,u,v)$ be an arbitrary permutation of $(1,2,3,4,5)$. The following monomials are strictly inadmissible:
	
\smallskip
\ \ {\rm i)} $x_i^2x_j^3x_t^3x_u^3$; $x_i^2x_jx_t^2x_u^3x_v^3,\, i < j$.

\smallskip
\ \ {\rm ii)} $x_ix_j^{3}x_t^{14}x_u^{14}x_v^{15}$; $x_i^{3}x_j^{13}x_t^{2}x_u^{14}x_v^{15}$; $x_i^{3}x_j^{5}x_t^{10}x_u^{14}x_v^{15}$; $x_i^{3}x_j^{13}x_t^{6}x_u^{10}x_v^{15}$; $x_i^{3}x_j^{13}x_t^{14}x_u^{3}x_v^{14}$;  $x_i^{3}x_j^{5}x_t^{14}x_u^{11}x_v^{14},\, t< u$;  $x_i^{3}x_j^{13}x_t^{6}x_u^{11}x_v^{14}\, j< u$; $x_i^{3}x_j^{13}x_t^{7}x_u^{10}x_v^{14}\, j< t$.

\smallskip
{\begin{tabular}{lllll}
{\rm iii)}&$x_ix_j^{7}x_t^{27}x_u^{30}x_v^{30}$& $x_i^{3}x_j^{5}x_t^{27}x_u^{30}x_v^{30}$& $x_i^{3}x_j^{7}x_t^{25}x_u^{30}x_v^{30}$& $x_i^{3}x_j^{7}x_t^{29}x_u^{26}x_v^{30}$\cr  &$x_i^{7}x_j^{11}x_t^{21}x_u^{26}x_v^{30}$& $x_i^{7}x_j^{11}x_t^{29}x_u^{22}x_v^{26}$& $x_i^{7}x_j^{27}x_t^{5}x_u^{26}x_v^{30}$& $x_i^{7}x_j^{27}x_t^{13}x_u^{22}x_v^{26}$\cr  &$x_i^{7}x_j^{27}x_t^{29}x_u^{2}x_v^{30}$& $x_i^{7}x_j^{27}x_t^{29}x_u^{6}x_v^{26}$.& &\cr  
\end{tabular}}
\end{lem}
\begin{proof} Part i) is easy. Part ii) is due to Ph\'uc \cite{ph1}. We prove Part iii) for the monomial $y = x_i^{3}x_j^{5}x_t^{27}x_u^{30}x_v^{30}$. The others can be proved by the similar computations. By a direct computation, we have
\begin{align*}
y &=  Sq^1\big(x_i^{3}x_j^{9}x_t^{23}x_u^{29}x_v^{30} + x_i^{3}x_j^{12}x_t^{23}x_u^{27}x_v^{29} + x_i^{9}x_j^{18}x_t^{15}x_u^{23}x_v^{29}\\ &\quad + x_i^{9}x_j^{20}x_t^{15}x_u^{23}x_v^{27} + x_i^{12}x_j^{17}x_t^{15}x_u^{23}x_v^{27}\big) + Sq^2\big(x_i^{3}x_j^{10}x_t^{23}x_u^{27}x_v^{30}\\ &\quad + x_i^{10}x_j^{18}x_t^{15}x_u^{23}x_v^{27}\big) + Sq^4\big(x_i^{3}x_j^{5}x_t^{23}x_u^{30}x_v^{30} + x_i^{5}x_j^{18}x_t^{15}x_u^{23}x_v^{30}\big)\\ &\quad +  Sq^8\big(x_i^{5}x_j^{10}x_t^{15}x_u^{27}x_v^{30}\big)
\ \mbox{mod}(P_5^-(3,4,3,3,3)).
\end{align*}
Hence, the monomial $y$ is strictly inadmissible.
\end{proof}
\begin{proof}[Proof of Lemma \ref{bdd62}] Let $x$ be an admissible monomial of degree $2^{d+1}-1$ such that $[x] \in \mbox{\rm Ker}\big((\widetilde{Sq}^0_*)_{(5,2^{d}-3)}\big)$. Since $2^{d+1}-1$ is odd, we obtain either $\omega_1(x)=1$ or $\omega_1(x)=3$ or $\omega_1(x)=5$. 
	
If $\omega_1(x) = 5$, then $x = X_{\emptyset} y^2$ with $y$ an admissible monomial of degree $2^{d}-3$ in $P_5$ and $\widetilde{Sq}^0_*)_{(5,2^{d}-3)}([x]) = [y] \ne 0$. This contradicts the the hypothesis that $[x] \in \mbox{\rm Ker}\big((\widetilde{Sq}^0_*)_{(5,2^{d}-3)}\big)$. 

If $\omega_1(x) = 1$, then $x = x_{i} y_1^2$ with $y_1$ an admissible monomial of degree $2^{d}-1$ in $P_5$ and $1 \leqslant i \leqslant 5$. By Ph\'uc \cite{ph21}, we have $\omega(y_1) = (1)|^5$, hence $\omega(x) = (1)|^6$ for $d = 5$. By induction on $d$ we easily obtain $\omega(x) = (1)|^{d+1}$ for any $d \geqslant 5$.

If $\omega_1(x) = 3$, then we get $x = X_{i,j}y_2^2$ with $1 \leqslant i < j \leqslant 5$ and $y_2$ an admissible monomial of degree $2^{d}-2$. From Lemma \ref{bdd5} we have $\omega(y_2) = (4)|(3)|^{d-3}|(1)$ for any $d \geqslant 6$. By a direct computation we can see that there is a monomial $w$ as given in Lemma \ref{bda71} such that $X_{i,j}y_2^2 = wh^{2^r}$ with $2 \leqslant r \leqslant 5$ and $h$ a suitable monomial. By Theorem \ref{dlcb1}, the monomial $X_{i,j}y_2^2$ is inadmissible. Thus, we get $\omega(y_2) = (2)|^{d-1}$ and $\omega(x) = (3)|(2)|^{d-1}$. The lemma is proved.
\end{proof}

By Lemma \ref{bdd62}, we have
\begin{align*}\mbox{\rm Ker}\big((\widetilde{Sq}^0_*)_{(5,2^{d}-3)}\big) &\cong QP_5\big((1)|^{d+1}\big) \oplus QP_5\big((3)|(2)|^{d-1}\big).
\end{align*}
The space $QP_5\big((1)|^{d+1}\big)$ has been determined in \cite{su2}. We have
\begin{prop}[See {\cite[Proposition 4.1]{su2}}]\label{mdd70} For any positive integer $s$, we have
$$B_k((1)|^s) = \big\{x_{i_1}x_{i_2}^2\ldots x_{i_{m-1}}^{2^{m-2}}x_{i_m}^{2^s-2^{m-1}}:\ 1 \leqslant i_1< \ldots <i_m \leqslant k;\,  m \leqslant \min\{s,k\}\big\}.
$$
\end{prop}
This proposition implies $\dim QP_5\big((1)|^{d+1}\big) = 31$ for any $d \geqslant 4$. So, we need only to determine the space $QP_5\big((3)|(2)|^{d-1}\big)$. 

From Kameko \cite{ka} and our work \cite{su2}, we have $$|B_3^+\big((3)|(2)|^{d-1}\big)|=7,\ |B_4^+\big((3)|(2)|^{d-1}\big)|=42.$$ Hence, we obtain $|B_5^0\big((3)|(2)|^{d-1}\big)|= 7\binom 53 + 42\binom 54 = 280$.

\begin{prop}\label{mdd71} Let $d$ be a positive integer. If $d \geqslant 5$, then $QP_5^+\big((3)|(2)|^{d-1}\big)$ is an $\mathbb F_2$-vector space of dimension $185$ with a basis consisting of the classes represented by the  admissible monomials which are determined as in Subsection $\ref{s53}$. 
Consequently, $\dim QP_5\big((3)|(2)|^{d-1}\big) = 465.$
\end{prop}
By combining Propositions \ref{mdd70} and \ref{mdd71} one gets Theorem \ref{thm2}. In the remaining part of the section we prove Proposition \ref{mdd71}. We need some lemmas.

\begin{lem}\label{bda72} Let $(i,j,t,u,v)$ be an arbitrary permutation of $(1,2,3,4,5)$. The following monomials are strictly inadmissible: 
	
\smallskip  
\  \ \ {\rm i)} \ \! $x_i^2x_jx_tx_u^3,\, i<j<t$; $x_i^2x_jx_tx_ux_v^2,\, i<j<t<u$; $x_1x_2^2x_3^2x_4x_5$.

\smallskip  
\  \  {\rm ii)} \  $x_i^3x_j^{12}x_tx_u^{15}$,\, $x_i^3x_j^{4}x_t^{9}x_u^{15}$,\, $x_i^3x_j^{5}x_t^{8}x_u^{15}$,\, $i<j<t$; $x_i^3x_j^{5}x_t^{9}x_u^{14}$,\, $x_i^3x_j^{5}x_t^{14}x_u^{9}$,\, $x_i^3x_j^{12}x_t^{3}x_u^{13}$,\, $x_i^3x_j^{7}x_t^{13}x_u^{8}$,\, $x_i^7x_j^{3}x_t^{13}x_u^{8}$,\, $x_i^7x_j^{11}x_t^{12}x_u^{}$,\, $x_i^{3}x_j^{7}x_t^{12}x_u^{9}$,\, $x_i^{7}x_j^{3}x_t^{12}x_u^{9}$,\linebreak $x_i^{7}x_j^{11}x_t^{4}x_u^{9}$,\, $x_i^{7}x_j^{11}x_t^{5}x_u^{8}$,\, $i<j<t<u$.

\smallskip  
\  \  {\rm iii)} \  $x_i^3x_j^{4}x_tx_u^{8}x_v^{15}$,\, $x_i^3x_j^{4}x_t^8x_ux_v^{15}$,\, $i<j<t<u$, and

\smallskip
{\begin{tabular}{llllll} 
$x_1x_2^{6}x_3^{3}x_4^{12}x_5^{9}$& $x_1x_2^{6}x_3^{3}x_4^{13}x_5^{8}$& $x_1x_2^{6}x_3^{11}x_4^{4}x_5^{9}$& $x_1x_2^{6}x_3^{11}x_4^{5}x_5^{8}$& $x_1x_2^{6}x_3^{11}x_4^{12}x_5 $\cr  $x_1x_2^{7}x_3^{10}x_4^{12}x_5$& $x_1x_2^{14}x_3^{3}x_4^{4}x_5^{9}$& $x_1x_2^{14}x_3^{3}x_4^{5}x_5^{8}$& $x_1x_2^{14}x_3^{3}x_4^{12}x_5$& $x_1^{3}x_2^{3}x_3^{12}x_4^{12}x_5 $\cr  $x_1^{3}x_2^{4}x_3x_4^{9}x_5^{14}$& $x_1^{3}x_2^{4}x_3x_4^{14}x_5^{9}$& $x_1^{3}x_2^{4}x_3^{3}x_4^{8}x_5^{13}$& $x_1^{3}x_2^{4}x_3^{3}x_4^{12}x_5^{9}$& $x_1^{3}x_2^{4}x_3^{3}x_4^{13}x_5^{8} $\cr  $x_1^{3}x_2^{4}x_3^{7}x_4^{8}x_5^{9}$& $x_1^{3}x_2^{4}x_3^{7}x_4^{9}x_5^{8}$& $x_1^{3}x_2^{4}x_3^{8}x_4^{3}x_5^{13}$& $x_1^{3}x_2^{4}x_3^{8}x_4^{7}x_5^{9}$& $x_1^{3}x_2^{4}x_3^{9}x_4x_5^{14} $\cr  $x_1^{3}x_2^{4}x_3^{9}x_4^{2}x_5^{13}$& $x_1^{3}x_2^{4}x_3^{9}x_4^{3}x_5^{12}$& $x_1^{3}x_2^{4}x_3^{9}x_4^{6}x_5^{9}$& $x_1^{3}x_2^{4}x_3^{9}x_4^{7}x_5^{8}$& $x_1^{3}x_2^{4}x_3^{9}x_4^{14}x_5 $\cr  $x_1^{3}x_2^{4}x_3^{11}x_4^{4}x_5^{9}$& $x_1^{3}x_2^{4}x_3^{11}x_4^{5}x_5^{8}$& $x_1^{3}x_2^{4}x_3^{11}x_4^{12}x_5$& $x_1^{3}x_2^{5}x_3x_4^{8}x_5^{14}$& $x_1^{3}x_2^{5}x_3x_4^{14}x_5^{8} $\cr  $x_1^{3}x_2^{5}x_3^{8}x_4x_5^{14}$& $x_1^{3}x_2^{5}x_3^{8}x_4^{2}x_5^{13}$& $x_1^{3}x_2^{5}x_3^{8}x_4^{3}x_5^{12}$& $x_1^{3}x_2^{5}x_3^{8}x_4^{6}x_5^{9}$& $x_1^{3}x_2^{5}x_3^{8}x_4^{7}x_5^{8} $\cr  $x_1^{3}x_2^{5}x_3^{8}x_4^{14}x_5$& $x_1^{3}x_2^{5}x_3^{9}x_4^{2}x_5^{12}$& $x_1^{3}x_2^{5}x_3^{9}x_4^{6}x_5^{8}$& $x_1^{3}x_2^{5}x_3^{14}x_4x_5^{8}$& $x_1^{3}x_2^{5}x_3^{14}x_4^{8}x_5 $\cr  $x_1^{3}x_2^{7}x_3^{8}x_4^{4}x_5^{9}$& $x_1^{3}x_2^{7}x_3^{8}x_4^{5}x_5^{8}$& $x_1^{3}x_2^{7}x_3^{8}x_4^{12}x_5$& $x_1^{3}x_2^{7}x_3^{12}x_4x_5^{8}$& $x_1^{3}x_2^{7}x_3^{12}x_4^{8}x_5 $\cr  $x_1^{3}x_2^{12}x_3x_4x_5^{14}$& $x_1^{3}x_2^{12}x_3x_4^{2}x_5^{13}$& $x_1^{3}x_2^{12}x_3x_4^{3}x_5^{12}$& $x_1^{3}x_2^{12}x_3x_4^{14}x_5$& $x_1^{3}x_2^{12}x_3^{3}x_4x_5^{12} $\cr  $x_1^{3}x_2^{12}x_3^{3}x_4^{4}x_5^{9}$& $x_1^{3}x_2^{12}x_3^{3}x_4^{5}x_5^{8}$& $x_1^{3}x_2^{12}x_3^{3}x_4^{12}x_5$& $x_1^{7}x_2x_3^{10}x_4^{12}x_5$& $x_1^{7}x_2^{3}x_3^{8}x_4^{4}x_5^{9} $\cr  $x_1^{7}x_2^{3}x_3^{8}x_4^{5}x_5^{8}$& $x_1^{7}x_2^{3}x_3^{8}x_4^{12}x_5$& $x_1^{7}x_2^{3}x_3^{12}x_4x_5^{8}$& $x_1^{7}x_2^{3}x_3^{12}x_4^{8}x_5$& $x_1^{7}x_2^{7}x_3^{8}x_4^{8}x_5 $\cr  $x_1^{7}x_2^{8}x_3^{3}x_4^{4}x_5^{9}$& $x_1^{7}x_2^{8}x_3^{3}x_4^{5}x_5^{8}$& $x_1^{7}x_2^{11}x_3^{4}x_4x_5^{8}$& $x_1^{7}x_2^{11}x_3^{4}x_4^{8}x_5$& \cr     
\end{tabular}}
\end{lem}
Part i) of this lemma is due to T\'in \cite{tin14}, Part ii) follows from our work \cite{su2} and Part iii) is due to Ph\'uc \cite{ph21}.

\begin{lem}\label{bda73} Let $(i,j,t,u,v)$ be an arbitrary permutation of $(1,2,3,4,5)$. The following monomials are strictly inadmissible: 
	
\smallskip	
\  \  {\rm i)} \ \! $x_i^3x_j^{7}x_t^{24}x_u^{29}$,\, $x_i^7x_j^{3}x_t^{24}x_u^{29}$,\, $i<j<t<u$.
	
\smallskip	
\  \  {\rm ii)} $x_i^{3}x_j^{4}x_t^{11}x_u^{29}x_v^{16}$,\, $x_i^{3}x_j^{7}x_t^{4}x_u^{25}x_v^{24}$,\, $x_i^{3}x_j^{7}x_t^{8}x_u^{29}x_v^{16}$,\, $x_i^{3}x_j^{7}x_t^{9}x_u^{28}x_v^{16}$,\, $x_i^{7}x_j^{3}x_t^{4}x_u^{25}x_v^{24}$,\, $x_i^{7}x_j^{3}x_t^{8}x_u^{29}x_v^{16}$,\, $x_i^{7}x_j^{3}x_t^{9}x_u^{28}x_v^{16}$,\, $i<j<t<u$, and
	
\smallskip
\centerline{\begin{tabular}{llll}
$x_1x_2^{2}x_3^{7}x_4^{28}x_5^{25}$& $x_1x_2^{6}x_3^{3}x_4^{24}x_5^{29}$& $x_1x_2^{7}x_3^{2}x_4^{28}x_5^{25}$& $x_1x_2^{7}x_3^{10}x_4^{20}x_5^{25} $\cr  $x_1x_2^{7}x_3^{10}x_4^{21}x_5^{24}$& $x_1x_2^{7}x_3^{11}x_4^{20}x_5^{24}$& $x_1x_2^{7}x_3^{26}x_4^{4}x_5^{25}$& $x_1x_2^{7}x_3^{26}x_4^{5}x_5^{24} $\cr  $x_1^{3}x_2^{3}x_3^{4}x_4^{28}x_5^{25}$& $x_1^{3}x_2^{3}x_3^{12}x_4^{20}x_5^{25}$& $x_1^{3}x_2^{3}x_3^{12}x_4^{21}x_5^{24}$& $x_1^{3}x_2^{3}x_3^{13}x_4^{20}x_5^{24} $\cr  $x_1^{3}x_2^{3}x_3^{28}x_4^{4}x_5^{25}$& $x_1^{3}x_2^{3}x_3^{28}x_4^{5}x_5^{24}$& $x_1^{3}x_2^{4}x_3^{11}x_4^{17}x_5^{28}$& $x_1^{3}x_2^{5}x_3^{6}x_4^{24}x_5^{25} $\cr  $x_1^{3}x_2^{5}x_3^{6}x_4^{25}x_5^{24}$& $x_1^{3}x_2^{5}x_3^{7}x_4^{24}x_5^{24}$& $x_1^{3}x_2^{5}x_3^{10}x_4^{16}x_5^{29}$& $x_1^{3}x_2^{5}x_3^{10}x_4^{17}x_5^{28} $\cr  $x_1^{3}x_2^{5}x_3^{10}x_4^{28}x_5^{17}$& $x_1^{3}x_2^{5}x_3^{10}x_4^{29}x_5^{16}$& $x_1^{3}x_2^{5}x_3^{11}x_4^{16}x_5^{28}$& $x_1^{3}x_2^{5}x_3^{11}x_4^{28}x_5^{16} $\cr  $x_1^{3}x_2^{7}x_3^{5}x_4^{24}x_5^{24}$& $x_1^{3}x_2^{7}x_3^{8}x_4^{17}x_5^{28}$& $x_1^{3}x_2^{7}x_3^{24}x_4x_5^{28}$& $x_1^{7}x_2x_3^{2}x_4^{28}x_5^{25} $\cr  $x_1^{7}x_2x_3^{10}x_4^{20}x_5^{25}$& $x_1^{7}x_2x_3^{10}x_4^{21}x_5^{24}$& $x_1^{7}x_2x_3^{11}x_4^{20}x_5^{24}$& $x_1^{7}x_2x_3^{26}x_4^{4}x_5^{25} $\cr  $x_1^{7}x_2x_3^{26}x_4^{5}x_5^{24}$& $x_1^{7}x_2^{3}x_3^{5}x_4^{24}x_5^{24}$& $x_1^{7}x_2^{3}x_3^{8}x_4^{17}x_5^{28}$& $x_1^{7}x_2^{3}x_3^{24}x_4x_5^{28} $\cr  $x_1^{7}x_2^{11}x_3x_4^{20}x_5^{24}$& & & \cr 
\end{tabular}}
\end{lem}
\begin{proof} Part i) follows from \cite{su2}. Let $x = x_1^{3}x_2^{4}x_3^{11}x_4^{29}x_5^{16}$ be a monomial in Part ii). We have
\begin{align*}
x &= x_1^{2}x_2x_3^{11}x_4^{29}x_5^{20} + x_1^{2}x_2x_3^{13}x_4^{29}x_5^{18} + x_1^{2}x_2^{4}x_3^{17}x_4^{27}x_5^{13} + x_1^{3}x_2x_3^{9}x_4^{30}x_5^{20}\\ &\quad + x_1^{3}x_2x_3^{12}x_4^{29}x_5^{18} + x_1^{3}x_2^{4}x_3^{8}x_4^{27}x_5^{21} + Sq^1\big(x_1^{3}x_2x_3^{11}x_4^{29}x_5^{18} + x_1^{3}x_2^{4}x_3^{11}x_4^{27}x_5^{17}\\ &\quad + x_1^{3}x_2^{8}x_3^{11}x_4^{27}x_5^{13}\big) + Sq^2\big(x_1^{2}x_2x_3^{11}x_4^{29}x_5^{18} + x_1^{2}x_2^{4}x_3^{11}x_4^{27}x_5^{17} + x_1^{2}x_2^{8}x_3^{11}x_4^{27}x_5^{13}\\ &\quad + x_1^{3}x_2^{2}x_3^{11}x_4^{17}x_5^{28} + x_1^{3}x_2^{2}x_3^{13}x_4^{17}x_5^{26} + x_1^{3}x_2^{4}x_3^{11}x_4^{17}x_5^{26} + x_1^{5}x_2x_3^{7}x_4^{30}x_5^{18}\\ &\quad + x_1^{5}x_2^{2}x_3^{11}x_4^{17}x_5^{26} + x_1^{5}x_2^{2}x_3^{11}x_4^{27}x_5^{16} + x_1^{5}x_2^{2}x_3^{11}x_4^{29}x_5^{14} + x_1^{5}x_2^{4}x_3^{11}x_4^{27}x_5^{14}\big)\\ &\quad + Sq^4\big(x_1^{2}x_2^{4}x_3^{11}x_4^{29}x_5^{13} + x_1^{2}x_2^{4}x_3^{13}x_4^{27}x_5^{13} + x_1^{3}x_2x_3^{7}x_4^{30}x_5^{18} + x_1^{3}x_2^{2}x_3^{11}x_4^{17}x_5^{26}\\ &\quad + x_1^{3}x_2^{2}x_3^{11}x_4^{27}x_5^{16} + x_1^{3}x_2^{2}x_3^{11}x_4^{29}x_5^{14} + x_1^{3}x_2^{4}x_3^{11}x_4^{27}x_5^{14} + x_1^{3}x_2^{4}x_3^{11}x_4^{28}x_5^{13}\\ &\quad + x_1^{3}x_2^{4}x_3^{12}x_4^{27}x_5^{13}\big) +  Sq^8\big(x_1^{3}x_2x_3^{7}x_4^{34}x_5^{10} + x_1^{3}x_2^{2}x_3^{19}x_4^{17}x_5^{14} + x_1^{3}x_2^{4}x_3^{8}x_4^{27}x_5^{13}\big)\\ &\quad +  Sq^{16}\big(x_1^{3}x_2^{2}x_3^{11}x_4^{33}x_5^{14} \big) \ \mbox{mod}(P_5^-((3)(2)|^4).
\end{align*}
Hence, the monomial $x = x_1^{3}x_2^{4}x_3^{11}x_4^{29}x_5^{16}$ is strictly inadmissible. Let $y$ be the monomial $x_1x_2^{7}x_3^{10}x_4^{20}x_5^{25}$. A direct computation shows
\begin{align*}
 y &= x_1x_2^{4}x_3^{3}x_4^{25}x_5^{30} + x_1x_2^{4}x_3^{3}x_4^{28}x_5^{27} + x_1x_2^{4}x_3^{5}x_4^{26}x_5^{27} + x_1x_2^{4}x_3^{6}x_4^{25}x_5^{27}\\ &\quad + x_1x_2^{4}x_3^{9}x_4^{22}x_5^{27} + x_1x_2^{4}x_3^{10}x_4^{21}x_5^{27} + x_1x_2^{5}x_3^{2}x_4^{25}x_5^{30} + x_1x_2^{5}x_3^{3}x_4^{24}x_5^{30}\\ &\quad + x_1x_2^{5}x_3^{10}x_4^{18}x_5^{29} + x_1x_2^{5}x_3^{10}x_4^{19}x_5^{28} + x_1x_2^{5}x_3^{10}x_4^{20}x_5^{27} + x_1x_2^{5}x_3^{10}x_4^{24}x_5^{23}\\ &\quad + x_1x_2^{6}x_3^{3}x_4^{25}x_5^{28} + x_1x_2^{6}x_3^{3}x_4^{28}x_5^{25} + x_1x_2^{6}x_3^{4}x_4^{25}x_5^{27} + x_1x_2^{6}x_3^{10}x_4^{21}x_5^{27}\\ &\quad + x_1x_2^{7}x_3^{3}x_4^{24}x_5^{28} + x_1x_2^{7}x_3^{4}x_4^{24}x_5^{27} + x_1x_2^{7}x_3^{8}x_4^{19}x_5^{28} + x_1x_2^{7}x_3^{10}x_4^{16}x_5^{29}\\ &\quad + Sq^1\big(x_1x_2^{7}x_3^{3}x_4^{22}x_5^{29} + x_1x_2^{7}x_3^{3}x_4^{24}x_5^{27} + x_1x_2^{7}x_3^{4}x_4^{21}x_5^{29} + x_1x_2^{7}x_3^{5}x_4^{21}x_5^{28}\\ &\quad + x_1x_2^{7}x_3^{5}x_4^{22}x_5^{27} + x_1x_2^{8}x_3^{3}x_4^{21}x_5^{29} + x_1x_2^{8}x_3^{5}x_4^{21}x_5^{27} + x_1x_2^{9}x_3^{3}x_4^{21}x_5^{28}\\ &\quad + x_1x_2^{9}x_3^{3}x_4^{22}x_5^{27} + x_1x_2^{9}x_3^{4}x_4^{21}x_5^{27} + x_1^{4}x_2^{7}x_3^{3}x_4^{21}x_5^{27}\big) +  Sq^2\big(x_1x_2^{6}x_3^{5}x_4^{26}x_5^{23}\\ &\quad + x_1x_2^{7}x_3^{2}x_4^{21}x_5^{30} + x_1x_2^{7}x_3^{5}x_4^{14}x_5^{34} + x_1x_2^{7}x_3^{5}x_4^{18}x_5^{30} + x_1x_2^{7}x_3^{6}x_4^{19}x_5^{28}\\ &\quad + x_1x_2^{7}x_3^{12}x_4^{14}x_5^{27} + x_1x_2^{7}x_3^{12}x_4^{18}x_5^{23} + x_1^{2}x_2^{7}x_3^{3}x_4^{21}x_5^{28} + x_1^{2}x_2^{7}x_3^{3}x_4^{22}x_5^{27}\\ &\quad + x_1^{2}x_2^{7}x_3^{4}x_4^{21}x_5^{27} + x_1^{2}x_2^{8}x_3^{3}x_4^{21}x_5^{27}\big) + Sq^4\big(x_1x_2^{4}x_3^{3}x_4^{21}x_5^{30} + x_1x_2^{4}x_3^{3}x_4^{28}x_5^{23}\\ &\quad + x_1x_2^{4}x_3^{5}x_4^{22}x_5^{27} + x_1x_2^{4}x_3^{6}x_4^{21}x_5^{27} + x_1x_2^{5}x_3^{2}x_4^{21}x_5^{30} + x_1x_2^{5}x_3^{3}x_4^{20}x_5^{30}\\ &\quad + x_1x_2^{5}x_3^{6}x_4^{19}x_5^{28} + x_1x_2^{5}x_3^{10}x_4^{14}x_5^{29} + x_1x_2^{5}x_3^{10}x_4^{20}x_5^{23} + x_1x_2^{6}x_3^{3}x_4^{21}x_5^{28}\\ &\quad + x_1x_2^{6}x_3^{3}x_4^{22}x_5^{27} + x_1x_2^{6}x_3^{4}x_4^{21}x_5^{27} + x_1x_2^{10}x_3^{3}x_4^{22}x_5^{23} + x_1x_2^{11}x_3^{3}x_4^{14}x_5^{30}\\ &\quad + x_1x_2^{11}x_3^{6}x_4^{14}x_5^{27} + x_1x_2^{11}x_3^{6}x_4^{18}x_5^{23}\big) + Sq^8\big(x_1x_2^{6}x_3^{3}x_4^{22}x_5^{23} + x_1x_2^{7}x_3^{3}x_4^{10}x_5^{34}\\ &\quad + x_1x_2^{7}x_3^{3}x_4^{14}x_5^{30} + x_1x_2^{7}x_3^{3}x_4^{22}x_5^{22} + x_1x_2^{7}x_3^{6}x_4^{14}x_5^{27} + x_1x_2^{7}x_3^{6}x_4^{18}x_5^{23}\\ &\quad + x_1x_2^{11}x_3^{3}x_4^{18}x_5^{22}\big) +  Sq^{16}\big(x_1x_2^{7}x_3^{3}x_4^{14}x_5^{22}\big) \ \mbox{mod}(P_5^-((3)(2)|^4).
\end{align*}
This implies $y$ is strictly inadmissible. Let $z = x_1x_2^{7}x_3^{26}x_4^{4}x_5^{25}$. We have
\begin{align*}
z &= x_1x_2^{4}x_3^{23}x_4^{8}x_5^{27} + x_1x_2^{4}x_3^{27}x_4x_5^{30} + x_1x_2^{4}x_3^{27}x_4^{4}x_5^{27} + x_1x_2^{4}x_3^{29}x_4^{2}x_5^{27}\\ &\quad + x_1x_2^{5}x_3^{25}x_4^{2}x_5^{30} + x_1x_2^{5}x_3^{26}x_4^{2}x_5^{29} + x_1x_2^{5}x_3^{26}x_4^{4}x_5^{27} + x_1x_2^{5}x_3^{26}x_4^{8}x_5^{23}\\ &\quad + x_1x_2^{5}x_3^{27}x_4^{2}x_5^{28} + x_1x_2^{6}x_3^{27}x_4x_5^{28} + x_1x_2^{6}x_3^{29}x_4^{2}x_5^{25} + x_1x_2^{7}x_3^{24}x_4x_5^{30}\\ &\quad + x_1x_2^{7}x_3^{24}x_4^{2}x_5^{29} + x_1x_2^{7}x_3^{24}x_4^{4}x_5^{27} + x_1x_2^{7}x_3^{25}x_4^{2}x_5^{28} + x_1x_2^{7}x_3^{26}x_4x_5^{28}\\ &\quad + Sq^1\big(x_1x_2^{7}x_3^{23}x_4^{2}x_5^{29} + x_1x_2^{7}x_3^{23}x_4^{4}x_5^{27} + x_1x_2^{7}x_3^{24}x_4x_5^{29} + x_1x_2^{7}x_3^{25}x_4x_5^{28}\\ &\quad + x_1x_2^{7}x_3^{25}x_4^{2}x_5^{27} + x_1x_2^{8}x_3^{23}x_4x_5^{29} + x_1x_2^{9}x_3^{23}x_4x_5^{28} + x_1x_2^{9}x_3^{23}x_4^{2}x_5^{27}\\ &\quad + x_1^{4}x_2^{7}x_3^{23}x_4x_5^{27}\big) + Sq^2\big(x_1x_2^{6}x_3^{27}x_4^{4}x_5^{23} + x_1x_2^{7}x_3^{15}x_4^{4}x_5^{34} + x_1x_2^{7}x_3^{19}x_4^{4}x_5^{30}\\ &\quad + x_1x_2^{7}x_3^{22}x_4^{2}x_5^{29} + x_1x_2^{7}x_3^{23}x_4^{2}x_5^{28} + x_1x_2^{7}x_3^{24}x_4^{2}x_5^{27} + x_1x_2^{7}x_3^{28}x_4^{2}x_5^{23}\\ &\quad + x_1x_2^{9}x_3^{22}x_4^{2}x_5^{27} + x_1^{2}x_2^{7}x_3^{23}x_4x_5^{28} + x_1^{2}x_2^{7}x_3^{23}x_4^{2}x_5^{27} + x_1^{2}x_2^{7}x_3^{24}x_4x_5^{27}\\ &\quad + x_1^{2}x_2^{8}x_3^{23}x_4x_5^{27}\big) + Sq^4\big(x_1x_2^{4}x_3^{23}x_4x_5^{30} + x_1x_2^{4}x_3^{23}x_4^{4}x_5^{27} + x_1x_2^{4}x_3^{29}x_4^{2}x_5^{23}\\ &\quad + x_1x_2^{5}x_3^{21}x_4^{2}x_5^{30} + x_1x_2^{5}x_3^{22}x_4^{2}x_5^{29} + x_1x_2^{5}x_3^{23}x_4^{2}x_5^{28} + x_1x_2^{5}x_3^{26}x_4^{4}x_5^{23}\\ &\quad + x_1x_2^{6}x_3^{23}x_4x_5^{28} + x_1x_2^{6}x_3^{23}x_4^{2}x_5^{27} + x_1x_2^{10}x_3^{23}x_4^{2}x_5^{23} + x_1x_2^{11}x_3^{11}x_4^{2}x_5^{34}\\ &\quad + x_1x_2^{11}x_3^{15}x_4^{2}x_5^{30} + x_1x_2^{11}x_3^{22}x_4^{2}x_5^{23}\big) + Sq^8\big(x_1x_2^{6}x_3^{23}x_4^{2}x_5^{23} + x_1x_2^{7}x_3^{11}x_4^{2}x_5^{34}\\ &\quad + x_1x_2^{7}x_3^{15}x_4^{2}x_5^{30} + x_1x_2^{7}x_3^{22}x_4^{2}x_5^{23} + x_1x_2^{7}x_3^{23}x_4^{2}x_5^{22} + x_1x_2^{11}x_3^{19}x_4^{2}x_5^{22}\big)\\ &\quad +  Sq^{16}\big(x_1x_2^{7}x_3^{15}x_4^{2}x_5^{22}\big) \ \mbox{mod}(P_5^-((3)(2)|^4).
\end{align*}
This equality implies $z$ is strictly inadmissible. Now, let $u = x_1x_2^{7}x_3^{26}x_4^{5}x_5^{24}$. We have 
\begin{align*}
u &= x_1x_2^{4}x_3^{3}x_4^{27}x_5^{28}  + x_1x_2^{4}x_3^{3}x_4^{29}x_5^{26}  + x_1x_2^{4}x_3^{6}x_4^{27}x_5^{25}  + x_1x_2^{4}x_3^{7}x_4^{26}x_5^{25}\\ &\quad  + x_1x_2^{4}x_3^{10}x_4^{23}x_5^{25}  + x_1x_2^{4}x_3^{11}x_4^{19}x_5^{28}  + x_1x_2^{4}x_3^{11}x_4^{22}x_5^{25}  + x_1x_2^{4}x_3^{13}x_4^{19}x_5^{26}\\ &\quad  + x_1x_2^{4}x_3^{23}x_4^{10}x_5^{25}  + x_1x_2^{4}x_3^{27}x_4^{3}x_5^{28}  + x_1x_2^{4}x_3^{27}x_4^{6}x_5^{25}  + x_1x_2^{4}x_3^{29}x_4^{3}x_5^{26} \\ &\quad + x_1x_2^{5}x_3^{2}x_4^{27}x_5^{28}  + x_1x_2^{5}x_3^{3}x_4^{26}x_5^{28}  + x_1x_2^{5}x_3^{3}x_4^{28}x_5^{26}  + x_1x_2^{5}x_3^{11}x_4^{18}x_5^{28} \\ &\quad + x_1x_2^{5}x_3^{26}x_4^{3}x_5^{28}  + x_1x_2^{5}x_3^{26}x_4^{5}x_5^{26}  + x_1x_2^{5}x_3^{26}x_4^{9}x_5^{22}  + x_1x_2^{5}x_3^{27}x_4^{2}x_5^{28}\\ &\quad  + x_1x_2^{6}x_3^{3}x_4^{29}x_5^{24}  + x_1x_2^{6}x_3^{4}x_4^{27}x_5^{25}  + x_1x_2^{6}x_3^{7}x_4^{24}x_5^{25}  + x_1x_2^{6}x_3^{8}x_4^{23}x_5^{25} \\ &\quad + x_1x_2^{6}x_3^{11}x_4^{20}x_5^{25}  + x_1x_2^{6}x_3^{13}x_4^{19}x_5^{24}  + x_1x_2^{6}x_3^{23}x_4^{8}x_5^{25}  + x_1x_2^{6}x_3^{27}x_4^{4}x_5^{25}\\ &\quad  + x_1x_2^{6}x_3^{29}x_4^{3}x_5^{24}  + x_1x_2^{7}x_3^{2}x_4^{25}x_5^{28}  + x_1x_2^{7}x_3^{3}x_4^{28}x_5^{24}  + x_1x_2^{7}x_3^{4}x_4^{26}x_5^{25}\\ &\quad  + x_1x_2^{7}x_3^{6}x_4^{24}x_5^{25}  + x_1x_2^{7}x_3^{8}x_4^{19}x_5^{28}  + x_1x_2^{7}x_3^{8}x_4^{22}x_5^{25}  + x_1x_2^{7}x_3^{9}x_4^{18}x_5^{28}\\ &\quad  + x_1x_2^{7}x_3^{10}x_4^{17}x_5^{28}  + x_1x_2^{7}x_3^{10}x_4^{20}x_5^{25}  + x_1x_2^{7}x_3^{11}x_4^{16}x_5^{28}  + x_1x_2^{7}x_3^{17}x_4^{10}x_5^{28} \\ &\quad + x_1x_2^{7}x_3^{19}x_4^{8}x_5^{28}  + x_1x_2^{7}x_3^{24}x_4^{6}x_5^{25}  + x_1x_2^{7}x_3^{25}x_4^{2}x_5^{28}  + x_1x_2^{7}x_3^{26}x_4^{4}x_5^{25} \\ &\quad +  Sq^1\big(x_1x_2^{7}x_3^{3}x_4^{23}x_5^{28}  + x_1x_2^{7}x_3^{3}x_4^{25}x_5^{26}  + x_1x_2^{7}x_3^{4}x_4^{25}x_5^{25}  + x_1x_2^{7}x_3^{5}x_4^{23}x_5^{26} \\ &\quad + x_1x_2^{7}x_3^{5}x_4^{24}x_5^{25}  + x_1x_2^{7}x_3^{7}x_4^{19}x_5^{28}  + x_1x_2^{7}x_3^{7}x_4^{21}x_5^{26}  + x_1x_2^{7}x_3^{8}x_4^{21}x_5^{25} \\ &\quad + x_1x_2^{7}x_3^{9}x_4^{19}x_5^{26}  + x_1x_2^{7}x_3^{9}x_4^{20}x_5^{25}  + x_1x_2^{7}x_3^{23}x_4^{3}x_5^{28}  + x_1x_2^{7}x_3^{23}x_4^{5}x_5^{26} \\ &\quad + x_1x_2^{7}x_3^{24}x_4^{5}x_5^{25}  + x_1x_2^{7}x_3^{25}x_4^{3}x_5^{26}  + x_1x_2^{7}x_3^{25}x_4^{4}x_5^{25}  + x_1x_2^{8}x_3^{5}x_4^{23}x_5^{25}\\ &\quad  + x_1x_2^{8}x_3^{7}x_4^{21}x_5^{25}  + x_1x_2^{8}x_3^{23}x_4^{5}x_5^{25}  + x_1x_2^{9}x_3^{3}x_4^{23}x_5^{26}  + x_1x_2^{9}x_3^{4}x_4^{23}x_5^{25} \\ &\quad + x_1x_2^{9}x_3^{7}x_4^{19}x_5^{26}  + x_1x_2^{9}x_3^{7}x_4^{20}x_5^{25}  + x_1x_2^{9}x_3^{23}x_4^{3}x_5^{26}  + x_1x_2^{9}x_3^{23}x_4^{4}x_5^{25}\\ &\quad  + x_1^{4}x_2^{7}x_3^{3}x_4^{23}x_5^{25}  + x_1^{4}x_2^{7}x_3^{7}x_4^{19}x_5^{25}  + x_1^{4}x_2^{7}x_3^{23}x_4^{3}x_5^{25}\big)  + Sq^2\big(x_1x_2^{6}x_3^{5}x_4^{27}x_5^{22}\\ &\quad  + x_1x_2^{6}x_3^{11}x_4^{21}x_5^{22}  + x_1x_2^{6}x_3^{27}x_4^{5}x_5^{22}  + x_1x_2^{7}x_3^{2}x_4^{23}x_5^{28}  + x_1x_2^{7}x_3^{5}x_4^{22}x_5^{26}\\ &\quad  + x_1x_2^{7}x_3^{5}x_4^{26}x_5^{22}  + x_1x_2^{7}x_3^{6}x_4^{19}x_5^{28}  + x_1x_2^{7}x_3^{8}x_4^{19}x_5^{26}  + x_1x_2^{7}x_3^{9}x_4^{18}x_5^{26} \\ &\quad + x_1x_2^{7}x_3^{12}x_4^{15}x_5^{26}  + x_1x_2^{7}x_3^{13}x_4^{14}x_5^{26}  + x_1x_2^{7}x_3^{15}x_4^{12}x_5^{26}  + x_1x_2^{7}x_3^{19}x_4^{6}x_5^{28} \\ &\quad + x_1x_2^{7}x_3^{19}x_4^{8}x_5^{26}  + x_1x_2^{7}x_3^{22}x_4^{3}x_5^{28}  + x_1x_2^{7}x_3^{23}x_4^{2}x_5^{28}  + x_1x_2^{7}x_3^{24}x_4^{3}x_5^{26}\\ &\quad  + x_1x_2^{7}x_3^{28}x_4^{3}x_5^{22}  + x_1x_2^{9}x_3^{6}x_4^{19}x_5^{26}  + x_1x_2^{9}x_3^{7}x_4^{18}x_5^{26}  + x_1x_2^{9}x_3^{19}x_4^{6}x_5^{26}\\ &\quad + x_1x_2^{9}x_3^{22}x_4^{3}x_5^{26}  + x_1^{2}x_2^{7}x_3^{3}x_4^{23}x_5^{26}  + x_1^{2}x_2^{7}x_3^{3}x_4^{24}x_5^{25}  + x_1^{2}x_2^{7}x_3^{4}x_4^{23}x_5^{25}\\ &\quad  + x_1^{2}x_2^{7}x_3^{7}x_4^{19}x_5^{26}  + x_1^{2}x_2^{7}x_3^{7}x_4^{20}x_5^{25}  + x_1^{2}x_2^{7}x_3^{8}x_4^{19}x_5^{25}  + x_1^{2}x_2^{7}x_3^{23}x_4^{3}x_5^{26}\\ &\quad  + x_1^{2}x_2^{7}x_3^{23}x_4^{4}x_5^{25}  + x_1^{2}x_2^{7}x_3^{24}x_4^{3}x_5^{25}  + x_1^{2}x_2^{8}x_3^{3}x_4^{23}x_5^{25}  + x_1^{2}x_2^{8}x_3^{7}x_4^{19}x_5^{25}\\ &\quad  + x_1^{2}x_2^{8}x_3^{23}x_4^{3}x_5^{25}\big)  +  Sq^4\big(x_1x_2^{4}x_3^{3}x_4^{23}x_5^{28}  + x_1x_2^{4}x_3^{3}x_4^{29}x_5^{22}  + x_1x_2^{4}x_3^{6}x_4^{23}x_5^{25}\\ &\quad  + x_1x_2^{4}x_3^{7}x_4^{19}x_5^{28}  + x_1x_2^{4}x_3^{7}x_4^{22}x_5^{25}  + x_1x_2^{4}x_3^{13}x_4^{19}x_5^{22}  + x_1x_2^{4}x_3^{23}x_4^{3}x_5^{28} \\ &\quad + x_1x_2^{4}x_3^{23}x_4^{6}x_5^{25}  + x_1x_2^{4}x_3^{29}x_4^{3}x_5^{22}  + x_1x_2^{5}x_3^{2}x_4^{23}x_5^{28}  + x_1x_2^{5}x_3^{3}x_4^{22}x_5^{28}\\ &\quad  + x_1x_2^{5}x_3^{3}x_4^{28}x_5^{22}  + x_1x_2^{5}x_3^{6}x_4^{19}x_5^{28}  + x_1x_2^{5}x_3^{10}x_4^{15}x_5^{28}  + x_1x_2^{5}x_3^{11}x_4^{14}x_5^{28}\\ &\quad  + x_1x_2^{5}x_3^{15}x_4^{10}x_5^{28}  + x_1x_2^{5}x_3^{19}x_4^{6}x_5^{28}  + x_1x_2^{5}x_3^{22}x_4^{3}x_5^{28}  + x_1x_2^{5}x_3^{23}x_4^{2}x_5^{28}\\ &\quad  + x_1x_2^{5}x_3^{26}x_4^{5}x_5^{22}  + x_1x_2^{6}x_3^{3}x_4^{23}x_5^{26}  + x_1x_2^{6}x_3^{4}x_4^{23}x_5^{25}  + x_1x_2^{6}x_3^{7}x_4^{19}x_5^{26} \\ &\quad + x_1x_2^{6}x_3^{7}x_4^{20}x_5^{25}  + x_1x_2^{6}x_3^{23}x_4^{3}x_5^{26}  + x_1x_2^{6}x_3^{23}x_4^{4}x_5^{25}  + x_1x_2^{10}x_3^{3}x_4^{23}x_5^{22} \\ &\quad + x_1x_2^{10}x_3^{7}x_4^{19}x_5^{22}  + x_1x_2^{10}x_3^{23}x_4^{3}x_5^{22}  + x_1x_2^{11}x_3^{3}x_4^{22}x_5^{22}  + x_1x_2^{11}x_3^{6}x_4^{15}x_5^{26}\\ &\quad + x_1x_2^{11}x_3^{7}x_4^{14}x_5^{26}  + x_1x_2^{11}x_3^{11}x_4^{10}x_5^{26}  + x_1x_2^{11}x_3^{11}x_4^{18}x_5^{18}  + x_1x_2^{11}x_3^{15}x_4^{6}x_5^{26}\\ &\quad  + x_1x_2^{11}x_3^{22}x_4^{3}x_5^{22}\big)  +  Sq^8\big(x_1x_2^{6}x_3^{3}x_4^{23}x_5^{22}  + x_1x_2^{6}x_3^{7}x_4^{19}x_5^{22}  + x_1x_2^{6}x_3^{23}x_4^{3}x_5^{22}\\ &\quad  + x_1x_2^{7}x_3^{3}x_4^{22}x_5^{22}  + x_1x_2^{7}x_3^{6}x_4^{15}x_5^{26}  + x_1x_2^{7}x_3^{6}x_4^{23}x_5^{18}  + x_1x_2^{7}x_3^{7}x_4^{14}x_5^{26} \\ &\quad + x_1x_2^{7}x_3^{7}x_4^{22}x_5^{18}  + x_1x_2^{7}x_3^{11}x_4^{10}x_5^{26}  + x_1x_2^{7}x_3^{15}x_4^{6}x_5^{26}  + x_1x_2^{7}x_3^{19}x_4^{10}x_5^{18}\\ &\quad  + x_1x_2^{7}x_3^{22}x_4^{3}x_5^{22}  + x_1x_2^{7}x_3^{23}x_4^{6}x_5^{18}  + x_1x_2^{11}x_3^{6}x_4^{19}x_5^{18}  + x_1x_2^{11}x_3^{7}x_4^{18}x_5^{18}\\ &\quad  + x_1x_2^{11}x_3^{19}x_4^{6}x_5^{18}\big)  +  Sq^{16}\big(x_1x_2^{7}x_3^{6}x_4^{15}x_5^{18}  + x_1x_2^{7}x_3^{7}x_4^{14}x_5^{18} \\ &\quad + x_1x_2^{7}x_3^{11}x_4^{10}x_5^{18}  + x_1x_2^{7}x_3^{15}x_4^{6}x_5^{18}\big) \ \mbox{mod}(P_5^-((3)(2)|^4).
\end{align*}
Thus, the monomial $u$ is strictly inadmissible.
\end{proof}
\begin{proof}[Proof of Proposition $\ref{mdd71}$] Let $d \geqslant 5$ and let $x \in P_5^+((3)|(2)|^{d-1})$ be an admissible monomial. Then $x = X_{i,j}y^2$ with $1 \leqslant i<j \leqslant 5$, and $y$ a monomial of weight vector $(2)|^{d-1}$. By Theorem \ref{dlcb1}, we have $y \in B_5((2)^{d-1})$.
	
Let $z \in B_5((2)^{d-1})$ such that $X_{i,j}z^2 \in P_5^+$ and $\widetilde A(d)$, $\widetilde C(d)$ as defined in Subsection \ref{s53}. By a direct computation we see that if $X_{i,j}z^2 \notin \widetilde A(d)\cup \widetilde C(d)$, then there is a monomial $w$ as given in one of Lemmas \ref{bda72} and \ref{bda73} such that $X_{i,j}z^2 = wh^{2^{s}}$ with $s$ a nonnegative integer, $2\leqslant s \leqslant 5$, and $h$ a monomial of weight vector $(2)|^{d-s}$. By Theorem \ref{dlcb1}, $X_{i,j}z^2$ is inadmissible. Since $x = X_{i,j}y^2$ is admissible and $y \in B_5((2)^{d-1})$, we obtain $x\in \widetilde A(d)\cup \widetilde C(d)$. This implies that
$$B_5((3)|(2)|^{d-1})\subset \widetilde A(d)\cup \widetilde C(d).$$
	
We prove that $\widetilde A(d)\cup \widetilde C(d)$ is a minimal set of $\mathcal A$-generators for $P_5^+(\tilde\omega)$ with $\tilde\omega := (3)|(2)|^{d-1}$.

Consider the subspaces $\langle [\widetilde A(d)]_{\tilde\omega}\rangle \subset QP_5(\tilde\omega)$ and $\langle [\widetilde C(d)]_{\tilde\omega}\rangle \subset QP_5(\tilde\omega)$. Observe that for any $x\in \widetilde A(d)$, we have $x = x_i^{2^d-1}f_i(y)$ with $y$ an admissible monomial of weight vector $(2)|(1)|^{d-1}$ in $P_4$. By Proposition \ref{mdmo}, $x$ is admissible and we get $\dim \langle [\widetilde A(d)]_{\tilde\omega}\rangle = 35$. Since $\nu(x) = 2^d-1$ for all $x\in \widetilde A(d)$ and $\nu(x) < 2^d-1$ for all $x\in \widetilde C(d)$, we obtain $\langle [\widetilde A(d)]_{\tilde\omega}\rangle \cap \langle [\widetilde C(d)]_{\tilde\omega}\rangle = \{0\}$. So, we need only to prove the set $[\widetilde C(d)]_{\tilde\omega}=\{[c_{d,t}]_{\tilde\omega}: 1 \leqslant t \leqslant 150 \}$ is linearly independent in $QP_5(\tilde\omega)$, where the monomials $c_{d,t},\, 1 \leqslant t \leqslant 150$, are determined as in Subsection \ref{s53}. 

Suppose there is a linear relation
\begin{equation}\label{ctd712}
\mathcal S:= \sum_{1\leqslant t \leqslant 150}\gamma_tc_{d,t} \equiv_{\tilde\omega} 0,
\end{equation}
where $\gamma_t \in \mathbb F_2$. We denote $\gamma_{\mathbb J} = \sum_{t \in \mathbb J}\gamma_t$ for any $\mathbb J \subset \{t\in \mathbb N:1\leqslant t \leqslant 150\}= \mathbb N_{150}$.

Let $w_u=w_{d,u},\, 1\leqslant u \leqslant 42$, be as in Subsection \ref{s52} and the homomorphism $p_{(i;I)}:P_5\to P_4$ defined by \eqref{ct23} for $k=5$. By applying $p_{(1;2)}$ and $p_{(4;5)}$  to (\ref{ctd712}), we obtain  
\begin{align*}
p_{(1;2)}&\equiv_{\tilde\omega} \gamma_{\{21,24\}}w_{6} + \gamma_{\{18,22\}}w_{7} + \gamma_{\{18,21,25\}}w_{9} + \gamma_{22}w_{10} + \gamma_{24}w_{11} + \gamma_{25}w_{13}\\ &\qquad + \gamma_{\{13,21\}}w_{20} + \gamma_{\{14,21\}}w_{21} + \gamma_{\{15,18,22\}}w_{23} + \gamma_{\{16,21,24\}}w_{24}\\ &\qquad + \gamma_{\{17,18,21,25\}}w_{25} + \gamma_{19}w_{26} + \gamma_{20}w_{27} + \gamma_{23}w_{28} + \gamma_{26}w_{29} + \gamma_{27}w_{30}\\ &\qquad + \gamma_{28}w_{31} + \gamma_{29}w_{32} + \gamma_{\{57,109\}}w_{34} + \gamma_{\{58,110\}}w_{35} + \gamma_{\{59,111\}}w_{36}\\ &\qquad + \gamma_{\{60,112\}}w_{37} + \gamma_{69}w_{38} + \gamma_{70}w_{39} + \gamma_{71}w_{40} + \gamma_{72}w_{41} + \gamma_{73}w_{42} \equiv_{\tilde\omega} 0,\\ 
p_{(4;5)}&\equiv_{\tilde\omega} \gamma_{\{9,10,11,12\}}w_{1} + \gamma_{\{26,27,28,29\}}w_{3} + \gamma_{\{48,49,50,51\}}w_{5} + \gamma_{52}w_{6}\\ &\qquad  + \gamma_{\{53,54,55,56\}}w_{7} + \gamma_{60}w_{9} + \gamma_{67}w_{10} + \gamma_{68}w_{11} + \gamma_{\{69,70,71,72\}}w_{12}\\ &\qquad  + \gamma_{73}w_{13} + \gamma_{\{92,93,94,95\}}w_{19} + \gamma_{96}w_{20} + \gamma_{\{97,98,99,100\}}w_{21} + \gamma_{107}w_{23}\\ &\qquad  + \gamma_{108}w_{24} + \gamma_{112}w_{25} + \gamma_{\{123,124,125,126\}}w_{26} + \gamma_{127}w_{27} + \gamma_{131}w_{28}\\ &\qquad  + \gamma_{132}w_{29} + \gamma_{\{133,134,135,136\}}w_{30} + \gamma_{137}w_{31} + \gamma_{144}w_{34} + \gamma_{145}w_{35}\\ &\qquad  + \gamma_{149}w_{36} + \gamma_{150}w_{37} \equiv_{\tilde\omega} 0. 
\end{align*}
These equalities implies
\begin{equation}\label{ct41}
\gamma_{109} = \gamma_{57},\, \gamma_{110} = \gamma_{58},\ \gamma_{111} = \gamma_{59}\mbox{ and } \gamma_t = 0,\mbox{ for } t\in \mathbb J_1,
\end{equation}
where $\mathbb J_1 =\{$13, 14, 15, 16, 17, 18, 19, 20, 21, 22, 23, 24, 25, 26, 27, 28, 29, 52, 60, 67, 68, 69, 70, 71, 72, 73, 96, 107, 108, 112, 127, 131, 132, 137, 144, 145, 149, 150$\}$. 

By applying $p_{(1;3)}$, $p_{(2;3)}$  to \eqref{ctd712} and using \eqref{ct41}, we have 
\begin{align*}
p_{(1;3)}&\equiv_{\tilde\omega} \gamma_{\{2,38,44,45,79,104\}}w_{6} + \gamma_{\{3,35,42,44,82,102\}}w_{7} + \gamma_{\{35,38\}}w_{9}\\ &\qquad + \gamma_{\{6,42,80,102\}}w_{10} + \gamma_{\{7,45,81,104\}}w_{11} + \gamma_{\{9,54,55,56,92\}}w_{15}\\ &\qquad + \gamma_{\{10,54,93\}}w_{16} + \gamma_{\{11,55,94\}}w_{17} + \gamma_{\{12,56,95\}}w_{18} + \gamma_{\{38,64\}}w_{20}\\ &\qquad + \gamma_{\{38,62\}}w_{21} + \gamma_{\{32,35,42,62,102\}}w_{23} + \gamma_{\{33,38,45,64,104\}}w_{24}\\ &\qquad + \gamma_{\{35,38\}}w_{25} + \gamma_{36}w_{26} + \gamma_{37}w_{27} + \gamma_{\{43,44,103\}}w_{28} + \gamma_{48}w_{29}\\ &\qquad + \gamma_{49}w_{30} + \gamma_{50}w_{31} + \gamma_{51}w_{32} + \gamma_{\{53,54,55,56\}}w_{33} + \gamma_{62}w_{34}\\ &\qquad + \gamma_{64}w_{35} + \gamma_{63}w_{36} \equiv_{\tilde\omega} 0,\\ 
p_{(2;3)}&\equiv_{\tilde\omega} \gamma_{\{2,35,82,88,89,104,116,117,120,121\}}w_{6} + \gamma_{\{3,38,79,86,88,102,114,116\}}w_{7}\\ &\qquad + \gamma_{\{79,82\}}w_{9} + \gamma_{\{6,36,57,86,102,114,120\}}w_{10} + \gamma_{\{7,37,58,89,104,117,121\}}w_{11}\\ &\qquad + \gamma_{\{9,48,98,99,100,124,125,126,134,135,136\}}w_{15} + \gamma_{\{10,49,98,124,134\}}w_{16}\\ &\qquad + \gamma_{\{11,50,99,125,135\}}w_{17} + \gamma_{\{12,51,100,126,136\}}w_{18} + \gamma_{\{82,141\}}w_{20}\\ &\qquad + \gamma_{\{82,139\}}w_{21} + \gamma_{\{76,79,86,102,114,139\}}w_{23} + \gamma_{\{77,82,89,104,117,141\}}w_{24}\\ &\qquad + \gamma_{\{79,82\}}w_{25} + \gamma_{\{57,80\}}w_{26} + \gamma_{\{58,81\}}w_{27} + \gamma_{92}w_{29} + \gamma_{93}w_{30}\\ &\qquad + \gamma_{\{59,87,88,103,115,116,120,121\}}w_{28} + \gamma_{94}w_{31} + \gamma_{95}w_{32}\\ &\qquad + \gamma_{\{97,98,99,100,123,124,125,126,133,134,135,136\}}w_{33} + \gamma_{139}w_{34}\\ &\qquad + \gamma_{141}w_{35} + \gamma_{140}w_{36} \equiv_{\tilde\omega} 0. 
\end{align*} 
Computing from the above equalities gives
\begin{equation}\label{ct42}
\gamma_{54} = \gamma_{10},\, \gamma_{55} = \gamma_{11},\ \gamma_{56} = \gamma_{12},\, \gamma_{80} = \gamma_{57},\, \gamma_{81} = \gamma_{58},\,   \gamma_t = 0,\mbox{ for } t\in \mathbb J_2,
\end{equation}
where $\mathbb J_2 =\{$35, 36, 37, 38, 48, 49, 50, 51, 62, 63, 64, 79, 82, 92, 93, 94, 95, 139, 140, 141$\}$.

By applying $p_{(1;5)}$ and $p_{(2;5)}$  to (\ref{ctd712}) and using \eqref{ct41}, \eqref{ct42} we get  
\begin{align*}
p_{(1;5)}&\equiv_{\tilde\omega} \gamma_{\{1,39,58,75,103,105,115,118,123,133,148\}}w_{1} + \gamma_{\{2,59,77,87,103,115\}}w_{3}\\ &\qquad + \gamma_{\{4,39,59,78,84,90,105,118,143,148\}}w_{5} + \gamma_{\{11,41,45,47\}}w_{6}\\ &\qquad + \gamma_{\{5,11,43,46,47,59,85\}}w_{7} + \gamma_{\{7,41,43,45,59,89,125,135\}}w_{9}\\ &\qquad + \gamma_{\{8,46,91\}}w_{10} + \gamma_{\{9,97,123,133\}}w_{12} + \gamma_{\{11,99,125,135\}}w_{13}\\ &\qquad + \gamma_{\{30,39,101,103,105,113,115,118,123,128,133,147,148\}}w_{19}\\ &\qquad + \gamma_{\{33,41,45,66,104\}}w_{20} + \gamma_{\{31,33,41,45,58,65,66,104,130,135\}}w_{21}\\ &\qquad + \gamma_{\{11,34,43,46,65,106\}}w_{23} + \gamma_{\{41,43,45,58,59,104,117,121,125,135\}}w_{25}\\ &\qquad + \gamma_{41}w_{24} + \gamma_{\{40,43,59,106,119,122,130\}}w_{26} + \gamma_{45}w_{27} + \gamma_{47}w_{28}\\ &\qquad + \gamma_{53}w_{29} + \gamma_{\{11,53\}}w_{30} + \gamma_{11}w_{31} + \gamma_{\{61,65,129,130,135\}}w_{34}\\ &\qquad + \gamma_{66}w_{36} \equiv_{\tilde\omega} 0,\\ 
p_{(2;5)}&\equiv_{\tilde\omega} \gamma_{\{1,31,83,103,105,130\}}w_{1} + \gamma_{\{2,33,43,103\}}w_{3} + \gamma_{\{4,34,40,46,66,83,105,130\}}w_{5}\\ &\qquad + \gamma_{\{85,89,91,99\}}w_{6} + \gamma_{\{5,41,87,90,91,99,121,122,125\}}w_{7}\\ &\qquad + \gamma_{\{7,45,85,87,89,121,125\}}w_{9} + \gamma_{\{8,47,90,122\}}w_{10} + \gamma_{\{9,53\}}w_{12}\\ &\qquad + \gamma_{\{74,83,101,103,105,129,130,146\}}w_{19} + \gamma_{\{77,85,89,104,117,143\}}w_{20}\\ &\qquad + \gamma_{\{75,77,85,89,104,117,142,143,148\}}w_{21}  + \gamma_{\{58,85,87,89,104,121,125\}}w_{25}\\ &\qquad + \gamma_{\{78,87,90,99,106,119,121,122,125,142\}}w_{23} + \gamma_{85}w_{24}\\ &\qquad + \gamma_{\{84,87,106,121,125,148\}}w_{26} + \gamma_{89}w_{27} + \gamma_{91}w_{28} + \gamma_{97}w_{29}\\ &\qquad + \gamma_{\{97,99\}}w_{30} + \gamma_{99}w_{31} + \gamma_{\{138,142,147,148\}}w_{34} + \gamma_{143}w_{36} \equiv_{\tilde\omega} 0. 
\end{align*} 
By computing from the above equalities, we get
\begin{equation}\label{ct43}
\begin{cases}\gamma_{12} = \gamma_{10},\, \gamma_{46} = \gamma_{8},\, \gamma_{100} = \gamma_{98},\, \gamma_{104} = \gamma_{33},\\ \gamma_{133} = \gamma_{123},\, \gamma_{135} = \gamma_{125},\, 
\gamma_t = 0,\mbox{ for } t\in \mathbb J_3,
\end{cases}
\end{equation}
where $\mathbb J_3 =\{$9, 11, 41, 45, 47, 53, 55, 66, 85, 89, 91, 97, 99, 143$\}$.

By applying $p_{(1;4)}$ and $p_{(2;4)}$  to (\ref{ctd712}) and using \eqref{ct41}, \eqref{ct42}, \eqref{ct43} we get  
\begin{align*}
p_{(1;4)}&\equiv_{\tilde\omega} \gamma_{\{1,74,106,116,119,126,136\}}w_{2} + \gamma_{\{3,76,88,116\}}w_{4} + \gamma_{\{4,10,44,83,120,122,124\}}w_{6}\\ &\qquad + \gamma_{\{10,39,42\}}w_{7} + \gamma_{\{5,78,106,119,142\}}w_{8} + \gamma_{\{6,39,42,44,86,120,134\}}w_{9}\\ &\qquad + \gamma_{\{8,90,122\}}w_{11} + \gamma_{\{10,98,124,134\}}w_{13} + \gamma_{\{10,98,126,136\}}w_{14}\\ &\qquad + \gamma_{\{30,32,44,114\}}w_{20} + \gamma_{\{31,101,106,116,119,126,129,136,146\}}w_{22}\\ &\qquad + \gamma_{\{32,44,57,114,134\}}w_{21} + \gamma_{\{10,39,42\}}w_{23} + \gamma_{\{10,34,118\}}w_{24}\\ &\qquad + \gamma_{\{39,42,44,57,102,114,134\}}w_{25} + \gamma_{42}w_{26} + \gamma_{\{40,44,105,118,120,122,124,130\}}w_{27}\\ &\qquad + \gamma_{8}w_{28} + \gamma_{10}w_{30} + \gamma_{10}w_{32} + \gamma_{\{61,128,134\}}w_{35} + \gamma_{65}w_{36} \equiv_{\tilde\omega} 0,\\ 
p_{(2;4)}&\equiv_{\tilde\omega} \gamma_{\{1,30,106\}}w_{2} + \gamma_{\{3,32,44\}}w_{4} + \gamma_{\{4,39,88,98\}}w_{6} + \gamma_{\{83,86,98\}}w_{7}\\ &\qquad + \gamma_{\{5,34,59,65,106\}}w_{8} + \gamma_{\{6,42,83,86,88\}}w_{9} + \gamma_{\{74,76,88,114\}}w_{20}\\ &\qquad + \gamma_{\{57,76,88,114\}}w_{21} + \gamma_{\{58,75,101,106,113,128,147\}}w_{22} + \gamma_{\{83,86,98\}}w_{23}\\ &\qquad + \gamma_{\{59,78,98,118\}}w_{24} + \gamma_{\{57,83,86,88,102\}}w_{25} + \gamma_{86}w_{26} + \gamma_{\{84,88,105,148\}}w_{27}\\ &\qquad + \gamma_{90}w_{28} + \gamma_{98}w_{30} + \gamma_{98}w_{32} + \gamma_{\{138,146\}}w_{35} + \gamma_{142}w_{36} \equiv_{\tilde\omega} 0. 
\end{align*} 
Computing from the above equalities gives
\begin{equation}\label{ct44}
\begin{cases} 
\gamma_{6} = \gamma_{88} = \gamma_{4},\, \gamma_{102} = \gamma_{32},\, \gamma_{118} = \gamma_{34},\, \gamma_{134} = \gamma_{124},\\ \gamma_{136} = \gamma_{126},\, \gamma_{146} = \gamma_{138},\, 
\gamma_t = 0,\mbox{ for } t\in \mathbb J_4,
\end{cases}
\end{equation}
where $\mathbb J_4 =\{$8, 10, 12, 39, 42, 46, 54, 56, 65, 83, 86, 90, 98, 100, 122, 142$\}$.

By applying $p_{(3;4)}$ and $p_{(3;5)}$  to (\ref{ctd712}) and using \eqref{ct41}, \eqref{ct42}, \eqref{ct43}, \eqref{ct44}, we obtain  
\begin{align*}
p_{(3;4)}&\equiv_{\tilde\omega} \gamma_{\{1,2,4,117,119,121,125\}}w_{2} + \gamma_{\{30,116,120,125,129,130\}}w_{6} + \gamma_{\{57,116,120,124\}}w_{9}\\ &\qquad + \gamma_{\{32,124\}}w_{7} + \gamma_{\{31,33,34,40,43,117,119,121,125\}}w_{8}  + \gamma_{\{61,125,129,130\}}w_{11}\\ &\qquad + \gamma_{\{74,116,120,124,147,148\}}w_{20} + \gamma_{\{57,76,116,120,124\}}w_{21}\\ &\qquad + \gamma_{\{58,75,77,78,84,87,117,119,121,125\}}w_{22} + \gamma_{\{32,124\}}w_{23}\\ &\qquad + \gamma_{\{101,103,105,116,120,124,129,130,147,148\}}w_{24} + \gamma_{\{57,116,120,124\}}w_{25}\\ &\qquad + \gamma_{114}w_{26} + \gamma_{\{34,113,115,123\}}w_{27} + \gamma_{\{124,128\}}w_{28} + \gamma_{124}w_{30}\\ &\qquad + \gamma_{\{123,125\}}w_{31} + \gamma_{126}w_{32} + \gamma_{\{138,147,148\}}w_{35} + \gamma_{138}w_{36} \equiv_{\tilde\omega} 0,\\ 
p_{(3;5)}&\equiv_{\tilde\omega} \gamma_{\{1,3,5,34,59,114\}}w_{1} + \gamma_{\{30,32,44,114\}}w_{5} + \gamma_{\{33,129\}}w_{6}\\ &\qquad + \gamma_{\{31,115,124,128,129\}}w_{7} + \gamma_{\{58,115\}}w_{9} + \gamma_{\{61,124,128\}}w_{10}\\ &\qquad + \gamma_{\{4,34,59,74,76,78,114\}}w_{19} + \gamma_{\{77,147\}}w_{20} + \gamma_{\{75,138,147\}}w_{21}\\ &\qquad + \gamma_{\{101,106,115,116,119,124,128,138\}}w_{23} + \gamma_{\{33,117\}}w_{24} + \gamma_{\{58,115\}}w_{25}\\ &\qquad + \gamma_{\{113,116,119,126\}}w_{26} + \gamma_{117}w_{27} + \gamma_{\{125,129\}}w_{28} + \gamma_{123}w_{29}\\ &\qquad + \gamma_{\{124,126\}}w_{30} + \gamma_{125}w_{31} + \gamma_{147}w_{36} \equiv_{\tilde\omega} 0. 
\end{align*} 
Now, by a direct computation from the above equalities, we get
$\gamma_t = 0$ for all $t\in \mathbb N_{150}$.
The proposition is proved.
\end{proof}

\section{Appendix}\label{sect5}

In this section, we list all needed admissible monomials which are used in the proofs of the main results.

\subsection{The admissible monomials of weight vector $(2)|^d$}\label{s51}\

\medskip
For any $d \geqslant 1$,
$$B_2(2^{d+1}-2) = B_2^+(2^{d+1}-2) = B((2)|^d) = \{(x_1x_2)^{2^d-1}\}.$$ 

For any $d \geqslant 2$, 
$$B_2((2)|^d)= \{x_1x_2^{2^d-2}x_3^{2^d-1}, x_1x_2^{2^d-1}x_3^{2^d-2}, x_1^{2^d-1}x_2x_3^{2^d-2}, x_1^3x_2^{2^d-3}x_3^{2^d-2} \}.$$

For any $d \geqslant 4$,
$B_3^+((2)|^d) = \{u_{d,t}: 1 \leqslant t \leqslant 13 \}$, where

\medskip
\centerline{\begin{tabular}{lll}
$u_{1} = x_1x_2x_3^{2^d-2}x_4^{2^d-2}$ &$u_{2} = x_1x_2^{2}x_3^{2^d-4}x_4^{2^d-1}$ &$u_{3} = x_1x_2^{2}x_3^{2^d-3}x_4^{2^d-2}$\cr  
$u_{4} = x_1x_2^{2}x_3^{2^d-1}x_4^{2^d-4}$ &$u_{5} = x_1x_2^{3}x_3^{2^d-4}x_4^{2^d-2}$ &$u_{6} = x_1x_2^{3}x_3^{2^d-2}x_4^{2^d-4}$\cr  
$u_{7} = x_1x_2^{2^d-2}x_3x_4^{2^d-2}$ &$u_{8} = x_1x_2^{2^d-1}x_3^{2}x_4^{2^d-4}$ &$u_{9} = x_1^{3}x_2x_3^{2^d-4}x_4^{2^d-2}$\cr  
$u_{10} = x_1^{3}x_2x_3^{2^d-2}x_4^{2^d-4}$ &$u_{11} = x_1^{3}x_2^{5}x_3^{2^d-6}x_4^{2^d-4}$ &$u_{12} = x_1^{3}x_2^{2^d-3}x_3^{2}x_4^{2^d-4}$\cr  
$u_{13} = x_1^{2^d-1}x_2x_3^{2}x_4^{2^d-4}$ &
\end{tabular}}

\bigskip
For any $d \geqslant 5$, $B_5^+((2)|^d) = \{a_t=a_{d,t}: 1 \leqslant t \leqslant 40 \}$, where

\medskip
\centerline{\begin{tabular}{ll}
$a_{1} = x_1x_2x_3^{2}x_4^{2^d-4}x_5^{2^d-2}$ &$a_{2} = x_1x_2x_3^{2}x_4^{2^d-2}x_5^{2^d-4}$\cr 
$a_{3} = x_1x_2x_3^{6}x_4^{2^d-6}x_5^{2^d-4}$   &$a_{4} = x_1x_2x_3^{2^d-2}x_4^{2}x_5^{2^d-4}$\cr
$a_{5} = x_1x_2^{2}x_3x_4^{2^d-4}x_5^{2^d-2}$ &$a_{6} = x_1x_2^{2}x_3x_4^{2^d-2}x_5^{2^d-4}$\cr  
$a_{7} = x_1x_2^{2}x_3^{3}x_4^{2^d-4}x_5^{2^d-4}$ &$a_{8} = x_1x_2^{2}x_3^{4}x_4^{2^d-8}x_5^{2^d-1}$\cr 
$a_{9} = x_1x_2^{2}x_3^{4}x_4^{2^d-7}x_5^{2^d-2}$  &$a_{10} = x_1x_2^{2}x_3^{4}x_4^{2^d-1}x_5^{2^d-8}$\cr
$a_{11} = x_1x_2^{2}x_3^{5}x_4^{2^d-8}x_5^{2^d-2}$ &$a_{12} = x_1x_2^{2}x_3^{5}x_4^{2^d-6}x_5^{2^d-4}$\cr  
$a_{13} = x_1x_2^{2}x_3^{5}x_4^{2^d-2}x_5^{2^d-8}$ &$a_{14} = x_1x_2^{2}x_3^{2^d-4}x_4x_5^{2^d-2}$\cr
$a_{15} = x_1x_2^{2}x_3^{2^d-3}x_4^{2}x_5^{2^d-4}$ &  $a_{16} = x_1x_2^{2}x_3^{2^d-1}x_4^{4}x_5^{2^d-8}$\cr 
$a_{17} = x_1x_2^{3}x_3^{2}x_4^{2^d-4}x_5^{2^d-4}$ &$a_{18} = x_1x_2^{3}x_3^{4}x_4^{2^d-8}x_5^{2^d-2}$\cr 
$a_{19} = x_1x_2^{3}x_3^{4}x_4^{2^d-6}x_5^{2^d-4}$ &$a_{20} = x_1x_2^{3}x_3^{4}x_4^{2^d-2}x_5^{2^d-8}$\cr 
$a_{21} = x_1x_2^{3}x_3^{6}x_4^{2^d-8}x_5^{2^d-4}$&  $a_{22} = x_1x_2^{3}x_3^{6}x_4^{2^d-4}x_5^{2^d-8}$\cr \end{tabular}}
\centerline{\begin{tabular}{ll} 
$a_{23} = x_1x_2^{3}x_3^{2^d-4}x_4^{2}x_5^{2^d-4}$ &$a_{24} = x_1x_2^{3}x_3^{2^d-2}x_4^{4}x_5^{2^d-8}$\cr  
$a_{25} = x_1x_2^{2^d-2}x_3x_4^{2}x_5^{2^d-4}$ &$a_{26} = x_1x_2^{2^d-1}x_3^{2}x_4^{4}x_5^{2^d-8}$\cr 
$a_{27} = x_1^{3}x_2x_3^{2}x_4^{2^d-4}x_5^{2^d-4}$&  $a_{28} = x_1^{3}x_2x_3^{4}x_4^{2^d-8}x_5^{2^d-2}$\cr 
$a_{29} = x_1^{3}x_2x_3^{4}x_4^{2^d-6}x_5^{2^d-4}$ &$a_{30} = x_1^{3}x_2x_3^{4}x_4^{2^d-2}x_5^{2^d-8}$\cr  
$a_{31} = x_1^{3}x_2x_3^{6}x_4^{2^d-8}x_5^{2^d-4}$ &$a_{32} = x_1^{3}x_2x_3^{6}x_4^{2^d-4}x_5^{2^d-8}$\cr 
$a_{33} = x_1^{3}x_2x_3^{2^d-4}x_4^{2}x_5^{2^d-4}$&  $a_{34} = x_1^{3}x_2x_3^{2^d-2}x_4^{4}x_5^{2^d-8}$\cr 
$a_{35} = x_1^{3}x_2^{5}x_3^{2}x_4^{2^d-8}x_5^{2^d-4}$ &$a_{36} = x_1^{3}x_2^{5}x_3^{2}x_4^{2^d-4}x_5^{2^d-8}$\cr  
$a_{37} = x_1^{3}x_2^{5}x_3^{2^d-6}x_4^{4}x_5^{2^d-8}$ &$a_{38} = x_1^{3}x_2^{2^d-3}x_3^{2}x_4^{4}x_5^{2^d-8}$\cr 
$a_{39} = x_1^{2^d-1}x_2x_3^{2}x_4^{4}x_5^{2^d-8}$&  
$a_{40} = x_1^{3}x_2^{5}x_3^{10}x_4^{2^d-12}x_5^{2^d-8}$\cr
\end{tabular}}

\medskip
In particular,
$B_5^+((2)|^4) = \{a_{4,t}: 1 \leqslant t \leqslant 39\}.$

\subsection{The admissible monomials of weight vector $(4)|(3)|^{d-1}$}\label{s52}\

\medskip
For any $d \geqslant 5$, $B_4((4)|(3)|^{d-1}) = \{v_t= v_{d,t} : 1\leqslant t \leqslant 15\}$, where

\medskip
\centerline{\begin{tabular}{lll}
$v_{1} = x_1x_2^{2^{d}-1}x_3^{2^{d}-1}x_4^{2^{d}-1}$ &$v_{2} = x_1^{3}x_2^{2^{d}-3}x_3^{2^{d}-1}x_4^{2^{d}-1}$ &$v_{3} = x_1^{3}x_2^{2^{d}-1}x_3^{2^{d}-3}x_4^{2^{d}-1}$\cr  $v_{4} = x_1^{3}x_2^{2^{d}-1}x_3^{2^{d}-1}x_4^{2^{d}-3}$ &$v_{5} = x_1^{7}x_2^{2^{d}-5}x_3^{2^{d}-3}x_4^{2^{d}-1}$ &$v_{6} = x_1^{7}x_2^{2^{d}-5}x_3^{2^{d}-1}x_4^{2^{d}-3}$\cr  $v_{7} = x_1^{7}x_2^{2^{d}-1}x_3^{2^{d}-5}x_4^{2^{d}-3}$ &$v_{8} = x_1^{15}x_2^{2^{d}-9}x_3^{2^{d}-5}x_4^{2^{d}-3}$ &$v_{9} = x_1^{2^{d}-1}x_2x_3^{2^{d}-1}x_4^{2^{d}-1}$\cr  $v_{10} = x_1^{2^{d}-1}x_2^{3}x_3^{2^{d}-3}x_4^{2^{d}-1}$ &$v_{11} = x_1^{2^{d}-1}x_2^{3}x_3^{2^{d}-1}x_4^{2^{d}-3}$ &$v_{12} = x_1^{2^{d}-1}x_2^{7}x_3^{2^{d}-5}x_4^{2^{d}-3}$\cr  $v_{13} = x_1^{2^{d}-1}x_2^{2^{d}-1}x_3x_4^{2^{d}-1}$ &$v_{14} = x_1^{2^{d}-1}x_2^{2^{d}-1}x_3^{3}x_4^{2^{d}-3}$ &$v_{15} = x_1^{2^{d}-1}x_2^{2^{d}-1}x_3^{2^{d}-1}x_4$\cr  
\end{tabular}}

\medskip
$B_4^+((3)|(2)|^{d-1}) =\{w_{t}=w_{d,t}: 1 \leqslant t \leqslant 44\}$, where

\medskip
\centerline{\begin{tabular}{lll}
$w_{1} = x_1x_2x_3^{2^d-2}x_4^{2^d-1}$ &$w_{2} = x_1x_2x_3^{2^d-1}x_4^{2^d-2}$ &$w_{3} = x_1x_2^{2}x_3^{2^d-3}x_4^{2^d-1}$\cr  $w_{4} = x_1x_2^{2}x_3^{2^d-1}x_4^{2^d-3}$ &$w_{5} = x_1x_2^{3}x_3^{2^d-4}x_4^{2^d-1}$ &$w_{6} = x_1x_2^{3}x_3^{2^d-3}x_4^{2^d-2}$\cr  $w_{7} = x_1x_2^{3}x_3^{2^d-2}x_4^{2^d-3}$ &$w_{8} = x_1x_2^{3}x_3^{2^d-1}x_4^{2^d-4}$ &$w_{9} = x_1x_2^{6}x_3^{2^d-5}x_4^{2^d-3}$\cr 
\end{tabular}}
\centerline{\begin{tabular}{lll}  
$w_{10} = x_1x_2^{7}x_3^{2^d-6}x_4^{2^d-3}$ &$w_{11} = x_1x_2^{7}x_3^{2^d-5}x_4^{2^d-4}$ &$w_{12} = x_1x_2^{2^d-2}x_3x_4^{2^d-1}$\cr   $w_{13} = x_1x_2^{2^d-2}x_3^{3}x_4^{2^d-3}$ &$w_{14} = x_1x_2^{2^d-2}x_3^{2^d-1}x_4$ &$w_{15} = x_1x_2^{2^d-1}x_3x_4^{2^d-2}$\cr  $w_{16} = x_1x_2^{2^d-1}x_3^{2}x_4^{2^d-3}$ &$w_{17} = x_1x_2^{2^d-1}x_3^{3}x_4^{2^d-4}$ &$w_{18} = x_1x_2^{2^d-1}x_3^{2^d-2}x_4$\cr $w_{19} = x_1^{3}x_2x_3^{2^d-4}x_4^{2^d-1}$ &$w_{20} = x_1^{3}x_2x_3^{2^d-3}x_4^{2^d-2}$ &$w_{21} = x_1^{3}x_2x_3^{2^d-2}x_4^{2^d-3}$\cr  $w_{22} = x_1^{3}x_2x_3^{2^d-1}x_4^{2^d-4}$ &$w_{23} = x_1^{3}x_2^{3}x_3^{2^d-4}x_4^{2^d-3}$ &$w_{24} = x_1^{3}x_2^{3}x_3^{2^d-3}x_4^{2^d-4}$\cr $w_{25} = x_1^{3}x_2^{4}x_3^{2^d-5}x_4^{2^d-3}$ &$w_{26} = x_1^{3}x_2^{5}x_3^{2^d-6}x_4^{2^d-3}$ &$w_{27} = x_1^{3}x_2^{5}x_3^{2^d-5}x_4^{2^d-4}$\cr $w_{28} = x_1^{3}x_2^{7}x_3^{2^d-7}x_4^{2^d-4}$ &$w_{29} = x_1^{3}x_2^{2^d-3}x_3x_4^{2^d-2}$ &$w_{30} = x_1^{3}x_2^{2^d-3}x_3^{2}x_4^{2^d-3}$\cr $w_{31} = x_1^{3}x_2^{2^d-3}x_3^{3}x_4^{2^d-4}$ &$w_{32} = x_1^{3}x_2^{2^d-3}x_3^{2^d-2}x_4$ &$w_{33} = x_1^{3}x_2^{2^d-1}x_3x_4^{2^d-4}$\cr $w_{34} = x_1^{7}x_2x_3^{2^d-6}x_4^{2^d-3}$ &$w_{35} = x_1^{7}x_2x_3^{2^d-5}x_4^{2^d-4}$ &$w_{36} = x_1^{7}x_2^{3}x_3^{2^d-7}x_4^{2^d-4}$\cr 
\end{tabular}}
\centerline{\begin{tabular}{lll} 
$w_{37} = x_1^{7}x_2^{2^d-5}x_3x_4^{2^d-4}$ &$w_{38} = x_1^{2^d-1}x_2x_3x_4^{2^d-2}$ &$w_{39} = x_1^{2^d-1}x_2x_3^{2}x_4^{2^d-3}$\cr  $w_{40} = x_1^{2^d-1}x_2x_3^{3}x_4^{2^d-4}$ &$w_{41} = x_1^{2^d-1}x_2x_3^{2^d-2}x_4$ &$w_{42} = x_1^{2^d-1}x_2^{3}x_3x_4^{2^d-4}$\cr  
\end{tabular}}

\medskip
$B_5^0((4)|(3)|^{d-1}) = \{f_i(x): x\in B_4((4)|(3)|^{d-1}), 1\leqslant i \leqslant 5\}$, where $f_i:P_4\to P_5$ is defined by \eqref{ct22}. Then, we have $|B_5^0((4)|(3)|^{d-1})| = 75$.

For any $d \geqslant 6$, we set 
$$A(d) = \{x_i^{2^d-1}f_i(w): w \in B_4^+((3)|(2)|^{d-1}), 1\leqslant i \leqslant 5\}.$$
We have $|A(d)| = 160$. 

Then, $B_5^+((4)|(3)|^{d-1}) = A(d)\cup C(d)$, where $C(d) = \{b_{d,t} : 1 \leqslant t \leqslant 75\}$ with the monomials $b_t=b_{d,t},\, 1 \leqslant t \leqslant 75$, determined as follows: 

\medskip
\centerline{\begin{tabular}{lll}
$b_{1} = x_1x_2^{7}x_3^{2^{d}-5}x_4^{2^{d}-3}x_5^{2^{d}-2}$ & &$b_{2} = x_1x_2^{7}x_3^{2^{d}-5}x_4^{2^{d}-2}x_5^{2^{d}-3}$\cr  $b_{3} = x_1x_2^{7}x_3^{2^{d}-2}x_4^{2^{d}-5}x_5^{2^{d}-3}$ & &$b_{4} = x_1x_2^{14}x_3^{2^{d}-9}x_4^{2^{d}-5}x_5^{2^{d}-3}$\cr  $b_{5} = x_1x_2^{15}x_3^{2^{d}-10}x_4^{2^{d}-5}x_5^{2^{d}-3}$ & &$b_{6} = x_1x_2^{15}x_3^{2^{d}-9}x_4^{2^{d}-6}x_5^{2^{d}-3}$\cr  $b_{7} = x_1x_2^{15}x_3^{2^{d}-9}x_4^{2^{d}-5}x_5^{2^{d}-4}$ & &$b_{8} = x_1x_2^{2^{d}-2}x_3^{7}x_4^{2^{d}-5}x_5^{2^{d}-3}$\cr  $b_{9} = x_1^{3}x_2^{3}x_3^{2^{d}-3}x_4^{2^{d}-3}x_5^{2^{d}-2}$ & &$b_{10} = x_1^{3}x_2^{3}x_3^{2^{d}-3}x_4^{2^{d}-2}x_5^{2^{d}-3}$\cr  $b_{11} = x_1^{3}x_2^{5}x_3^{2^{d}-5}x_4^{2^{d}-3}x_5^{2^{d}-2}$ & &$b_{12} = x_1^{3}x_2^{5}x_3^{2^{d}-5}x_4^{2^{d}-2}x_5^{2^{d}-3}$\cr  $b_{13} = x_1^{3}x_2^{5}x_3^{2^{d}-2}x_4^{2^{d}-5}x_5^{2^{d}-3}$ & &$b_{14} = x_1^{3}x_2^{7}x_3^{2^{d}-7}x_4^{2^{d}-3}x_5^{2^{d}-2}$\cr  $b_{15} = x_1^{3}x_2^{7}x_3^{2^{d}-7}x_4^{2^{d}-2}x_5^{2^{d}-3}$ & &$b_{16} = x_1^{3}x_2^{7}x_3^{2^{d}-5}x_4^{2^{d}-4}x_5^{2^{d}-3}$\cr  $b_{17} = x_1^{3}x_2^{7}x_3^{2^{d}-5}x_4^{2^{d}-3}x_5^{2^{d}-4}$ & &$b_{18} = x_1^{3}x_2^{7}x_3^{2^{d}-4}x_4^{2^{d}-5}x_5^{2^{d}-3}$\cr  $b_{19} = x_1^{3}x_2^{7}x_3^{2^{d}-3}x_4^{2^{d}-6}x_5^{2^{d}-3}$ & &$b_{20} = x_1^{3}x_2^{7}x_3^{2^{d}-3}x_4^{2^{d}-5}x_5^{2^{d}-4}$\cr  $b_{21} = x_1^{3}x_2^{13}x_3^{2^{d}-10}x_4^{2^{d}-5}x_5^{2^{d}-3}$ & &$b_{22} = x_1^{3}x_2^{13}x_3^{2^{d}-9}x_4^{2^{d}-6}x_5^{2^{d}-3}$\cr  $b_{23} = x_1^{3}x_2^{13}x_3^{2^{d}-9}x_4^{2^{d}-5}x_5^{2^{d}-4}$ & &$b_{24} = x_1^{3}x_2^{15}x_3^{2^{d}-11}x_4^{2^{d}-6}x_5^{2^{d}-3}$\cr  $b_{25} = x_1^{3}x_2^{15}x_3^{2^{d}-11}x_4^{2^{d}-5}x_5^{2^{d}-4}$ & &$b_{26} = x_1^{3}x_2^{15}x_3^{2^{d}-9}x_4^{2^{d}-7}x_5^{2^{d}-4}$\cr  $b_{27} = x_1^{3}x_2^{2^{d}-3}x_3^{3}x_4^{2^{d}-3}x_5^{2^{d}-2}$ & &$b_{28} = x_1^{3}x_2^{2^{d}-3}x_3^{3}x_4^{2^{d}-2}x_5^{2^{d}-3}$\cr  $b_{29} = x_1^{3}x_2^{2^{d}-3}x_3^{6}x_4^{2^{d}-5}x_5^{2^{d}-3}$ & &$b_{30} = x_1^{3}x_2^{2^{d}-3}x_3^{7}x_4^{2^{d}-6}x_5^{2^{d}-3}$\cr  $b_{31} = x_1^{3}x_2^{2^{d}-3}x_3^{7}x_4^{2^{d}-5}x_5^{2^{d}-4}$ & &$b_{32} = x_1^{3}x_2^{2^{d}-3}x_3^{2^{d}-2}x_4^{3}x_5^{2^{d}-3}$\cr  $b_{33} = x_1^{7}x_2x_3^{2^{d}-5}x_4^{2^{d}-3}x_5^{2^{d}-2}$ & &$b_{34} = x_1^{7}x_2x_3^{2^{d}-5}x_4^{2^{d}-2}x_5^{2^{d}-3}$\cr  $b_{35} = x_1^{7}x_2x_3^{2^{d}-2}x_4^{2^{d}-5}x_5^{2^{d}-3}$ & &$b_{36} = x_1^{7}x_2^{3}x_3^{2^{d}-7}x_4^{2^{d}-3}x_5^{2^{d}-2}$\cr  $b_{37} = x_1^{7}x_2^{3}x_3^{2^{d}-7}x_4^{2^{d}-2}x_5^{2^{d}-3}$ & &$b_{38} = x_1^{7}x_2^{3}x_3^{2^{d}-5}x_4^{2^{d}-4}x_5^{2^{d}-3}$\cr  $b_{39} = x_1^{7}x_2^{3}x_3^{2^{d}-5}x_4^{2^{d}-3}x_5^{2^{d}-4}$ & &$b_{40} = x_1^{7}x_2^{3}x_3^{2^{d}-4}x_4^{2^{d}-5}x_5^{2^{d}-3}$\cr  $b_{41} = x_1^{7}x_2^{3}x_3^{2^{d}-3}x_4^{2^{d}-6}x_5^{2^{d}-3}$ & &$b_{42} = x_1^{7}x_2^{3}x_3^{2^{d}-3}x_4^{2^{d}-5}x_5^{2^{d}-4}$\cr 
\end{tabular}}
\centerline{\begin{tabular}{lll}  
$b_{43} = x_1^{7}x_2^{7}x_3^{2^{d}-8}x_4^{2^{d}-5}x_5^{2^{d}-3}$ & &$b_{44} = x_1^{7}x_2^{7}x_3^{2^{d}-7}x_4^{2^{d}-6}x_5^{2^{d}-3}$\cr  $b_{45} = x_1^{7}x_2^{7}x_3^{2^{d}-7}x_4^{2^{d}-5}x_5^{2^{d}-4}$ & &$b_{46} = x_1^{7}x_2^{7}x_3^{2^{d}-5}x_4^{2^{d}-8}x_5^{2^{d}-3}$\cr  $b_{47} = x_1^{7}x_2^{7}x_3^{2^{d}-5}x_4^{2^{d}-7}x_5^{2^{d}-4}$ & &$b_{48} = x_1^{7}x_2^{7}x_3^{2^{d}-5}x_4^{2^{d}-3}x_5^{2^{d}-8}$\cr  $b_{49} = x_1^{7}x_2^{11}x_3^{2^{d}-11}x_4^{2^{d}-6}x_5^{2^{d}-3}$ & &$b_{50} = x_1^{7}x_2^{11}x_3^{2^{d}-11}x_4^{2^{d}-5}x_5^{2^{d}-4}$\cr  $b_{51} = x_1^{7}x_2^{11}x_3^{2^{d}-9}x_4^{2^{d}-7}x_5^{2^{d}-4}$ & &$b_{52} = x_1^{7}x_2^{15}x_3^{2^{d}-13}x_4^{2^{d}-7}x_5^{2^{d}-4}$\cr  $b_{53} = x_1^{7}x_2^{2^{d}-5}x_3x_4^{2^{d}-3}x_5^{2^{d}-2}$ & &$b_{54} = x_1^{7}x_2^{2^{d}-5}x_3x_4^{2^{d}-2}x_5^{2^{d}-3}$\cr  $b_{55} = x_1^{7}x_2^{2^{d}-5}x_3^{3}x_4^{2^{d}-4}x_5^{2^{d}-3}$ & &$b_{56} = x_1^{7}x_2^{2^{d}-5}x_3^{3}x_4^{2^{d}-3}x_5^{2^{d}-4}$\cr  $b_{57} = x_1^{7}x_2^{2^{d}-5}x_3^{4}x_4^{2^{d}-5}x_5^{2^{d}-3}$ & &$b_{58} = x_1^{7}x_2^{2^{d}-5}x_3^{5}x_4^{2^{d}-6}x_5^{2^{d}-3}$\cr  $b_{59} = x_1^{7}x_2^{2^{d}-5}x_3^{5}x_4^{2^{d}-5}x_5^{2^{d}-4}$ & &$b_{60} = x_1^{7}x_2^{2^{d}-5}x_3^{7}x_4^{2^{d}-7}x_5^{2^{d}-4}$\cr  $b_{61} = x_1^{7}x_2^{2^{d}-5}x_3^{2^{d}-3}x_4x_5^{2^{d}-2}$ & &$b_{62} = x_1^{7}x_2^{2^{d}-5}x_3^{2^{d}-3}x_4^{2}x_5^{2^{d}-3}$\cr $b_{63} = x_1^{7}x_2^{2^{d}-5}x_3^{2^{d}-3}x_4^{3}x_5^{2^{d}-4}$ & &$b_{64} = x_1^{7}x_2^{2^{d}-5}x_3^{2^{d}-3}x_4^{2^{d}-2}x_5$\cr  $b_{65} = x_1^{15}x_2x_3^{2^{d}-10}x_4^{2^{d}-5}x_5^{2^{d}-3}$ & &$b_{66} = x_1^{15}x_2x_3^{2^{d}-9}x_4^{2^{d}-6}x_5^{2^{d}-3}$\cr  $b_{67} = x_1^{15}x_2x_3^{2^{d}-9}x_4^{2^{d}-5}x_5^{2^{d}-4}$ & &$b_{68} = x_1^{15}x_2^{3}x_3^{2^{d}-11}x_4^{2^{d}-6}x_5^{2^{d}-3}$\cr  $b_{69} = x_1^{15}x_2^{3}x_3^{2^{d}-11}x_4^{2^{d}-5}x_5^{2^{d}-4}$ & &$b_{70} = x_1^{15}x_2^{3}x_3^{2^{d}-9}x_4^{2^{d}-7}x_5^{2^{d}-4}$\cr  $b_{71} = x_1^{15}x_2^{7}x_3^{2^{d}-13}x_4^{2^{d}-7}x_5^{2^{d}-4}$ & &$b_{72} = x_1^{15}x_2^{2^{d}-9}x_3x_4^{2^{d}-6}x_5^{2^{d}-3}$\cr  $b_{73} = x_1^{15}x_2^{2^{d}-9}x_3x_4^{2^{d}-5}x_5^{2^{d}-4}$ & &$b_{74} = x_1^{15}x_2^{2^{d}-9}x_3^{3}x_4^{2^{d}-7}x_5^{2^{d}-4}$\cr  $b_{75} = x_1^{15}x_2^{2^{d}-9}x_3^{2^{d}-5}x_4x_5^{2^{d}-4}$ & &\cr
\end{tabular}}

\medskip
Thus, we obtain $|B_5((4)|(3)|^{d-1})| = 75 + 160 + 75 = 310$.

\subsection{The admissible monomials of weight vector $(3)|(2)|^{d-1}$}\label{s53}\

\medskip
For any $d \geqslant 4$, $B_3^+((3)|(2)|^{d-1})$ is the set of 7 monomials:

\medskip
\centerline{\begin{tabular}{llll}
$x_1x_2^{2^d-1}x_3^{2^d-1}$ &$x_1^{3}x_2^{2^d-3}x_3^{2^d-1}$ &$x_1^{3}x_2^{2^d-1}x_3^{2^d-3}$ &$x_1^{7}x_2^{2^d-5}x_3^{2^d-3}$\cr  $x_1^{2^d-1}x_2x_3^{2^d-1}$ &$x_1^{2^d-1}x_2^{3}x_3^{2^d-3}$ &$x_1^{2^d-1}x_2^{2^d-1}x_3$ &\cr 
\end{tabular}}

\medskip
Hence, $|B_5^0\big((3)|(2)|^{d-1}\big)| = |B_3^+\big((3)|(2)|^{d-1}\big)|\binom 53 + |B_4^+\big((3)|(2)|^{d-1}\big)|\binom 54 = 280$.

For any $d \geqslant 6$, we set 
$$\widetilde A(d) = \{x_i^{2^d-1}f_i(u): u \in B_4^+((2)|(1)|^{d-1}), 1\leqslant i \leqslant 5\},$$
where $B_4^+((2)|(1)|^{d-1})$ is the set of 7 monomials:

\medskip
\centerline{\begin{tabular}{llll}
$x_1x_2x_3^{2}x_4^{2^d-4}$ &$x_1x_2^{2}x_3x_4^{2^d-4}$ &$x_1x_2^{2}x_3^{4}x_4^{2^d-7}$ &$x_1x_2^{2}x_3^{5}x_4^{2^d-8}$\cr  $x_1x_2^{2}x_3^{2^d-4}x_4$ &$x_1x_2^{3}x_3^{4}x_4^{2^d-8}$ &$x_1^{3}x_2x_3^{4}x_4^{2^d-8}$ &\cr  
\end{tabular}}

\medskip
We have $|\widetilde A(d)| = 35$. 

Then, $B_5^+((3)|(2)|^{d-1}) = \widetilde A(d)\cup \widetilde C(d)$, where $\widetilde C(d) = \{c_t=c_{d,t}: 1 \leqslant t \leqslant 150\}$ is determined as follows: 

\medskip
\centerline{\begin{tabular}{lll}
$c_{1} = x_1x_2x_3x_4^{2^d-2}x_5^{2^d-2}$ & &$c_{2} = x_1x_2x_3^{2}x_4^{2^d-3}x_5^{2^d-2}$\cr  $c_{3} = x_1x_2x_3^{2}x_4^{2^d-2}x_5^{2^d-3}$ & &$c_{4} = x_1x_2x_3^{3}x_4^{2^d-4}x_5^{2^d-2}$\cr  $c_{5} = x_1x_2x_3^{3}x_4^{2^d-2}x_5^{2^d-4}$ & &$c_{6} = x_1x_2x_3^{6}x_4^{2^d-6}x_5^{2^d-3}$\cr  $c_{7} = x_1x_2x_3^{6}x_4^{2^d-5}x_5^{2^d-4}$ & &$c_{8} = x_1x_2x_3^{7}x_4^{2^d-6}x_5^{2^d-4}$\cr  $c_{9} = x_1x_2x_3^{2^d-2}x_4x_5^{2^d-2}$ & &$c_{10} = x_1x_2x_3^{2^d-2}x_4^{2}x_5^{2^d-3}$\cr  $c_{11} = x_1x_2x_3^{2^d-2}x_4^{3}x_5^{2^d-4}$ & &$c_{12} = x_1x_2x_3^{2^d-2}x_4^{2^d-2}x_5$\cr  $c_{13} = x_1x_2^{2}x_3x_4^{2^d-3}x_5^{2^d-2}$ & &$c_{14} = x_1x_2^{2}x_3x_4^{2^d-2}x_5^{2^d-3}$\cr  $c_{15} = x_1x_2^{2}x_3^{3}x_4^{2^d-4}x_5^{2^d-3}$ & &$c_{16} = x_1x_2^{2}x_3^{3}x_4^{2^d-3}x_5^{2^d-4}$\cr   $c_{17} = x_1x_2^{2}x_3^{4}x_4^{2^d-5}x_5^{2^d-3}$ & &$c_{18} = x_1x_2^{2}x_3^{5}x_4^{2^d-7}x_5^{2^d-2}$\cr  $c_{19} = x_1x_2^{2}x_3^{5}x_4^{2^d-6}x_5^{2^d-3}$ & &$c_{20} = x_1x_2^{2}x_3^{5}x_4^{2^d-5}x_5^{2^d-4}$\cr  $c_{21} = x_1x_2^{2}x_3^{5}x_4^{2^d-2}x_5^{2^d-7}$ & &$c_{22} = x_1x_2^{2}x_3^{7}x_4^{2^d-8}x_5^{2^d-3}$\cr  $c_{23} = x_1x_2^{2}x_3^{7}x_4^{2^d-7}x_5^{2^d-4}$ & &$c_{24} = x_1x_2^{2}x_3^{7}x_4^{2^d-3}x_5^{2^d-8}$\cr  $c_{25} = x_1x_2^{2}x_3^{2^d-4}x_4^{3}x_5^{2^d-3}$ & &$c_{26} = x_1x_2^{2}x_3^{2^d-3}x_4x_5^{2^d-2}$\cr  $c_{27} = x_1x_2^{2}x_3^{2^d-3}x_4^{2}x_5^{2^d-3}$ & &$c_{28} = x_1x_2^{2}x_3^{2^d-3}x_4^{3}x_5^{2^d-4}$\cr  $c_{29} = x_1x_2^{2}x_3^{2^d-3}x_4^{2^d-2}x_5$ & &$c_{30} = x_1x_2^{3}x_3x_4^{2^d-4}x_5^{2^d-2}$\cr  $c_{31} = x_1x_2^{3}x_3x_4^{2^d-2}x_5^{2^d-4}$ & &$c_{32} = x_1x_2^{3}x_3^{2}x_4^{2^d-4}x_5^{2^d-3}$\cr  $c_{33} = x_1x_2^{3}x_3^{2}x_4^{2^d-3}x_5^{2^d-4}$ & &$c_{34} = x_1x_2^{3}x_3^{3}x_4^{2^d-4}x_5^{2^d-4}$\cr $c_{35} = x_1x_2^{3}x_3^{4}x_4^{2^d-7}x_5^{2^d-2}$ & &$c_{36} = x_1x_2^{3}x_3^{4}x_4^{2^d-6}x_5^{2^d-3}$\cr $c_{37} = x_1x_2^{3}x_3^{4}x_4^{2^d-5}x_5^{2^d-4}$ & &$c_{38} = x_1x_2^{3}x_3^{4}x_4^{2^d-2}x_5^{2^d-7}$\cr  $c_{39} = x_1x_2^{3}x_3^{5}x_4^{2^d-8}x_5^{2^d-2}$ & &$c_{40} = x_1x_2^{3}x_3^{5}x_4^{2^d-6}x_5^{2^d-4}$\cr  $c_{41} = x_1x_2^{3}x_3^{5}x_4^{2^d-2}x_5^{2^d-8}$ & &$c_{42} = x_1x_2^{3}x_3^{6}x_4^{2^d-8}x_5^{2^d-3}$\cr $c_{43} = x_1x_2^{3}x_3^{6}x_4^{2^d-7}x_5^{2^d-4}$ & &$c_{44} = x_1x_2^{3}x_3^{6}x_4^{2^d-4}x_5^{2^d-7}$\cr  $c_{45} = x_1x_2^{3}x_3^{6}x_4^{2^d-3}x_5^{2^d-8}$ & &$c_{46} = x_1x_2^{3}x_3^{7}x_4^{2^d-8}x_5^{2^d-4}$\cr  $c_{47} = x_1x_2^{3}x_3^{7}x_4^{2^d-4}x_5^{2^d-8}$ & &$c_{48} = x_1x_2^{3}x_3^{2^d-4}x_4x_5^{2^d-2}$\cr  $c_{49} = x_1x_2^{3}x_3^{2^d-4}x_4^{2}x_5^{2^d-3}$ & &$c_{50} = x_1x_2^{3}x_3^{2^d-4}x_4^{3}x_5^{2^d-4}$\cr  $c_{51} = x_1x_2^{3}x_3^{2^d-4}x_4^{2^d-2}x_5$ & &$c_{52} = x_1x_2^{3}x_3^{2^d-3}x_4^{2}x_5^{2^d-4}$\cr  $c_{53} = x_1x_2^{3}x_3^{2^d-2}x_4x_5^{2^d-4}$ & &$c_{54} = x_1x_2^{3}x_3^{2^d-2}x_4^{4}x_5^{2^d-7}$\cr \end{tabular}}
\centerline{\begin{tabular}{lll} $c_{55} = x_1x_2^{3}x_3^{2^d-2}x_4^{5}x_5^{2^d-8}$ & &$c_{56} = x_1x_2^{3}x_3^{2^d-2}x_4^{2^d-4}x_5$\cr $c_{57} = x_1x_2^{6}x_3x_4^{2^d-6}x_5^{2^d-3}$ & &$c_{58} = x_1x_2^{6}x_3x_4^{2^d-5}x_5^{2^d-4}$\cr  $c_{59} = x_1x_2^{6}x_3^{3}x_4^{2^d-7}x_5^{2^d-4}$ & &$c_{60} = x_1x_2^{6}x_3^{2^d-5}x_4x_5^{2^d-4}$\cr  $c_{61} = x_1x_2^{7}x_3x_4^{2^d-6}x_5^{2^d-4}$ & &$c_{62} = x_1x_2^{7}x_3^{2}x_4^{2^d-8}x_5^{2^d-3}$\cr  $c_{63} = x_1x_2^{7}x_3^{2}x_4^{2^d-7}x_5^{2^d-4}$ & &$c_{64} = x_1x_2^{7}x_3^{2}x_4^{2^d-3}x_5^{2^d-8}$\cr  $c_{65} = x_1x_2^{7}x_3^{3}x_4^{2^d-8}x_5^{2^d-4}$ & &$c_{66} = x_1x_2^{7}x_3^{3}x_4^{2^d-4}x_5^{2^d-8}$\cr  $c_{67} = x_1x_2^{7}x_3^{2^d-6}x_4x_5^{2^d-4}$ & &$c_{68} = x_1x_2^{7}x_3^{2^d-5}x_4^{4}x_5^{2^d-8}$\cr  $c_{69} = x_1x_2^{2^d-2}x_3x_4x_5^{2^d-2}$ & &$c_{70} = x_1x_2^{2^d-2}x_3x_4^{2}x_5^{2^d-3}$\cr  $c_{71} = x_1x_2^{2^d-2}x_3x_4^{3}x_5^{2^d-4}$ & &$c_{72} = x_1x_2^{2^d-2}x_3x_4^{2^d-2}x_5$\cr  $c_{73} = x_1x_2^{2^d-2}x_3^{3}x_4x_5^{2^d-4}$ & &$c_{74} = x_1^{3}x_2x_3x_4^{2^d-4}x_5^{2^d-2}$\cr  $c_{75} = x_1^{3}x_2x_3x_4^{2^d-2}x_5^{2^d-4}$ & &$c_{76} = x_1^{3}x_2x_3^{2}x_4^{2^d-4}x_5^{2^d-3}$\cr  $c_{77} = x_1^{3}x_2x_3^{2}x_4^{2^d-3}x_5^{2^d-4}$ & &$c_{78} = x_1^{3}x_2x_3^{3}x_4^{2^d-4}x_5^{2^d-4}$\cr  $c_{79} = x_1^{3}x_2x_3^{4}x_4^{2^d-7}x_5^{2^d-2}$ & &$c_{80} = x_1^{3}x_2x_3^{4}x_4^{2^d-6}x_5^{2^d-3}$\cr  $c_{81} = x_1^{3}x_2x_3^{4}x_4^{2^d-5}x_5^{2^d-4}$ & &$c_{82} = x_1^{3}x_2x_3^{4}x_4^{2^d-2}x_5^{2^d-7}$\cr  $c_{83} = x_1^{3}x_2x_3^{5}x_4^{2^d-8}x_5^{2^d-2}$ & &$c_{84} = x_1^{3}x_2x_3^{5}x_4^{2^d-6}x_5^{2^d-4}$\cr  $c_{85} = x_1^{3}x_2x_3^{5}x_4^{2^d-2}x_5^{2^d-8}$ & &$c_{86} = x_1^{3}x_2x_3^{6}x_4^{2^d-8}x_5^{2^d-3}$\cr  $c_{87} = x_1^{3}x_2x_3^{6}x_4^{2^d-7}x_5^{2^d-4}$ & &$c_{88} = x_1^{3}x_2x_3^{6}x_4^{2^d-4}x_5^{2^d-7}$\cr  $c_{89} = x_1^{3}x_2x_3^{6}x_4^{2^d-3}x_5^{2^d-8}$ & &$c_{90} = x_1^{3}x_2x_3^{7}x_4^{2^d-8}x_5^{2^d-4}$\cr  $c_{91} = x_1^{3}x_2x_3^{7}x_4^{2^d-4}x_5^{2^d-8}$ & &$c_{92} = x_1^{3}x_2x_3^{2^d-4}x_4x_5^{2^d-2}$\cr  $c_{93} = x_1^{3}x_2x_3^{2^d-4}x_4^{2}x_5^{2^d-3}$ & &$c_{94} = x_1^{3}x_2x_3^{2^d-4}x_4^{3}x_5^{2^d-4}$\cr $c_{95} = x_1^{3}x_2x_3^{2^d-4}x_4^{2^d-2}x_5$ & &$c_{96} = x_1^{3}x_2x_3^{2^d-3}x_4^{2}x_5^{2^d-4}$\cr  $c_{97} = x_1^{3}x_2x_3^{2^d-2}x_4x_5^{2^d-4}$ & &$c_{98} = x_1^{3}x_2x_3^{2^d-2}x_4^{4}x_5^{2^d-7}$\cr  $c_{99} = x_1^{3}x_2x_3^{2^d-2}x_4^{5}x_5^{2^d-8}$ & &$c_{100} = x_1^{3}x_2x_3^{2^d-2}x_4^{2^d-4}x_5$\cr  $c_{101} = x_1^{3}x_2^{3}x_3x_4^{2^d-4}x_5^{2^d-4}$ & &$c_{102} = x_1^{3}x_2^{3}x_3^{4}x_4^{2^d-8}x_5^{2^d-3}$\cr  $c_{103} = x_1^{3}x_2^{3}x_3^{4}x_4^{2^d-7}x_5^{2^d-4}$ & &$c_{104} = x_1^{3}x_2^{3}x_3^{4}x_4^{2^d-3}x_5^{2^d-8}$\cr  $c_{105} = x_1^{3}x_2^{3}x_3^{5}x_4^{2^d-8}x_5^{2^d-4}$ & &$c_{106} = x_1^{3}x_2^{3}x_3^{5}x_4^{2^d-4}x_5^{2^d-8}$\cr  $c_{107} = x_1^{3}x_2^{3}x_3^{2^d-4}x_4x_5^{2^d-4}$ & &$c_{108} = x_1^{3}x_2^{3}x_3^{2^d-3}x_4^{4}x_5^{2^d-8}$\cr  $c_{109} = x_1^{3}x_2^{4}x_3x_4^{2^d-6}x_5^{2^d-3}$ & &$c_{110} = x_1^{3}x_2^{4}x_3x_4^{2^d-5}x_5^{2^d-4}$\cr $c_{111} = x_1^{3}x_2^{4}x_3^{3}x_4^{2^d-7}x_5^{2^d-4}$ & &$c_{112} = x_1^{3}x_2^{4}x_3^{2^d-5}x_4x_5^{2^d-4}$\cr  $c_{113} = x_1^{3}x_2^{5}x_3x_4^{2^d-6}x_5^{2^d-4}$ & &$c_{114} = x_1^{3}x_2^{5}x_3^{2}x_4^{2^d-8}x_5^{2^d-3}$\cr  $c_{115} = x_1^{3}x_2^{5}x_3^{2}x_4^{2^d-7}x_5^{2^d-4}$ & &$c_{116} = x_1^{3}x_2^{5}x_3^{2}x_4^{2^d-4}x_5^{2^d-7}$\cr $c_{117} = x_1^{3}x_2^{5}x_3^{2}x_4^{2^d-3}x_5^{2^d-8}$ & &$c_{118} = x_1^{3}x_2^{5}x_3^{3}x_4^{2^d-8}x_5^{2^d-4}$\cr $c_{119} = x_1^{3}x_2^{5}x_3^{3}x_4^{2^d-4}x_5^{2^d-8}$ & &$c_{120} = x_1^{3}x_2^{5}x_3^{10}x_4^{2^d-12}x_5^{2^d-7}$\cr  $c_{121} = x_1^{3}x_2^{5}x_3^{10}x_4^{2^d-11}x_5^{2^d-8}$ & &$c_{122} = x_1^{3}x_2^{5}x_3^{11}x_4^{2^d-12}x_5^{2^d-8}$\cr  $c_{123} = x_1^{3}x_2^{5}x_3^{2^d-6}x_4x_5^{2^d-4}$ & &$c_{124} = x_1^{3}x_2^{5}x_3^{2^d-6}x_4^{4}x_5^{2^d-7}$\cr  $c_{125} = x_1^{3}x_2^{5}x_3^{2^d-6}x_4^{5}x_5^{2^d-8}$ & &$c_{126} = x_1^{3}x_2^{5}x_3^{2^d-6}x_4^{2^d-4}x_5$\cr  $c_{127} = x_1^{3}x_2^{5}x_3^{2^d-5}x_4^{4}x_5^{2^d-8}$ & &$c_{128} = x_1^{3}x_2^{7}x_3x_4^{2^d-8}x_5^{2^d-4}$\cr  $c_{129} = x_1^{3}x_2^{7}x_3x_4^{2^d-4}x_5^{2^d-8}$ & &$c_{130} = x_1^{3}x_2^{7}x_3^{9}x_4^{2^d-12}x_5^{2^d-8}$\cr $c_{131} = x_1^{3}x_2^{7}x_3^{2^d-7}x_4^{4}x_5^{2^d-8}$ & &$c_{132} = x_1^{3}x_2^{2^d-3}x_3x_4^{2}x_5^{2^d-4}$\cr  $c_{133} = x_1^{3}x_2^{2^d-3}x_3^{2}x_4x_5^{2^d-4}$ & &$c_{134} = x_1^{3}x_2^{2^d-3}x_3^{2}x_4^{4}x_5^{2^d-7}$\cr
\end{tabular}}
\centerline{\begin{tabular}{lll} 
$c_{135} = x_1^{3}x_2^{2^d-3}x_3^{2}x_4^{5}x_5^{2^d-8}$ & &$c_{136} = x_1^{3}x_2^{2^d-3}x_3^{2}x_4^{2^d-4}x_5$\cr $c_{137} = x_1^{3}x_2^{2^d-3}x_3^{3}x_4^{4}x_5^{2^d-8}$ & &$c_{138} = x_1^{7}x_2x_3x_4^{2^d-6}x_5^{2^d-4}$\cr $c_{139} = x_1^{7}x_2x_3^{2}x_4^{2^d-8}x_5^{2^d-3}$ & &$c_{140} = x_1^{7}x_2x_3^{2}x_4^{2^d-7}x_5^{2^d-4}$\cr  $c_{141} = x_1^{7}x_2x_3^{2}x_4^{2^d-3}x_5^{2^d-8}$ & &$c_{142} = x_1^{7}x_2x_3^{3}x_4^{2^d-8}x_5^{2^d-4}$\cr  $c_{143} = x_1^{7}x_2x_3^{3}x_4^{2^d-4}x_5^{2^d-8}$ & &$c_{144} = x_1^{7}x_2x_3^{2^d-6}x_4x_5^{2^d-4}$\cr  $c_{145} = x_1^{7}x_2x_3^{2^d-5}x_4^{4}x_5^{2^d-8}$ & &$c_{146} = x_1^{7}x_2^{3}x_3x_4^{2^d-8}x_5^{2^d-4}$\cr  $c_{147} = x_1^{7}x_2^{3}x_3x_4^{2^d-4}x_5^{2^d-8}$ & &$c_{148} = x_1^{7}x_2^{3}x_3^{9}x_4^{2^d-12}x_5^{2^d-8}$\cr  $c_{149} = x_1^{7}x_2^{3}x_3^{2^d-7}x_4^{4}x_5^{2^d-8}$ & &$c_{150} = x_1^{7}x_2^{2^d-5}x_3x_4^{4}x_5^{2^d-8}$\cr  
\end{tabular}}

\medskip
Thus, we obtain $|B_5((3)|(2)|^{d-1})| = 280 + 35 + 150 = 465$.

\subsection{A generating set for $(QP_5)_{62}$}\label{s54}\

\medskip
By Lemma \ref{bdd5} we have
$$(QP_5)_{62}\cong QP_5((2)|^5) \oplus QP_5((4)|(3)|^3|(1).$$
By Proposition \ref{mdd51}, $\dim QP_5((2)|^5) = 155$. The admissible monomials of weight vector $(2)|^5$ are given in Subsection \ref{s51}. 

We now describe a generating set for $QP_5((4)|(3)|^3|(1))$. We have $$B_5((4)|(3)|^3|(1)) = B_5^0((4)|(3)|^3|(1)\cup B_5^+((4)|(3)|^3|(1).$$
From our work \cite{su2}, $B_4((4)|(3)|^3|(1) = B_4^+((4)|(3)|^3|(1)$ is the set of 45 monomial:

\medskip
\centerline{\begin{tabular}{lllll}
$x_1x_2^{15}x_3^{15}x_4^{31}$& $x_1x_2^{15}x_3^{31}x_4^{15}$& $x_1x_2^{31}x_3^{15}x_4^{15}$& $x_1^{3}x_2^{13}x_3^{15}x_4^{31}$& $x_1^{3}x_2^{13}x_3^{31}x_4^{15} $\cr  $x_1^{3}x_2^{15}x_3^{13}x_4^{31}$& $x_1^{3}x_2^{15}x_3^{15}x_4^{29}$& $x_1^{3}x_2^{15}x_3^{29}x_4^{15}$& $x_1^{3}x_2^{15}x_3^{31}x_4^{13}$& $x_1^{3}x_2^{29}x_3^{15}x_4^{15} $\cr  $x_1^{3}x_2^{31}x_3^{13}x_4^{15}$& $x_1^{3}x_2^{31}x_3^{15}x_4^{13}$& $x_1^{7}x_2^{11}x_3^{13}x_4^{31}$& $x_1^{7}x_2^{11}x_3^{15}x_4^{29}$& $x_1^{7}x_2^{11}x_3^{29}x_4^{15} $\cr  $x_1^{7}x_2^{11}x_3^{31}x_4^{13}$& $x_1^{7}x_2^{15}x_3^{11}x_4^{29}$& $x_1^{7}x_2^{15}x_3^{27}x_4^{13}$& $x_1^{7}x_2^{27}x_3^{13}x_4^{15}$& $x_1^{7}x_2^{27}x_3^{15}x_4^{13} $\cr  $x_1^{7}x_2^{31}x_3^{11}x_4^{13}$& $x_1^{15}x_2x_3^{15}x_4^{31}$& $x_1^{15}x_2x_3^{31}x_4^{15}$& $x_1^{15}x_2^{3}x_3^{13}x_4^{31}$& $x_1^{15}x_2^{3}x_3^{15}x_4^{29} $\cr   $x_1^{15}x_2^{3}x_3^{29}x_4^{15}$& $x_1^{15}x_2^{3}x_3^{31}x_4^{13}$& $x_1^{15}x_2^{7}x_3^{11}x_4^{29}$& $x_1^{15}x_2^{7}x_3^{27}x_4^{13}$& $x_1^{15}x_2^{15}x_3x_4^{31} $\cr  $x_1^{15}x_2^{15}x_3^{3}x_4^{29}$& $x_1^{15}x_2^{15}x_3^{15}x_4^{17}$& $x_1^{15}x_2^{15}x_3^{19}x_4^{13}$& $x_1^{15}x_2^{15}x_3^{31}x_4$& $x_1^{15}x_2^{23}x_3^{11}x_4^{13} $\cr  $x_1^{15}x_2^{31}x_3x_4^{15}$& $x_1^{15}x_2^{31}x_3^{3}x_4^{13}$& $x_1^{15}x_2^{31}x_3^{15}x_4$& $x_1^{31}x_2x_3^{15}x_4^{15}$& $x_1^{31}x_2^{3}x_3^{13}x_4^{15} $\cr  $x_1^{31}x_2^{3}x_3^{15}x_4^{13}$& $x_1^{31}x_2^{7}x_3^{11}x_4^{13}$& $x_1^{31}x_2^{15}x_3x_4^{15}$& $x_1^{31}x_2^{15}x_3^{3}x_4^{13}$& $x_1^{31}x_2^{15}x_3^{15}x_4 $\cr    
\end{tabular}}

\medskip
Hence, $B_5^0((4)|(3)|^3|(1) = \bigcup_{i=1}^5f_i(B_4^+((4)|(3)|^3|(1))$ is the set of 225 monomials.

We set 
$$A(5) = \{x_i^{31}f_i(w): w \in B_4^+((3)|(2)|^{3}), 1\leqslant i \leqslant 5\},$$
where $B_4^+((3)|(2)|^{3})$ is the set of 46 monomials:

\medskip
\centerline{\begin{tabular}{llllll}
$x_1x_2x_3^{14}x_4^{15}$& $x_1x_2x_3^{15}x_4^{14}$& $x_1x_2^{2}x_3^{13}x_4^{15}$& $x_1x_2^{2}x_3^{15}x_4^{13}$& $x_1x_2^{3}x_3^{12}x_4^{15}$& $x_1x_2^{3}x_3^{13}x_4^{14} $\cr  $x_1x_2^{3}x_3^{14}x_4^{13}$& $x_1x_2^{3}x_3^{15}x_4^{12}$& $x_1x_2^{6}x_3^{11}x_4^{13}$& $x_1x_2^{7}x_3^{10}x_4^{13}$& $x_1x_2^{7}x_3^{11}x_4^{12}$& $x_1x_2^{14}x_3x_4^{15} $\cr  $x_1x_2^{14}x_3^{3}x_4^{13}$& $x_1x_2^{14}x_3^{15}x_4$& $x_1x_2^{15}x_3x_4^{14}$& $x_1x_2^{15}x_3^{2}x_4^{13}$& $x_1x_2^{15}x_3^{3}x_4^{12}$& $x_1x_2^{15}x_3^{14}x_4 $\cr  $x_1^{3}x_2x_3^{12}x_4^{15}$& $x_1^{3}x_2x_3^{13}x_4^{14}$& $x_1^{3}x_2x_3^{14}x_4^{13}$& $x_1^{3}x_2x_3^{15}x_4^{12}$& $x_1^{3}x_2^{3}x_3^{12}x_4^{13}$& $x_1^{3}x_2^{3}x_3^{13}x_4^{12} $\cr  $x_1^{3}x_2^{4}x_3^{11}x_4^{13}$& $x_1^{3}x_2^{5}x_3^{10}x_4^{13}$& $x_1^{3}x_2^{5}x_3^{11}x_4^{12}$& $x_1^{3}x_2^{7}x_3^{8}x_4^{13}$& $x_1^{3}x_2^{7}x_3^{9}x_4^{12}$& $x_1^{3}x_2^{13}x_3x_4^{14} $\cr  $x_1^{3}x_2^{13}x_3^{2}x_4^{13}$& $x_1^{3}x_2^{13}x_3^{3}x_4^{12}$& $x_1^{3}x_2^{13}x_3^{14}x_4$& $x_1^{3}x_2^{15}x_3x_4^{12}$& $x_1^{7}x_2x_3^{10}x_4^{13}$& $x_1^{7}x_2x_3^{11}x_4^{12} $\cr  $x_1^{7}x_2^{3}x_3^{8}x_4^{13}$& $x_1^{7}x_2^{3}x_3^{9}x_4^{12}$& $x_1^{7}x_2^{7}x_3^{8}x_4^{9}$& $x_1^{7}x_2^{7}x_3^{9}x_4^{8}$& $x_1^{7}x_2^{11}x_3x_4^{12}$& $x_1^{15}x_2x_3x_4^{14} $\cr  $x_1^{15}x_2x_3^{2}x_4^{13}$& $x_1^{15}x_2x_3^{3}x_4^{12}$& $x_1^{15}x_2x_3^{14}x_4$& $x_1^{15}x_2^{3}x_3x_4^{12}$ && \cr 
\end{tabular}}

\medskip
We have $|A(5)| = 230$. Denote
$$B(5) = \{x_i^{15}f_i(u): u \in B\subset B_4^+((3)|(2)|^{3}|(1)), 1\leqslant i \leqslant 5\},$$
where $B$ is the set of 64 monomials:

\medskip
\centerline{\begin{tabular}{llllll}
$x_1x_2x_3^{15}x_4^{30}$& $x_1x_2x_3^{30}x_4^{15}$& $x_1x_2^{2}x_3^{15}x_4^{29}$& $x_1x_2^{2}x_3^{29}x_4^{15}$& $x_1x_2^{3}x_3^{13}x_4^{30}$& $x_1x_2^{3}x_3^{14}x_4^{29} $\cr  $x_1x_2^{3}x_3^{15}x_4^{28}$& $x_1x_2^{3}x_3^{28}x_4^{15}$& $x_1x_2^{3}x_3^{29}x_4^{14}$& $x_1x_2^{3}x_3^{30}x_4^{13}$& $x_1x_2^{6}x_3^{11}x_4^{29}$& $x_1x_2^{6}x_3^{27}x_4^{13} $\cr  $x_1x_2^{7}x_3^{10}x_4^{29}$& $x_1x_2^{7}x_3^{11}x_4^{28}$& $x_1x_2^{7}x_3^{26}x_4^{13}$& $x_1x_2^{7}x_3^{27}x_4^{12}$& $x_1x_2^{14}x_3^{3}x_4^{29}$& $x_1x_2^{15}x_3x_4^{30} $\cr  $x_1x_2^{15}x_3^{2}x_4^{29}$& $x_1x_2^{15}x_3^{3}x_4^{28}$& $x_1x_2^{15}x_3^{30}x_4$& $x_1x_2^{30}x_3x_4^{15}$& $x_1x_2^{30}x_3^{3}x_4^{13}$& $x_1x_2^{30}x_3^{15}x_4 $\cr  $x_1^{3}x_2x_3^{13}x_4^{30}$& $x_1^{3}x_2x_3^{14}x_4^{29}$& $x_1^{3}x_2x_3^{15}x_4^{28}$& $x_1^{3}x_2x_3^{28}x_4^{15}$& $x_1^{3}x_2x_3^{29}x_4^{14}$& $x_1^{3}x_2x_3^{30}x_4^{13} $\cr  $x_1^{3}x_2^{3}x_3^{12}x_4^{29}$& $x_1^{3}x_2^{3}x_3^{13}x_4^{28}$& $x_1^{3}x_2^{3}x_3^{28}x_4^{13}$& $x_1^{3}x_2^{3}x_3^{29}x_4^{12}$& $x_1^{3}x_2^{4}x_3^{11}x_4^{29}$& $x_1^{3}x_2^{4}x_3^{27}x_4^{13} $\cr  $x_1^{3}x_2^{5}x_3^{10}x_4^{29}$& $x_1^{3}x_2^{5}x_3^{11}x_4^{28}$& $x_1^{3}x_2^{5}x_3^{26}x_4^{13}$& $x_1^{3}x_2^{5}x_3^{27}x_4^{12}$& $x_1^{3}x_2^{7}x_3^{9}x_4^{28}$& $x_1^{3}x_2^{7}x_3^{25}x_4^{12} $\cr  $x_1^{3}x_2^{13}x_3x_4^{30}$& $x_1^{3}x_2^{13}x_3^{2}x_4^{29}$& $x_1^{3}x_2^{13}x_3^{3}x_4^{28}$& $x_1^{3}x_2^{13}x_3^{30}x_4$& $x_1^{3}x_2^{15}x_3x_4^{28}$& $x_1^{3}x_2^{29}x_3x_4^{14} $\cr  $x_1^{3}x_2^{29}x_3^{2}x_4^{13}$& $x_1^{3}x_2^{29}x_3^{3}x_4^{12}$& $x_1^{3}x_2^{29}x_3^{14}x_4$& $x_1^{7}x_2x_3^{10}x_4^{29}$& $x_1^{7}x_2x_3^{11}x_4^{28}$& $x_1^{7}x_2x_3^{26}x_4^{13} $\cr  $x_1^{7}x_2x_3^{27}x_4^{12}$& $x_1^{7}x_2^{3}x_3^{9}x_4^{28}$& $x_1^{7}x_2^{3}x_3^{25}x_4^{12}$& $x_1^{7}x_2^{11}x_3x_4^{28}$& $x_1^{7}x_2^{27}x_3x_4^{12}$& $x_1^{15}x_2x_3x_4^{30} $\cr  $x_1^{15}x_2x_3^{2}x_4^{29}$& $x_1^{15}x_2x_3^{3}x_4^{28}$& $x_1^{15}x_2x_3^{30}x_4$& $x_1^{15}x_2^{3}x_3x_4^{28}$& & \cr  
\end{tabular}}

\medskip
We have $|B(5)| = 270$. Using Proposition \ref{mdmo} we obtain 
$$A(5) \cup B(5) \subset B_5^+\big((3)|(3)|^3|(1)\big).$$ 

Denote by $C(5)$ the set of 225 monomials:

\medskip
\centerline{\begin{tabular}{lll}  
$1.\ \  x_1x_2^{7}x_3^{11}x_4^{13}x_5^{30}$& $2.\ \  x_1x_2^{7}x_3^{11}x_4^{14}x_5^{29}$& $3.\ \  x_1x_2^{7}x_3^{11}x_4^{29}x_5^{14} $\cr  $4.\ \  x_1x_2^{7}x_3^{11}x_4^{30}x_5^{13}$& $5.\ \  x_1x_2^{7}x_3^{14}x_4^{11}x_5^{29}$& $6.\ \  x_1x_2^{7}x_3^{14}x_4^{27}x_5^{13} $\cr  $7.\ \  x_1x_2^{7}x_3^{27}x_4^{13}x_5^{14}$& $8.\ \  x_1x_2^{7}x_3^{27}x_4^{14}x_5^{13}$& $9.\ \  x_1x_2^{7}x_3^{30}x_4^{11}x_5^{13} $\cr  $10.\ \  x_1x_2^{14}x_3^{7}x_4^{11}x_5^{29}$& $11.\ \  x_1x_2^{14}x_3^{7}x_4^{27}x_5^{13}$& $12.\ \  x_1x_2^{14}x_3^{23}x_4^{11}x_5^{13} $\cr  $13.\ \  x_1x_2^{15}x_3^{22}x_4^{11}x_5^{13}$& $14.\ \  x_1x_2^{15}x_3^{23}x_4^{11}x_5^{12}$& $15.\ \  x_1x_2^{30}x_3^{7}x_4^{11}x_5^{13} $\cr  $16.\ \  x_1^{3}x_2^{3}x_3^{13}x_4^{13}x_5^{30}$& $17.\ \  x_1^{3}x_2^{3}x_3^{13}x_4^{14}x_5^{29}$& $18.\ \  x_1^{3}x_2^{3}x_3^{13}x_4^{29}x_5^{14} $\cr  $19.\ \  x_1^{3}x_2^{3}x_3^{13}x_4^{30}x_5^{13}$& $20.\ \  x_1^{3}x_2^{3}x_3^{29}x_4^{13}x_5^{14}$& $21.\ \  x_1^{3}x_2^{3}x_3^{29}x_4^{14}x_5^{13} $\cr  $22.\ \  x_1^{3}x_2^{5}x_3^{11}x_4^{13}x_5^{30}$& $23.\ \  x_1^{3}x_2^{5}x_3^{11}x_4^{14}x_5^{29}$& $24.\ \  x_1^{3}x_2^{5}x_3^{11}x_4^{29}x_5^{14} $\cr  $25.\ \  x_1^{3}x_2^{5}x_3^{11}x_4^{30}x_5^{13}$& $26.\ \  x_1^{3}x_2^{5}x_3^{14}x_4^{11}x_5^{29}$& $27.\ \  x_1^{3}x_2^{5}x_3^{14}x_4^{27}x_5^{13} $\cr  $28.\ \  x_1^{3}x_2^{5}x_3^{27}x_4^{13}x_5^{14}$& $29.\ \  x_1^{3}x_2^{5}x_3^{27}x_4^{14}x_5^{13}$& $30.\ \  x_1^{3}x_2^{5}x_3^{30}x_4^{11}x_5^{13} $\cr  $31.\ \  x_1^{3}x_2^{7}x_3^{9}x_4^{13}x_5^{30}$& $32.\ \  x_1^{3}x_2^{7}x_3^{9}x_4^{14}x_5^{29}$& $33.\ \  x_1^{3}x_2^{7}x_3^{9}x_4^{29}x_5^{14} $\cr  $34.\ \  x_1^{3}x_2^{7}x_3^{9}x_4^{30}x_5^{13}$& $35.\ \  x_1^{3}x_2^{7}x_3^{11}x_4^{12}x_5^{29}$& $36.\ \  x_1^{3}x_2^{7}x_3^{11}x_4^{13}x_5^{28} $\cr  $37.\ \  x_1^{3}x_2^{7}x_3^{11}x_4^{28}x_5^{13}$& $38.\ \  x_1^{3}x_2^{7}x_3^{11}x_4^{29}x_5^{12}$& $39.\ \  x_1^{3}x_2^{7}x_3^{12}x_4^{11}x_5^{29} $\cr  $40.\ \  x_1^{3}x_2^{7}x_3^{12}x_4^{27}x_5^{13}$& $41.\ \  x_1^{3}x_2^{7}x_3^{13}x_4^{10}x_5^{29}$& $42.\ \  x_1^{3}x_2^{7}x_3^{13}x_4^{11}x_5^{28} $\cr  $43.\ \  x_1^{3}x_2^{7}x_3^{13}x_4^{26}x_5^{13}$& $44.\ \  x_1^{3}x_2^{7}x_3^{13}x_4^{27}x_5^{12}$& $45.\ \  x_1^{3}x_2^{7}x_3^{25}x_4^{13}x_5^{14} $\cr  $46.\ \  x_1^{3}x_2^{7}x_3^{25}x_4^{14}x_5^{13}$& $47.\ \  x_1^{3}x_2^{7}x_3^{27}x_4^{12}x_5^{13}$& $48.\ \  x_1^{3}x_2^{7}x_3^{27}x_4^{13}x_5^{12} $\cr  $49.\ \  x_1^{3}x_2^{7}x_3^{28}x_4^{11}x_5^{13}$& $50.\ \  x_1^{3}x_2^{7}x_3^{29}x_4^{10}x_5^{13}$& $51.\ \  x_1^{3}x_2^{7}x_3^{29}x_4^{11}x_5^{12} $\cr  $52.\ \  x_1^{3}x_2^{13}x_3^{3}x_4^{13}x_5^{30}$& $53.\ \  x_1^{3}x_2^{13}x_3^{3}x_4^{14}x_5^{29}$& $54.\ \  x_1^{3}x_2^{13}x_3^{3}x_4^{29}x_5^{14} $\cr  $55.\ \  x_1^{3}x_2^{13}x_3^{3}x_4^{30}x_5^{13}$& $56.\ \  x_1^{3}x_2^{13}x_3^{6}x_4^{11}x_5^{29}$& $57.\ \  x_1^{3}x_2^{13}x_3^{6}x_4^{27}x_5^{13} $\cr  $58.\ \  x_1^{3}x_2^{13}x_3^{7}x_4^{10}x_5^{29}$& $59.\ \  x_1^{3}x_2^{13}x_3^{7}x_4^{11}x_5^{28}$& $60.\ \  x_1^{3}x_2^{13}x_3^{7}x_4^{26}x_5^{13} $\cr  $61.\ \  x_1^{3}x_2^{13}x_3^{7}x_4^{27}x_5^{12}$& $62.\ \  x_1^{3}x_2^{13}x_3^{14}x_4^{3}x_5^{29}$& $63.\ \  x_1^{3}x_2^{13}x_3^{22}x_4^{11}x_5^{13} $\cr  $64.\ \  x_1^{3}x_2^{13}x_3^{23}x_4^{10}x_5^{13}$& $65.\ \  x_1^{3}x_2^{13}x_3^{23}x_4^{11}x_5^{12}$& $66.\ \  x_1^{3}x_2^{13}x_3^{30}x_4^{3}x_5^{13} $\cr  $67.\ \  x_1^{3}x_2^{15}x_3^{21}x_4^{10}x_5^{13}$& $68.\ \  x_1^{3}x_2^{15}x_3^{21}x_4^{11}x_5^{12}$& $69.\ \  x_1^{3}x_2^{15}x_3^{23}x_4^{9}x_5^{12} $\cr  $70.\ \  x_1^{3}x_2^{29}x_3^{3}x_4^{13}x_5^{14}$& $71.\ \  x_1^{3}x_2^{29}x_3^{3}x_4^{14}x_5^{13}$& $72.\ \  x_1^{3}x_2^{29}x_3^{6}x_4^{11}x_5^{13} $\cr  $73.\ \  x_1^{3}x_2^{29}x_3^{7}x_4^{10}x_5^{13}$& $74.\ \  x_1^{3}x_2^{29}x_3^{7}x_4^{11}x_5^{12}$& $75.\ \  x_1^{3}x_2^{29}x_3^{14}x_4^{3}x_5^{13} $\cr $76.\ \  x_1^{7}x_2x_3^{11}x_4^{13}x_5^{30}$& $77.\ \  x_1^{7}x_2x_3^{11}x_4^{14}x_5^{29}$& $78.\ \  x_1^{7}x_2x_3^{11}x_4^{29}x_5^{14} $\cr  $79.\ \  x_1^{7}x_2x_3^{11}x_4^{30}x_5^{13}$& $80.\ \  x_1^{7}x_2x_3^{14}x_4^{11}x_5^{29}$& $81.\ \  x_1^{7}x_2x_3^{14}x_4^{27}x_5^{13} $\cr  $82.\ \  x_1^{7}x_2x_3^{27}x_4^{13}x_5^{14}$& $83.\ \  x_1^{7}x_2x_3^{27}x_4^{14}x_5^{13}$& $84.\ \  x_1^{7}x_2x_3^{30}x_4^{11}x_5^{13} $\cr  $85.\ \  x_1^{7}x_2^{3}x_3^{9}x_4^{13}x_5^{30}$& $86.\ \  x_1^{7}x_2^{3}x_3^{9}x_4^{14}x_5^{29}$& $87.\ \  x_1^{7}x_2^{3}x_3^{9}x_4^{29}x_5^{14} $\cr  $88.\ \  x_1^{7}x_2^{3}x_3^{9}x_4^{30}x_5^{13}$& $89.\ \  x_1^{7}x_2^{3}x_3^{11}x_4^{12}x_5^{29}$& $90.\ \  x_1^{7}x_2^{3}x_3^{11}x_4^{13}x_5^{28} $\cr  $91.\ \  x_1^{7}x_2^{3}x_3^{11}x_4^{28}x_5^{13}$& $92.\ \  x_1^{7}x_2^{3}x_3^{11}x_4^{29}x_5^{12}$& $93.\ \  x_1^{7}x_2^{3}x_3^{12}x_4^{11}x_5^{29} $\cr  $94.\ \  x_1^{7}x_2^{3}x_3^{12}x_4^{27}x_5^{13}$& $95.\ \  x_1^{7}x_2^{3}x_3^{13}x_4^{10}x_5^{29}$& $96.\ \  x_1^{7}x_2^{3}x_3^{13}x_4^{11}x_5^{28} $\cr  $97.\ \  x_1^{7}x_2^{3}x_3^{13}x_4^{26}x_5^{13}$& $98.\ \  x_1^{7}x_2^{3}x_3^{13}x_4^{27}x_5^{12}$& $99.\ \  x_1^{7}x_2^{3}x_3^{25}x_4^{13}x_5^{14} $\cr \end{tabular}}
\centerline{\begin{tabular}{lll}  $100.\ \  x_1^{7}x_2^{3}x_3^{25}x_4^{14}x_5^{13}$& $101.\ \  x_1^{7}x_2^{3}x_3^{27}x_4^{12}x_5^{13}$& $102.\ \  x_1^{7}x_2^{3}x_3^{27}x_4^{13}x_5^{12} $\cr  $103.\ \  x_1^{7}x_2^{3}x_3^{28}x_4^{11}x_5^{13}$& $104.\ \  x_1^{7}x_2^{3}x_3^{29}x_4^{10}x_5^{13}$& $105.\ \  x_1^{7}x_2^{3}x_3^{29}x_4^{11}x_5^{12} $\cr  $106.\ \  x_1^{7}x_2^{7}x_3^{8}x_4^{11}x_5^{29}$& $107.\ \  x_1^{7}x_2^{7}x_3^{8}x_4^{27}x_5^{13}$& $108.\ \  x_1^{7}x_2^{7}x_3^{9}x_4^{10}x_5^{29} $\cr  $109.\ \  x_1^{7}x_2^{7}x_3^{9}x_4^{11}x_5^{28}$& $110.\ \  x_1^{7}x_2^{7}x_3^{9}x_4^{26}x_5^{13}$& $111.\ \  x_1^{7}x_2^{7}x_3^{9}x_4^{27}x_5^{12} $\cr  $112.\ \  x_1^{7}x_2^{7}x_3^{11}x_4^{8}x_5^{29}$& $113.\ \  x_1^{7}x_2^{7}x_3^{11}x_4^{9}x_5^{28}$& $114.\ \  x_1^{7}x_2^{7}x_3^{11}x_4^{13}x_5^{24} $\cr  $115.\ \  x_1^{7}x_2^{7}x_3^{11}x_4^{24}x_5^{13}$& $116.\ \  x_1^{7}x_2^{7}x_3^{11}x_4^{25}x_5^{12}$& $117.\ \  x_1^{7}x_2^{7}x_3^{11}x_4^{29}x_5^{8} $\cr  $118.\ \  x_1^{7}x_2^{7}x_3^{24}x_4^{11}x_5^{13}$& $119.\ \  x_1^{7}x_2^{7}x_3^{25}x_4^{10}x_5^{13}$& $120.\ \  x_1^{7}x_2^{7}x_3^{25}x_4^{11}x_5^{12} $\cr  $121.\ \  x_1^{7}x_2^{7}x_3^{27}x_4^{8}x_5^{13}$& $122.\ \  x_1^{7}x_2^{7}x_3^{27}x_4^{9}x_5^{12}$& $123.\ \  x_1^{7}x_2^{7}x_3^{27}x_4^{13}x_5^{8} $\cr  $124.\ \  x_1^{7}x_2^{11}x_3x_4^{13}x_5^{30}$& $125.\ \  x_1^{7}x_2^{11}x_3x_4^{14}x_5^{29}$& $126.\ \  x_1^{7}x_2^{11}x_3x_4^{29}x_5^{14} $\cr  $127.\ \  x_1^{7}x_2^{11}x_3x_4^{30}x_5^{13}$& $128.\ \  x_1^{7}x_2^{11}x_3^{3}x_4^{12}x_5^{29}$& $129.\ \  x_1^{7}x_2^{11}x_3^{3}x_4^{13}x_5^{28} $\cr  $130.\ \  x_1^{7}x_2^{11}x_3^{3}x_4^{28}x_5^{13}$& $131.\ \  x_1^{7}x_2^{11}x_3^{3}x_4^{29}x_5^{12}$& $132.\ \  x_1^{7}x_2^{11}x_3^{4}x_4^{11}x_5^{29} $\cr  $133.\ \  x_1^{7}x_2^{11}x_3^{4}x_4^{27}x_5^{13}$& $134.\ \  x_1^{7}x_2^{11}x_3^{5}x_4^{10}x_5^{29}$& $135.\ \  x_1^{7}x_2^{11}x_3^{5}x_4^{11}x_5^{28} $\cr  $136.\ \  x_1^{7}x_2^{11}x_3^{5}x_4^{26}x_5^{13}$& $137.\ \  x_1^{7}x_2^{11}x_3^{5}x_4^{27}x_5^{12}$& $138.\ \  x_1^{7}x_2^{11}x_3^{7}x_4^{9}x_5^{28} $\cr  $139.\ \  x_1^{7}x_2^{11}x_3^{7}x_4^{25}x_5^{12}$& $140.\ \  x_1^{7}x_2^{11}x_3^{13}x_4x_5^{30}$& $141.\ \  x_1^{7}x_2^{11}x_3^{13}x_4^{2}x_5^{29} $\cr  $142.\ \  x_1^{7}x_2^{11}x_3^{13}x_4^{3}x_5^{28}$& $143.\ \  x_1^{7}x_2^{11}x_3^{13}x_4^{30}x_5$& $144.\ \  x_1^{7}x_2^{11}x_3^{21}x_4^{10}x_5^{13} $\cr  $145.\ \  x_1^{7}x_2^{11}x_3^{21}x_4^{11}x_5^{12}$& $146.\ \  x_1^{7}x_2^{11}x_3^{23}x_4^{9}x_5^{12}$& $147.\ \  x_1^{7}x_2^{11}x_3^{29}x_4x_5^{14} $\cr  $148.\ \  x_1^{7}x_2^{11}x_3^{29}x_4^{2}x_5^{13}$& $149.\ \  x_1^{7}x_2^{11}x_3^{29}x_4^{3}x_5^{12}$& $150.\ \  x_1^{7}x_2^{11}x_3^{29}x_4^{14}x_5 $\cr   
$151.\ \  x_1^{7}x_2^{15}x_3^{19}x_4^{9}x_5^{12}$& $152.\ \  x_1^{7}x_2^{27}x_3x_4^{13}x_5^{14}$& $153.\ \  x_1^{7}x_2^{27}x_3x_4^{14}x_5^{13} $\cr  $154.\ \  x_1^{7}x_2^{27}x_3^{3}x_4^{12}x_5^{13}$& $155.\ \  x_1^{7}x_2^{27}x_3^{3}x_4^{13}x_5^{12}$& $156.\ \  x_1^{7}x_2^{27}x_3^{4}x_4^{11}x_5^{13} $\cr  $157.\ \  x_1^{7}x_2^{27}x_3^{5}x_4^{10}x_5^{13}$& $158.\ \  x_1^{7}x_2^{27}x_3^{5}x_4^{11}x_5^{12}$& $159.\ \  x_1^{7}x_2^{27}x_3^{7}x_4^{9}x_5^{12} $\cr  $160.\ \  x_1^{7}x_2^{27}x_3^{13}x_4x_5^{14}$& $161.\ \  x_1^{7}x_2^{27}x_3^{13}x_4^{2}x_5^{13}$& $162.\ \  x_1^{7}x_2^{27}x_3^{13}x_4^{3}x_5^{12} $\cr  $163.\ \  x_1^{7}x_2^{27}x_3^{13}x_4^{14}x_5$& $164.\ \  x_1^{15}x_2x_3^{22}x_4^{11}x_5^{13}$& $165.\ \  x_1^{15}x_2x_3^{23}x_4^{11}x_5^{12} $\cr  $166.\ \  x_1^{15}x_2^{3}x_3^{21}x_4^{10}x_5^{13}$& $167.\ \  x_1^{15}x_2^{3}x_3^{21}x_4^{11}x_5^{12}$& $168.\ \  x_1^{15}x_2^{3}x_3^{23}x_4^{9}x_5^{12} $\cr  $169.\ \  x_1^{15}x_2^{7}x_3^{19}x_4^{9}x_5^{12}$& $170.\ \  x_1^{15}x_2^{23}x_3x_4^{11}x_5^{12}$& $171.\ \  x_1^{15}x_2^{23}x_3^{3}x_4^{9}x_5^{12} $\cr  $172.\ \  x_1^{15}x_2^{23}x_3^{11}x_4x_5^{12}$& $173.\ \  x_1x_2^{14}x_3^{15}x_4^{15}x_5^{17}$& $174.\ \  x_1x_2^{14}x_3^{15}x_4^{19}x_5^{13} $\cr  $175.\ \  x_1x_2^{15}x_3^{14}x_4^{15}x_5^{17}$& $176.\ \  x_1x_2^{15}x_3^{14}x_4^{19}x_5^{13}$& $177.\ \  x_1x_2^{15}x_3^{15}x_4^{14}x_5^{17} $\cr  $178.\ \  x_1x_2^{15}x_3^{15}x_4^{15}x_5^{16}$& $179.\ \  x_1x_2^{15}x_3^{15}x_4^{18}x_5^{13}$& $180.\ \  x_1x_2^{15}x_3^{15}x_4^{19}x_5^{12} $\cr  $181.\ \  x_1x_2^{15}x_3^{23}x_4^{10}x_5^{13}$& $182.\ \  x_1^{3}x_2^{13}x_3^{14}x_4^{15}x_5^{17}$& $183.\ \  x_1^{3}x_2^{13}x_3^{14}x_4^{19}x_5^{13} $\cr  $184.\ \  x_1^{3}x_2^{13}x_3^{15}x_4^{14}x_5^{17}$& $185.\ \  x_1^{3}x_2^{13}x_3^{15}x_4^{15}x_5^{16}$& $186.\ \  x_1^{3}x_2^{13}x_3^{15}x_4^{18}x_5^{13} $\cr  $187.\ \  x_1^{3}x_2^{13}x_3^{15}x_4^{19}x_5^{12}$& $188.\ \  x_1^{3}x_2^{15}x_3^{13}x_4^{14}x_5^{17}$& $189.\ \  x_1^{3}x_2^{15}x_3^{13}x_4^{15}x_5^{16} $\cr  $190.\ \  x_1^{3}x_2^{15}x_3^{13}x_4^{18}x_5^{13}$& $191.\ \  x_1^{3}x_2^{15}x_3^{13}x_4^{19}x_5^{12}$& $192.\ \  x_1^{3}x_2^{15}x_3^{15}x_4^{13}x_5^{16} $\cr  $193.\ \  x_1^{3}x_2^{15}x_3^{15}x_4^{17}x_5^{12}$& $194.\ \  x_1^{7}x_2^{11}x_3^{13}x_4^{14}x_5^{17}$& $195.\ \  x_1^{7}x_2^{11}x_3^{13}x_4^{15}x_5^{16} $\cr  $196.\ \  x_1^{7}x_2^{11}x_3^{13}x_4^{18}x_5^{13}$& $197.\ \  x_1^{7}x_2^{11}x_3^{13}x_4^{19}x_5^{12}$& $198.\ \  x_1^{7}x_2^{11}x_3^{15}x_4^{13}x_5^{16} $\cr  $199.\ \  x_1^{7}x_2^{11}x_3^{15}x_4^{17}x_5^{12}$& $200.\ \  x_1^{7}x_2^{15}x_3^{11}x_4^{17}x_5^{12}$& $201.\ \  x_1^{15}x_2x_3^{14}x_4^{15}x_5^{17} $\cr  $202.\ \  x_1^{15}x_2x_3^{14}x_4^{19}x_5^{13}$& $203.\ \  x_1^{15}x_2x_3^{15}x_4^{14}x_5^{17}$& $204.\ \  x_1^{15}x_2x_3^{15}x_4^{15}x_5^{16} $\cr  $205.\ \  x_1^{15}x_2x_3^{15}x_4^{18}x_5^{13}$& $206.\ \  x_1^{15}x_2x_3^{15}x_4^{19}x_5^{12}$& $207.\ \  x_1^{15}x_2x_3^{23}x_4^{10}x_5^{13} $\cr  $208.\ \  x_1^{15}x_2^{3}x_3^{13}x_4^{14}x_5^{17}$& $209.\ \  x_1^{15}x_2^{3}x_3^{13}x_4^{15}x_5^{16}$& $210.\ \  x_1^{15}x_2^{3}x_3^{13}x_4^{18}x_5^{13} $\cr  $211.\ \  x_1^{15}x_2^{3}x_3^{13}x_4^{19}x_5^{12}$& $212.\ \  x_1^{15}x_2^{3}x_3^{15}x_4^{13}x_5^{16}$& $213.\ \  x_1^{15}x_2^{3}x_3^{15}x_4^{17}x_5^{12} $\cr  $214.\ \  x_1^{15}x_2^{7}x_3^{11}x_4^{17}x_5^{12}$& $215.\ \  x_1^{15}x_2^{15}x_3x_4^{14}x_5^{17}$& $216.\ \  x_1^{15}x_2^{15}x_3x_4^{15}x_5^{16} $\cr  $217.\ \  x_1^{15}x_2^{15}x_3x_4^{18}x_5^{13}$& $218.\ \  x_1^{15}x_2^{15}x_3x_4^{19}x_5^{12}$& $219.\ \  x_1^{15}x_2^{15}x_3^{3}x_4^{13}x_5^{16} $\cr  $220.\ \  x_1^{15}x_2^{15}x_3^{3}x_4^{17}x_5^{12}$& $221.\ \  x_1^{15}x_2^{15}x_3^{15}x_4x_5^{16}$& $222.\ \  x_1^{15}x_2^{15}x_3^{19}x_4x_5^{12} $\cr  $223.\ \  x_1^{15}x_2^{19}x_3^{5}x_4^{10}x_5^{13}$& $224.\ \  x_1^{15}x_2^{19}x_3^{7}x_4^{9}x_5^{12}$& $225.\ \  x_1^{15}x_2^{23}x_3x_4^{10}x_5^{13} $\cr   
\end{tabular}}

\medskip
Denote by $D(5)$ the set of 70 monomials:

\medskip
\centerline{\begin{tabular}{lll} 
$226.\ \  x_1^{3}x_2^{7}x_3^{8}x_4^{15}x_5^{29}$& $227.\ \  x_1^{3}x_2^{7}x_3^{15}x_4^{8}x_5^{29}$& $228.\ \  x_1^{3}x_2^{7}x_3^{15}x_4^{24}x_5^{13} $\cr  $229.\ \  x_1^{3}x_2^{15}x_3^{7}x_4^{8}x_5^{29}$& $230.\ \  x_1^{3}x_2^{15}x_3^{7}x_4^{24}x_5^{13}$& $231.\ \  x_1^{3}x_2^{15}x_3^{15}x_4^{16}x_5^{13} $\cr  $232.\ \  x_1^{3}x_2^{15}x_3^{23}x_4^{8}x_5^{13}$& $233.\ \  x_1^{7}x_2^{3}x_3^{8}x_4^{15}x_5^{29}$& $234.\ \  x_1^{7}x_2^{3}x_3^{15}x_4^{8}x_5^{29} $\cr  $235.\ \  x_1^{7}x_2^{3}x_3^{15}x_4^{24}x_5^{13}$& $236.\ \  x_1^{7}x_2^{7}x_3^{8}x_4^{15}x_5^{25}$& $237.\ \  x_1^{7}x_2^{7}x_3^{9}x_4^{9}x_5^{30} $\cr  $238.\ \  x_1^{7}x_2^{7}x_3^{9}x_4^{15}x_5^{24}$& $239.\ \  x_1^{7}x_2^{7}x_3^{11}x_4^{12}x_5^{25}$& $240.\ \  x_1^{7}x_2^{7}x_3^{11}x_4^{28}x_5^{9} $\cr \end{tabular}}
\centerline{\begin{tabular}{lll}  $241.\ \  x_1^{7}x_2^{7}x_3^{15}x_4^{8}x_5^{25}$& $242.\ \  x_1^{7}x_2^{7}x_3^{15}x_4^{9}x_5^{24}$& $243.\ \  x_1^{7}x_2^{7}x_3^{15}x_4^{24}x_5^{9} $\cr  $244.\ \  x_1^{7}x_2^{7}x_3^{15}x_4^{25}x_5^{8}$& $245.\ \  x_1^{7}x_2^{7}x_3^{27}x_4^{12}x_5^{9}$& $246.\ \  x_1^{7}x_2^{9}x_3^{7}x_4^{10}x_5^{29} $\cr  $247.\ \  x_1^{7}x_2^{9}x_3^{7}x_4^{11}x_5^{28}$& $248.\ \  x_1^{7}x_2^{9}x_3^{7}x_4^{26}x_5^{13}$& $249.\ \  x_1^{7}x_2^{9}x_3^{7}x_4^{27}x_5^{12} $\cr  $250.\ \  x_1^{7}x_2^{9}x_3^{23}x_4^{10}x_5^{13}$& $251.\ \  x_1^{7}x_2^{9}x_3^{23}x_4^{11}x_5^{12}$& $252.\ \  x_1^{7}x_2^{11}x_3^{7}x_4^{8}x_5^{29} $\cr  $253.\ \  x_1^{7}x_2^{11}x_3^{7}x_4^{24}x_5^{13}$& $254.\ \  x_1^{7}x_2^{11}x_3^{15}x_4^{16}x_5^{13}$& $255.\ \  x_1^{7}x_2^{11}x_3^{23}x_4^{8}x_5^{13} $\cr  $256.\ \  x_1^{7}x_2^{15}x_3^{3}x_4^{8}x_5^{29}$& $257.\ \  x_1^{7}x_2^{15}x_3^{3}x_4^{24}x_5^{13}$& $258.\ \  x_1^{7}x_2^{15}x_3^{7}x_4^{8}x_5^{25} $\cr  $259.\ \  x_1^{7}x_2^{15}x_3^{7}x_4^{9}x_5^{24}$& $260.\ \  x_1^{7}x_2^{15}x_3^{7}x_4^{24}x_5^{9}$& $261.\ \  x_1^{7}x_2^{15}x_3^{7}x_4^{25}x_5^{8} $\cr  $262.\ \  x_1^{7}x_2^{15}x_3^{11}x_4^{13}x_5^{16}$& $263.\ \  x_1^{7}x_2^{15}x_3^{11}x_4^{16}x_5^{13}$& $264.\ \  x_1^{7}x_2^{15}x_3^{17}x_4^{10}x_5^{13} $\cr  $265.\ \  x_1^{7}x_2^{15}x_3^{17}x_4^{11}x_5^{12}$& $266.\ \  x_1^{7}x_2^{15}x_3^{19}x_4^{8}x_5^{13}$& $267.\ \  x_1^{7}x_2^{15}x_3^{23}x_4^{8}x_5^{9} $\cr  $268.\ \  x_1^{7}x_2^{15}x_3^{23}x_4^{9}x_5^{8}$& $269.\ \  x_1^{7}x_2^{27}x_3^{7}x_4^{8}x_5^{13}$& $270.\ \  x_1^{15}x_2^{3}x_3^{7}x_4^{8}x_5^{29} $\cr  $271.\ \  x_1^{15}x_2^{3}x_3^{7}x_4^{24}x_5^{13}$& $272.\ \  x_1^{15}x_2^{3}x_3^{15}x_4^{16}x_5^{13}$& $273.\ \  x_1^{15}x_2^{3}x_3^{23}x_4^{8}x_5^{13} $\cr  $274.\ \  x_1^{15}x_2^{7}x_3^{3}x_4^{8}x_5^{29}$& $275.\ \  x_1^{15}x_2^{7}x_3^{3}x_4^{24}x_5^{13}$& $276.\ \  x_1^{15}x_2^{7}x_3^{7}x_4^{8}x_5^{25} $\cr  $277.\ \  x_1^{15}x_2^{7}x_3^{7}x_4^{9}x_5^{24}$& $278.\ \  x_1^{15}x_2^{7}x_3^{7}x_4^{24}x_5^{9}$& $279.\ \  x_1^{15}x_2^{7}x_3^{7}x_4^{25}x_5^{8} $\cr  $280.\ \  x_1^{15}x_2^{7}x_3^{11}x_4^{13}x_5^{16}$& $281.\ \  x_1^{15}x_2^{7}x_3^{11}x_4^{16}x_5^{13}$& $282.\ \  x_1^{15}x_2^{7}x_3^{17}x_4^{10}x_5^{13} $\cr  $283.\ \  x_1^{15}x_2^{7}x_3^{17}x_4^{11}x_5^{12}$& $284.\ \  x_1^{15}x_2^{7}x_3^{19}x_4^{8}x_5^{13}$& $285.\ \  x_1^{15}x_2^{7}x_3^{23}x_4^{8}x_5^{9} $\cr  $286.\ \  x_1^{15}x_2^{7}x_3^{23}x_4^{9}x_5^{8}$& $287.\ \  x_1^{15}x_2^{15}x_3^{3}x_4^{16}x_5^{13}$& $288.\ \  x_1^{15}x_2^{15}x_3^{15}x_4^{16}x_5 $\cr  $289.\ \  x_1^{15}x_2^{15}x_3^{16}x_4^{3}x_5^{13}$& $290.\ \  x_1^{15}x_2^{15}x_3^{17}x_4^{2}x_5^{13}$& $291.\ \  x_1^{15}x_2^{15}x_3^{17}x_4^{3}x_5^{12} $\cr  $292.\ \  x_1^{15}x_2^{19}x_3^{5}x_4^{11}x_5^{12}$& $293.\ \  x_1^{15}x_2^{23}x_3^{3}x_4^{8}x_5^{13}$& $294.\ \  x_1^{15}x_2^{23}x_3^{7}x_4^{8}x_5^{9} $\cr  $295.\ \  x_1^{15}x_2^{23}x_3^{7}x_4^{9}x_5^{8}$& & \cr
\end{tabular}}
\medskip
By a direct computation one gets the following.
\begin{lems}\label{bda81} The following monomials are strictly inadmissible:

\medskip
\centerline{\begin{tabular}{lllll}  
$x_1^{3}x_2^{7}x_3^{8}x_4^{29}x_5^{15}$& $x_1^{3}x_2^{7}x_3^{24}x_4^{13}x_5^{15}$& $x_1^{3}x_2^{7}x_3^{24}x_4^{15}x_5^{13}$& $x_1^{3}x_2^{15}x_3^{20}x_4^{11}x_5^{13} $\cr  $x_1^{7}x_2^{3}x_3^{8}x_4^{29}x_5^{15}$& $x_1^{7}x_2^{3}x_3^{24}x_4^{13}x_5^{15}$& $x_1^{7}x_2^{3}x_3^{24}x_4^{15}x_5^{13}$& $x_1^{7}x_2^{11}x_3^{13}x_4^{16}x_5^{15} $\cr  $x_1^{7}x_2^{11}x_3^{13}x_4^{17}x_5^{14}$& $x_1^{7}x_2^{11}x_3^{20}x_4^{11}x_5^{13}$& $x_1^{7}x_2^{15}x_3^{16}x_4^{11}x_5^{13}$& $x_1^{15}x_2^{3}x_3^{20}x_4^{11}x_5^{13} $\cr  $x_1^{15}x_2^{7}x_3^{16}x_4^{11}x_5^{13}$& $x_1^{15}x_2^{19}x_3^{4}x_4^{11}x_5^{13}$& $x_1^{15}x_2^{19}x_3^{7}x_4^{8}x_5^{13}$& \cr   
\end{tabular}}
\end{lems}

Based on Lemmas \ref{bda62}, \ref{bda81} and Theorem \ref{dlcb1} we can show that if a monomial $y \in B_5((3)|^3|(1))$, such that $X_iy^2 \in P_5^+$ and $X_iy^2 \notin A(5)\cup B(5)\cup C(5)\cup D(5)$, then $X_iy^2$ is inadmissible, where  $1 \leqslant i \leqslant 5$. This implies  
$$B_5^+((4)|(3)|^3|(1)) \subset A(5)\cup B(5)\cup C(5)\cup D(5).$$ 
Hence, we obtain 
\begin{align*}
&\dim QP_5^+((4)|(3)|^3|(1)) \leqslant 230 + 270 + 295 = 795,\\
&\dim (QP_5)_{62} \leqslant 155+ 225 + 795 = 1175.
\end{align*}

Furthermore, we can give a lower bound for $\dim (QP_5)_{62}$ as follows. 

Set $\omega^* := (4)|(3)|^3|(1)$ and consider the subspaces 
$$\langle [A(5)]_{\omega^*}\rangle \subset QP_5(\omega^*),\ \langle [B(5)\cup C(d)]_{\omega^*}\rangle \subset QP_5(\omega^*).$$ Observe that for any $x\in A(5)$, we have $x = x_i^{31}f_i(y)$ with $y$ an admissible monomial of weight vector $(3)|(2)|^3$ in $P_4$. By Proposition \ref{mdmo}, $x$ is admissible and we get $\dim \langle [A(5)]_{\omega^*}\rangle = 230$. Since $\nu(x) = 31$ for all $x\in A(5)$ and $\nu(x) < 31$ for all $x\in B(5)\cup C(5)$, we obtain $\langle [A(5)]_{\omega^*}\rangle \cap \langle [B(5)\cup C(5)]_{\omega^*}\rangle = \{0\}$. By a direct computation we get the following.

\begin{props} 
The set $[B(5)\cup C(5)]_{\omega^*}$ is linearly independent in the vector space $QP_5(\omega^*)$.
\end{props}

So, we obtain a lower bound 
$$\dim (QP_5)_{62} \geqslant 155+ 225 + 725 = 1105.$$

\subsection{A generating set for $(QP_5)_{63}$}\label{s55}\

\medskip
By a similar argument as in the proof of Lemma \ref{bdd62} we have
\begin{lems}\label{bdbs1}  Let $x$ be an admissible monomial of degree $63$ in $P_5$ such that $[x] \in \mbox{\rm Ker}\big((\widetilde{Sq}^0_*)_{(5,29)}\big)$. Then, either $\omega(x) = (1)|^{6}$ or $\omega(x) = (3)|(2)|^{4}$ or $\omega(x) = (3,4,3,3,1)$.
\end{lems}
By Propositions \ref{mdd70} and \ref{mdd71}, we have $\dim QP_5((1)|^6) = 31$, $\dim QP_5((3)|(2)|^{4}) = 465.$ 
We need to determine $QP_5(3,4,3,3,1).$

\begin{lems}\label{bdbs2}  Let $(i,j,t,u,v)$ be an arbitrary permutation of $(1,2,3,4,5)$. The following monomials are strictly inadmissible:
	
\medskip
\centerline{\begin{tabular}{llll}  
$x_ix_j^{7}x_t^{11}x_u^{14}x_v^{30}$& $x_ix_j^{7}x_t^{27}x_u^{14}x_v^{14}$& $x_i^{3}x_j^{5}x_t^{11}x_u^{14}x_v^{30}$& $x_i^{3}x_j^{5}x_t^{27}x_u^{14}x_v^{14} $\cr  $x_i^{3}x_j^{7}x_t^{9}x_u^{14}x_v^{30}$& $x_i^{3}x_j^{7}x_t^{13}x_u^{10}x_v^{30}$& $x_i^{3}x_j^{7}x_t^{29}x_u^{10}x_v^{14}$& $x_i^{7}x_j^{11}x_t^{5}x_u^{10}x_v^{30} $\cr  $x_i^{7}x_j^{11}x_t^{13}x_u^{2}x_v^{30}$& $x_i^{7}x_j^{11}x_t^{29}x_u^{2}x_v^{14}$& $x_i^{7}x_j^{11}x_t^{29}x_u^{6}x_v^{10}$& $x_i^{7}x_j^{27}x_t^{5}x_u^{10}x_v^{14} $\cr  $x_i^{7}x_j^{27}x_t^{13}x_u^{2}x_v^{14}$& $x_i^{7}x_j^{27}x_t^{13}x_u^{6}x_v^{10}$& & \cr 
\end{tabular}}
\end{lems}
\begin{proof} We prove the lemma for $x = x_i^{7}x_j^{27}x_t^{13}x_u^{6}x_v^{10}$. We have
\begin{align*}
x &= Sq^1\big(x_i^{7}x_j^{15}x_t^{11}x_u^{9}x_v^{20} + x_i^{7}x_j^{15}x_t^{11}x_u^{12}x_v^{17} + x_i^{7}x_j^{15}x_t^{19}x_u^{9}x_v^{12} + x_i^{7}x_j^{15}x_t^{19}x_u^{12}x_v^{9}\\ &\quad + x_i^{7}x_j^{15}x_t^{23}x_u^{5}x_v^{12} + x_i^{7}x_j^{15}x_t^{23}x_u^{8}x_v^{9} + x_i^{7}x_j^{23}x_t^{11}x_u^{9}x_v^{12} + x_i^{7}x_j^{23}x_t^{11}x_u^{12}x_v^{9}\\ &\quad  + x_i^{9}x_j^{15}x_t^{23}x_u^{3}x_v^{12} + x_i^{9}x_j^{15}x_t^{23}x_u^{5}x_v^{10} + x_i^{12}x_j^{15}x_t^{23}x_u^{3}x_v^{9} + x_i^{19}x_j^{15}x_t^{7}x_u^{9}x_v^{12}\\ &\quad  + x_i^{19}x_j^{15}x_t^{7}x_u^{12}x_v^{9} + x_i^{19}x_j^{15}x_t^{11}x_u^{5}x_v^{12} + x_i^{23}x_j^{15}x_t^{7}x_u^{5}x_v^{12} + x_i^{23}x_j^{15}x_t^{7}x_u^{8}x_v^{9}\\ &\quad  + x_i^{23}x_j^{15}x_t^{9}x_u^{3}x_v^{12} + x_i^{23}x_j^{15}x_t^{9}x_u^{5}x_v^{10} + x_i^{23}x_j^{15}x_t^{12}x_u^{3}x_v^{9}\big) + Sq^2\big(x_i^{7}x_j^{15}x_t^{11}x_u^{10}x_v^{18}\\ &\quad  + x_i^{7}x_j^{15}x_t^{19}x_u^{10}x_v^{10} + x_1^{7}x_2^{15}x_3^{23}x_4^{6}x_5^{10} + x_1^{7}x_2^{23}x_3^{11}x_4^{10}x_5^{10} + x_1^{10}x_2^{15}x_3^{23}x_4^{3}x_5^{10}\\ &\quad  + x_1^{19}x_2^{15}x_3^{7}x_4^{10}x_5^{10} + x_1^{19}x_2^{15}x_3^{11}x_4^{6}x_5^{10} + x_1^{23}x_2^{15}x_3^{7}x_4^{6}x_5^{10} + x_1^{23}x_2^{15}x_3^{10}x_4^{3}x_5^{10}\big)\\ &\quad  +  Sq^4\big(x_1^{7}x_2^{23}x_3^{13}x_4^{6}x_5^{10} + x_1^{7}x_2^{15}x_3^{13}x_4^{6}x_5^{18} + x_1^{7}x_2^{15}x_3^{21}x_4^{6}x_5^{10} + x_1^{21}x_2^{15}x_3^{7}x_4^{6}x_5^{10}\big)\\ &\quad  + Sq^8\big(x_1^{11}x_2^{15}x_3^{13}x_4^{6}x_5^{10}\big) \ \mbox{mod}(P_5^-(3,4,3,3,1))
\end{align*}
Hence, the monomial $x$ is strictly inadmissible.	
\end{proof}
Based on Lemmas \ref{bda71} and \ref{bdbs2}, we have
\begin{props} We denote 
\begin{align*}p &= x_1x_2x_3^{3}x_4^{30}x_5^{28} + x_1x_2x_3^{7}x_4^{26}x_5^{28} + x_1x_2^{3}x_3^{6}x_4^{25}x_5^{28} + x_1x_2^{3}x_3^{7}x_4^{24}x_5^{28}\\ &\quad+ x_1x_2^{6}x_3^{3}x_4^{25}x_5^{28} + x_1^{3}x_2x_3x_4^{30}x_5^{28} + x_1^{3}x_2x_3^{7}x_4^{24}x_5^{28} + x_1^{3}x_2^{3}x_3x_4^{28}x_5^{28}\\ &\quad + x_1^{3}x_2^{3}x_3^{4}x_4^{25}x_5^{28} + x_1^{3}x_2^{3}x_3^{5}x_4^{24}x_5^{28} + x_1^{3}x_2^{4}x_3^{3}x_4^{25}x_5^{28} + x_1^{7}x_2x_3x_4^{26}x_5^{28}\\ &\quad + x_1^{7}x_2x_3^{3}x_4^{24}x_5^{28} + x_1^{7}x_2^{3}x_3x_4^{24}x_5^{28} + x_1^{7}x_2^{3}x_3^{9}x_4^{20}x_5^{24} + x_1^{7}x_2^{3}x_3^{13}x_4^{14}x_5^{26}. 
\end{align*}

If the polynomial $p$ is hit, then $QP_5(3,4,3,3,1) = 0$ while $QP_5(3,4,3,3,1)$ is an $\mathbb F_2$-vector space of dimension $14$ with a basis consisting of all the classes represented by the following admissible monomials:

\medskip
\centerline{\begin{tabular}{llll}  
$x_1^{3}x_2^{7}x_3^{13}x_4^{14}x_5^{26}$& $x_1^{3}x_2^{7}x_3^{13}x_4^{26}x_5^{14}$& $x_1^{3}x_2^{7}x_3^{25}x_4^{14}x_5^{14}$& $x_1^{7}x_2^{3}x_3^{13}x_4^{14}x_5^{26} $\cr  $x_1^{7}x_2^{3}x_3^{13}x_4^{26}x_5^{14}$& $x_1^{7}x_2^{3}x_3^{25}x_4^{14}x_5^{14}$& $x_1^{7}x_2^{11}x_3^{5}x_4^{14}x_5^{26}$& $x_1^{7}x_2^{11}x_3^{5}x_4^{26}x_5^{14} $\cr  $x_1^{7}x_2^{11}x_3^{13}x_4^{6}x_5^{26}$& $x_1^{7}x_2^{11}x_3^{13}x_4^{14}x_5^{18}$& $x_1^{7}x_2^{11}x_3^{13}x_4^{18}x_5^{14}$& $x_1^{7}x_2^{11}x_3^{13}x_4^{22}x_5^{10} $\cr  $x_1^{7}x_2^{11}x_3^{21}x_4^{10}x_5^{14}$& $x_1^{7}x_2^{11}x_3^{21}x_4^{14}x_5^{10}$.& & \cr 
\end{tabular}}
\end{props}
Thus, if $p$ is hit, then $\dim (QP_5)_{63} = 1141$ while $\dim (QP_5)_{63} = 1155$.
{}

\end{document}